\documentclass[11pt]{amsart}
\usepackage{hyperref,amssymb}
\usepackage[english]{babel}
\usepackage{fancyhdr}
\usepackage{colortbl}
\usepackage{enumerate}
\usepackage{graphicx} 
\usepackage{float}
\usepackage{mathrsfs}
\usepackage{cite}
\usepackage{kpfonts}
\usepackage[left=2.3cm,right=2.3cm,top=2.75cm,bottom=2.65cm]{geometry}
\usepackage{tikz}
\usetikzlibrary{arrows,automata,shapes,calc,patterns}
\newtheorem{theorem}{Theorem}[section]
\newtheorem{prop}[theorem]{Proposition}
\newtheorem{lemma}[theorem]{Lemma}
\newtheorem{cor}[theorem]{Corollary}

\theoremstyle{definition}
\newtheorem{definition}{Definition}[section]

\theoremstyle{remark}
\newtheorem{remark}{Remark}[section]

\numberwithin{equation}{section}

\def \H {\mathcal{H}} % Otro operador
\def \N {\mathbb{N}}  % naturales
\def \R {\mathbb{R}}  % reales
\def\RN {\mathbb{R}^{N}} %R^{N}

\def \H {H_0^1(\Omega)}
\def \HLinfty {H_0^1(\Omega) \cap L^{\infty}(\Omega)}

\def \Linfty {L^{\infty}(\Omega)}

\def \gradu {\nabla u}
\def \gradv {\nabla v}

\renewcommand{\div}{\operatorname{div}}

\renewcommand{\epsilon} {\varepsilon}

\hypersetup{
colorlinks=true, breaklinks=true, bookmarks=true,bookmarksnumbered,
urlcolor=green!50!black, linkcolor=blue!85!black, citecolor=red!85!black, % Link colors
pdftitle={}, % PDF title
pdfauthor={\textcopyright}, % PDF Author
pdfsubject={}, % PDF Subject
pdfkeywords={}, % PDF Keywords
pdfcreator={pdfLaTeX}, % PDF Creator
pdfproducer={LaTeX with hyperref and ClassicThesis} % PDF producer
}

\begin{document}

\title[Existence and multiplicity for an elliptic problem with  sign-changing coefficients]{Existence and multiplicity for an elliptic problem with critical growth in the gradient and sign-changing coefficients}

%    Remove any unused author tags.

%    author one information
\author[Colette De Coster and Antonio J. Fern\'andez]{}
\address{}
\curraddr{}
\email{}
\thanks{}

%    author two information
%\author[A. J. Fern\'andez]{Antonio J. Fern\'andez}
%\address{Univ. Valenciennes, EA 4015 - LAMAV - FR CNRS 2956, F-59313 Valenciennes, France}
%\address{Laboratoire de Math\'ematiques (UMR 6623), Universit\'e de Bourgogne Franche-Comt\'e, 
%16 route de Gray, 25030 Besan\c con Cedex, France}
%\curraddr{}
%\email{}
%\thanks{}

%\email{colette.decoster@univ-valenciennes.fr}
%\email{antonio\_jesus.fernandez\_sanchez@univ-fcomte.fr}



\date{}

\dedicatory{}

\maketitle

\centerline{\scshape Colette De Coster}
\smallskip
{\footnotesize
 % please put the address of the second  and third author
 \centerline{Univ. Polytechnique Hauts-de-France, EA 4015 - LAMAV - FR CNRS 2956, F-59313 Valenciennes, France}
\vspace{0.1cm}
\centerline{\textit{E-Mail address} : \texttt{colette.decoster@uphf.fr}}
}

\bigskip

\centerline{\scshape Antonio J. Fern\'andez}
\smallskip
{\footnotesize
% please put the address of the first author
 \centerline{Univ. Polytechnique Hauts-de-France, EA 4015 - LAMAV - FR CNRS 2956, F-59313 Valenciennes, France}
 \smallbreak
   \centerline{Laboratoire de Math\'ematiques (UMR 6623), Universit\'e de Bourgogne Franche-Comt\'e,}
 \centerline{16 route de Gray, 25030 Besan\c con Cedex, France}
  \vspace{0.1cm}
\centerline{\textit{E-Mail address} : \texttt{antonio\_jesus.fernandez\_sanchez@univ-fcomte.fr}}
} % Do not forget to end the {\footnotesize by the sign }

\medbreak

\begin{center}\rule{1\textwidth}{0.1mm} \end{center}
\vspace{-0.25cm}
\begin{abstract}
Let $\Omega \subset \RN$, $N \geq 2$, be a smooth bounded domain. We consider the boundary value problem
\[ \label{Plambda-Abstract-ch3} \tag{$P_{\lambda}$} -\Delta u = c_{\lambda}(x) u + \mu |\gradu|^2 + h(x)\,, \quad u \in \HLinfty\,,\]
where $c_{\lambda}$ and $h$ belong to $L^q(\Omega)$ for some $q > N/2$, $\mu$ belongs to $\R \setminus \{0\}$ and we write $c_{\lambda}$ under the form $c_{\lambda}:= \lambda c_{+} - c_{-}$ with $c_{+} \gneqq 0$, $c_{-} \geq 0$, $c_{+} c_{-} \equiv 0$ and $\lambda \in \R$. Here $c_{\lambda}$ and $h$ are both allowed to change sign. As a first main result we give a necessary and sufficient condition which guarantees the existence of a unique solution to \eqref{Plambda-Abstract-ch3} when $\lambda \leq 0$. Then, assuming that $(P_0)$ has a solution, we prove existence and multiplicity results for $\lambda > 0$. Our proofs rely on a suitable change of variable of type $v = F(u)$ and the combination of variational methods with lower and upper solution techniques.
%\vspace{3cm}
% $\mu \in \R \setminus\{0\}$ and 
%$c_{\lambda}$, $h \in L^q(\Omega)$ for some $q > N/2$. We allow that  $c_{\lambda} $ changes sign. More precisely, we write $c_{\lambda}$ under the form 
%$c_{\lambda} := \lambda c_{+} - c_{-}$  with $c^{+} \gneqq 0$, $c_{-} \geq 0$. As a first result we give a necessary and sufficient condition for the existence of a solution in case $\lambda\leq 0$.  
%Then, we  prove existence and multiplicity results for $\lambda$ positive. To that end, we transform 
%$(P_{\lambda})$ in an equivalent semilinear problem having variational structure. 
%Variational methods and lower and upper solution techniques are then used to prove the existence of at 
%least two solutions.
\end{abstract}
\begin{center} \rule{1 \textwidth}{0.1mm} \end{center}

\maketitle

\section{Introduction and main results} \label{I-ch3}

In this paper we study the existence and multiplicity of solutions to boundary value problems of the form
\begin{equation} \tag{$P$} \label{P-ch3}
\left\{
\begin{aligned}
-\Delta u & = c(x) u + \mu(x)|\nabla u|^2+h(x), & \quad \textup{ in } \Omega,\\
u & = 0, & \quad \textup { on } \partial \Omega,
\end{aligned}
\right.
\end{equation}
where $\Omega \subset \RN$, $N \geq 2$, is a bounded domain with smooth boundary, $c$ and $h$ belong to $L^q(\Omega)$ for some $q > N/2$, $\mu$ belongs to $\Linfty$ and the solutions are searched in $\HLinfty$.

\medbreak
There exist several mathematical reasons that make the study of nonlinear elliptic PDEs with quadratic growth in the gradient interesting. For instance, J.L. Kazdan and R.J. Kramer observed in 1978 that second order PDEs with quadratic growth in the gradient are invariant under changes of variable of type $v = F(u)$. This took them to claim in \cite[page 619]{K_K_1978} that \textit{``In the long run, the class of semilinear equations should be less important than some more general class of equations that is invariant under changes of variables"}. From a pure mathematical point of view, it is worth noting that, in Riemannian geometry, this type of equations naturally appears in the study of gradient Ricci solitons, see for instance  \cite[Section 1]{M_M_R_2018}. We also mention that problem \eqref{P-ch3} with $c \equiv 0$ corresponds to the stationary case of the Kardar-Parisi-Zhang model of growing interfaces introduced in \cite{K_P_Z_1986}. 

\medbreak
The study of nonlinear elliptic PDEs with a gradient dependence up to the critical growth was essentially initiated by L. Boccardo, F. Murat and J.-P. Puel in the 80's \cite{B_M_P_1983, B_M_P_1988, B_M_P_1992}. This type of problems have generated since then a large literature. In addition to several works strictly related to the content of this paper (that we will detail next), several directions have been investigated. For instance, let us mention here some recent works concerning subcritical growth in the gradient \cite{G_M_P_2006, G_M_P_2014, P_2010}, supercritical growth in the gradient \cite{P_2014}, low regularity coefficients and regularizing effects \cite{A_B_2015, A_B_2018} and pointwise estimates via symmetrization methods \cite{H_R_2017}. 

\medbreak
Now, we focus precisely on the existing literature concerning existence and multiplicity of solutions to problem \eqref{P-ch3}.

\medbreak
In the case where $c(x) \leq \alpha_0 < 0$ a.e. in $\Omega$ 
for some $\alpha_0 < 0$, now referred to as the \textit{coercive case}, the existence of a solution to \eqref{P-ch3} 
is a particular case of the results of  \cite{B_M_P_1983, B_M_P_1988, B_M_P_1992} and its uniqueness follows from 
\cite{B_B_G_K_1999, B_M_1995}. The \textit{weakly coercive case} $c \equiv 0$ was first studied in 
\cite{F_M_2000} where, for $\|\mu h \|_{N/2}$ small enough, the authors proved the existence of a solution to 
\eqref{P-ch3}. For $\mu(x) \equiv \mu > 0$ constant and $h \gneqq 0$, these results were then improved in 
\cite{A_DA_P_2006}. Finally, in the recent work \cite{DC_F_2018} we completely characterized the existence of solutions to \eqref{P-ch3} in the \textit{weakly coercive case} $c \equiv 0$. The \textit{limit coercive case} where one only requires $c(x) \leq 0$ a.e. in $\Omega$ (i.e. allowing parts of the domain where $c \equiv 0$ and parts of it where $c < 0$) proved to be more complex to treat. In \cite{A_DC_J_T_2015}, the authors observed that the existence of a solution to \eqref{P-ch3} is not guaranteed and gave sufficient conditions to ensure such existence. In case $h$ does not have a sign, weaker sufficient conditions can be found in \cite{DC_F_2018}. The fact that the uniqueness also holds in the \textit{limit coercive case} $c \leq 0$ was proved in \cite{A_DC_J_T_2015}. We refer likewise to \cite{A_DC_J_T_2014} for more general uniqueness results in this framework. Finally, let us point out that, except for \cite{A_DA_P_2006}, all these results were obtained without requiring any sign conditions on $\mu$ and $h$.

\medbreak
If $c(x) \not \leq 0$ a.e. in $\Omega$, i.e. $c \gneqq 0$ or $c$ changes sign, problem \eqref{P-ch3} behaves 
very differently and becomes much more richer than for $c \leq 0$. The first paper which addressed this situation 
was \cite{J_S_2013}. Following \cite{S_2010}, which considered a particular case, 
the authors studied \eqref{P-ch3} with $c \gneqq 0$ and $\mu(x) \equiv \mu > 0$ constant. 
For $\|c\|_q$ and $\|\mu h\|_{N/2}$ small enough the existence of two solutions to \eqref{P-ch3} was obtained. 
This result has now been complemented and improved in several ways. The restriction $\mu$ constant was first 
removed in \cite{A_DC_J_T_2015}. In that paper the authors imposed on $c$ a dependence on a real parameter 
$\lambda$ and considered $\lambda c \gneqq 0$. For $\mu(x) \geq \mu_1 > 0$ a.e. in $\Omega$ and 
$h \gneqq 0$, they proved the existence of at least two solutions for $\lambda > 0$ small enough. 
In this direction we refer also to \cite{DC_J_2017} where, imposing stronger regularity on $c$ and $h$, the authors 
removed the condition $h \gneqq 0$. Under different sets of assumptions, the authors clarified the structure of the 
set of solutions to \eqref{P-ch3} for $\lambda c \gneqq 0$. Note that in \cite{DC_F_2018} the above results were 
extended to the more general $p$-Laplacian case at the expense of considering $\mu$ constant. 
Also, in the frame of viscosity solutions and fully nonlinear equations, similar conclusions have been obtained 
in \cite{SN_S_2018} under corresponding assumptions. All the above mentioned results require either $\mu$ to 
be constant or to be uniformly bounded from below by a positive constant (or similarly bounded from above by a 
negative constant). In \cite{Soup_2015}, assuming that $\lambda c,\, \mu$ and $h$ were non-negative, a first 
attempt to remove this restriction was presented. Under suitable assumptions on the support of the coefficient 
functions and for $N \leq 5$, the existence of at least two solutions for $\lambda > 0$ small enough was obtained. 
Finally, let us point out that the only papers dealing with $c$ which may change sign are 
\cite{DC_F_J_2018-A2, J_RQ_2016}. In \cite{J_RQ_2016}, the authors dealt with $\mu(x) \equiv \mu > 0$ constant 
and $h \gneqq 0$ and they proved the existence of two solutions to \eqref{P-ch3} for  $\|c^{+}\|_q$ and 
$\|\mu h\|_{N/2}$ small enough. The restrictions $\mu > 0$ constant and $h \gneqq 0$ were removed in 
\cite{DC_F_J_2018-A2} at the expense of considering a ``thick zero set'' condition on the support of $c$ and 
suitable assumptions on $\mu$. Let us stress that \cite{DC_F_J_2018-A2} is the unique paper dealing with the 
\textit{non-coercive case} $c \not \leq 0$  where $\mu$ may change sign. 

\medbreak
In this paper we pursue the study of \eqref{P-ch3} and consider several situations where $c$ and $h$ may change sign. At the expense of considering $\mu$ constant we remove the ``thick zero set'' condition on $c$ considered in \cite{DC_F_J_2018-A2}. Moreover, we extend in several directions the previously known results and clarify the structure of the set of solutions in the case  $c^{+} \not \equiv 0$.  In order to state our main results, let us introduce the following order notions.

\begin{definition} For $h_1$, $h_2\in L^1(\Omega)$ we write
\begin{itemize}
\item $h_1\leq h_2$ if $h_1(x)\leq h_2(x)$ for a.e. $x\in\Omega$,
\item $h_1\lneqq h_2$ if $h_1\leq h_2$ and 
$\mbox{meas}(\{x\in\Omega:
h_1(x)<h_2(x)\})>0$.
\end{itemize}

\noindent
For $u$, $v \in \mathcal{C}^{1}(\overline{\Omega})$ we write
\begin{itemize}
\item $u < v$ if, for all $x \in \Omega\,,$ $u(x) < v(x)$,
\item $u \ll v$ if $u < v$  and, for all $x \in \partial \Omega$, either $u(x) < v(x)$, or, $u(x) = v(x)$ and $\frac{\partial u}{\partial \nu}(x) > \frac{\partial v}{\partial \nu}(x)$, where $\nu$ denotes the exterior unit 
normal. 
\end{itemize}
\end{definition}

As a first main result, we completely characterize the\textit{ limit coercive case}. Let us consider the boundary value problem 
\begin{equation} \label{Pcoercive-ch3}
-\Delta u = -d(x)u + \mu |\gradu|^2 + h(x), \quad u \in \HLinfty,
\end{equation}
under the assumption
\begin{equation} \label{hypCoercive-ch3}
\left\{
\begin{aligned}
& \Omega \subset \RN,\ N \geq 2, \textup{ is a bounded domain with } \partial \Omega \textup{ of class } \mathcal{C}^{0,1},\\
& d \textup{ and } h \textup{ belong to } L^q(\Omega) \textup{ for some } q > N/2,\\
& \mu > 0 \textup{ and } d \geq 0,
\end{aligned}
\right.
\end{equation}
and define 
\begin{equation}\label{md-intro-ch3}
m_d:= \left\{
\begin{aligned}
& \inf_{u \in W_{d}} \int_{\Omega} \big( |\nabla u|^2 - \mu h(x) u^{2}\big)\, dx\,, \quad & \textup{if} \quad W_{d} \neq \emptyset\,,\\
& + \infty\,, & \quad \textup{if} \quad W_{d} = \emptyset\,,
\end{aligned}
\right.
\end{equation}
where
\[ W_d:= \{ w \in H_0^1(\Omega): d(x)w(x) = 0 \textup{ a.e. in } \Omega,\ \|w\| = 1\}.\]
\medbreak
\noindent We prove the following sharp result. 

\begin{theorem}  \label{necessary and sufficient-ch3} \it
Assume \eqref{hypCoercive-ch3}. Then \eqref{Pcoercive-ch3} has a solution if, and only if, $m_{d} > 0$.
\end{theorem}

\begin{remark} \label{remark 2-ch3} $ $ 
\begin{itemize}
\item[a)] By \cite[Theorem 1.1]{A_DC_J_T_2014} we know that the solution obtained is unique.
\item[b)] This result generalizes \cite[Proposition 3.1 and Remark 3.2]{A_DC_J_T_2015} and \cite[Theorem 1.3 with $p = 2$]{DC_F_2018}.
\item[c)] Since $h$ does not have a sign, there is no loss of generality in assuming $\mu > 0$. If $u$ is a solution to \eqref{Pcoercive-ch3} with $\mu < 0$ then $w = -u$ solves
\begin{equation*} 
-\Delta w = -d(x)w - \mu |\nabla w|^2 - h(x)\,, \quad w \in \HLinfty.
\end{equation*}
%\item[a)] The constant $\mu > 0$ can be replaced by a function $\mu \in \Linfty$ with 
%$\mu(x) \geq \mu_1 > 0$ a.e. in $\Omega$ and the result still holds true.
%\item[c)] Under the regularity assumptions imposed in \eqref{hypCoercive-II}, 
%Theorem \ref{necessary and sufficient-II} is sharp. It cannot be improved.
\end{itemize}
\end{remark}

As observed in \cite{DC_F_J_2018-A2}, the structure of the set of solutions to \eqref{P-ch3} depends on the size of $c^{+}$ but it is not affected by the size of $c^{-}$. To enlighten this,
%In order to clarify this, \textcolor{blue}{je n'aime pas l'expression. Je propose ``To enlighten this" ou encore ``To stress this"' }
we replace $c$ by a function $c_{\lambda} := \lambda c_{+} - c_{-}$ with $\lambda$ a real parameter. More precisely, we consider the boundary value problem 
\[ \label{Plambda-ch3} \tag{$P_{\lambda}$}-\Delta u = c_{\lambda}(x) u + \mu|\gradu|^2 + h(x)\,, \quad u \in \HLinfty\,,\]
under the assumption
\[ \label{A1-ch3} \tag{$A_1$} 
\left\{
\begin{aligned}
& \Omega \subset \RN,\, N \geq 2, \textup{ is a bounded domain with }  \partial \Omega \textup{ of class } \mathcal{C}^{0,1},\\
& c_{+},\ c_{-} \textup{ and } h \textup{ belong to } L^q(\Omega) \textup{ for some } q > N/2 \,,\\
& \mu > 0\,,\ c_{+} \gneqq 0\,,\ c_{-} \geq 0 \textup{ and } c_{+}(x)c_{-}(x) = 0 \textup{ a.e. } x \in \Omega\,.
\end{aligned}
\right.
\] 
\begin{remark}
As in \eqref{hypCoercive-ch3}, since $h$ has no sign, there is no loss of generality in assuming $\mu > 0$.
\end{remark}
\bigbreak
Before going further and due to its importance in the rest of the paper, let us stress that for $\lambda = 0$ the problem \eqref{Plambda-ch3} reduces to
\begin{equation} \tag{$P_0$} \label{P0-ch3}
-\Delta u = - c_-(x)  u + \mu |\gradu|^2 + h(x)\,, \quad u \in \HLinfty.
\end{equation}
As an immediate corollary of Theorem \ref{necessary and sufficient-ch3}, we have the following result.

\begin{cor}  
\label{characterization P0-ch3} 
Assume \eqref{A1-ch3}. Then \eqref{P0-ch3} has a solution if, and only if, $m_{c_{-}} > 0$.
\end{cor}

Having at hand this information about the \textit{limit coercive case}, we turn to the study of the \textit{non-coercive case} $\lambda > 0$. First, using mainly variational techniques, we prove the following theorem.

\begin{theorem} \label{th1-ch3} Assume \eqref{A1-ch3} and suppose that \eqref{P0-ch3} has a solution. Then, there exists $\Lambda > 0$ such that, for all $ 0 < \lambda < \Lambda$, \eqref{Plambda-ch3} has at least two solutions.
\end{theorem}

\begin{remark}
This result improves and generalizes the main results obtained in \cite{J_RQ_2016, J_S_2013}.
\end{remark}
 
Considering stronger regularity assumptions on the coefficient functions and combining lower and upper solution methods with variational techniques, we improve the conclusions of Theorem \ref{th1-ch3}. We derive a more precise information on the structure of the set of solutions to \eqref{Plambda-ch3} when $\lambda > 0$. Under the assumption 
\begin{equation} \tag{$A_2$} \label{A2-ch3}
\left\{
\begin{aligned}
& \Omega \subset \RN,\ N \geq 2, \textup{ is a bounded domain with } \partial \Omega \textup{ of class } \mathcal{C}^{1,1},\\
& c_{+}\,,\, c_{-}\,, \textup{ and } h \textup{ belong to } L^p(\Omega) \textup{ for some } p > N,\\
& \mu > 0, \ c_{+} \gneqq 0,\ c_{-} \geq 0 \textup{ and } c_{+}(x)c_{-}(x) = 0 \textup{ a.e. in } \Omega,
\end{aligned}
\right.
\end{equation}
we prove the following theorems.

\begin{theorem}  \label{th2-ch3} \it
Assume \eqref{A2-ch3} and suppose that \eqref{P0-ch3} has a solution $u_0$ with $c_{+}u_0 \gneqq 0$. Then, every solution  $u$  to \eqref{Plambda-ch3} with $\lambda > 0$  and $c_{+} u \geq 0$ satisfies $u \gg u_0$. Moreover, there exists $\overline{\lambda} \in \,]0,+\infty[$, such that: 
\begin{itemize}
\item[$\bullet$] for $\lambda \in\, ]0,\overline{\lambda}[\,,$ the problem \eqref{Plambda-ch3} has at least two solutions $u_{\lambda,1}$, $u_{\lambda,2} \in \mathcal{C}_0^1(\overline{\Omega})$ such that:
\begin{itemize}
\item[$\circ$] $u_{\lambda,2} \gg u_{\lambda,1} \gg u_0$;
\item[$\circ$] $\lambda_1 < \lambda_2$ implies $u_{\lambda_1,1} \ll u_{\lambda_2,1}$;
\end{itemize}
\item[$\bullet$] \eqref{Plambda-ch3} with $\lambda = \overline{\lambda}$ has exactly one solution $u_{\overline{\lambda}} \in \mathcal{C}_0^1(\overline{\Omega})$ such that $u_{\overline{\lambda}} \gg u_0$;
\item[$\bullet$] for $\lambda > \overline{\lambda}$ the problem \eqref{Plambda-ch3} has no solutions $u$ such that $c_{+} u \geq 0$.
\end{itemize}
\end{theorem}

\medbreak

\begin{theorem}  \label{th3-ch3} 
Assume \eqref{A2-ch3} and suppose that \eqref{P0-ch3} has a solution $u_0$ with $c_{+}u_0 \lneqq 0$. Then, for every $\lambda > 0$, the problem \eqref{Plambda-ch3} has at least two solutions $u_{\lambda,1}, u_{\lambda,2} \in \mathcal{C}_0^1(\overline{\Omega})$ such that:
\begin{itemize} 
\item[$\bullet$] $u_{\lambda,1} \ll u_{\lambda,2}$;
\item[$\bullet$] $u_{\lambda,1} \ll u_0$ and $c_+u_{\lambda,2}\not\leq 0$;
%$\max_{\overline{\Omega}} u_{\lambda,2} > 0$;
\item[$\bullet$] $\lambda_1 < \lambda_2$ implies $u_{\lambda_1,1} \gg u_{\lambda_2,1}$.
\end{itemize}
\end{theorem}

%{\color{green!45!black}
To state our next result, let us assume that \eqref{P0-ch3} has a solution $u_0$,
%and consider the slightly stronger assumption
%\begin{equation} \tag{$A_3$} \label{A3-ch3}
%\left\{
%\begin{aligned}
%& \Omega \subset \RN,\ N \geq 2, \textup{ is a bounded domain with } \partial \Omega \textup{ of class } 
%\mathcal{C}^{1,1},\\
%&  c_{+}\,,\, c_{-}\,, \textup{ belong to } \Linfty  \textup{ and } h \textup{ belongs to } L^p(\Omega) 
%\textup{ for some } p > N,\\
%& \mu > 0, \ c_{+} \gneqq 0,\ c_{-} \geq 0 \textup{ and } c_{+}(x)c_{-}(x) = 0 \textup{ a.e. in } \Omega.
%\end{aligned}
%\right.
%\end{equation}
%We 
define the linearized operator 
\begin{equation} \label{linearized-operator-ch3}
L_{u_0}(\varphi) : = -\Delta \varphi - 2\mu \langle \nabla u_0, \nabla \varphi \rangle + c_{-}(x)\varphi,
\end{equation}
and denote by $\gamma_1 > 0$ the principal eigenvalue of
\begin{equation} \label{linearized-eigenvalue-problem-ch3}
L_{u_0}(\varphi) = \gamma c_{+}(x) \varphi, \quad \varphi \in H_0^1(\Omega).
\end{equation}
We refer to the proof of Proposition \ref{prop-maximum-anti-maximum-ch3}
% \cite[Theorem 3.1]{F_H_dT_2004} (see also \cite[Proposition 2.5]{G_G_P_2007}) 
for its existence under the assumption \eqref{A2-ch3}.
% and characterization.
% under the assumption \eqref{A3-ch3}. 
Then, we have the following result.

\begin{theorem} \label{th-h-0-ch3}
Assume \eqref{A2-ch3} and suppose that \eqref{P0-ch3} has a solution $u_0$ with $c_{+} u_0 \equiv 0$.
% Denote by
 %$\gamma_1 > 0$ the principal eigenvalue of
%\eqref{linearized-eigenvalue-problem-ch3}.
Then:
\begin{itemize}
\item[$\bullet$] for $\lambda \in \, ]0, \gamma_1[$, the problem \eqref{Plambda-ch3} has at least two solutions $u_{\lambda,1} \equiv u_0 \ll u_{\lambda,2} \in \mathcal{C}_0^1(\overline{\Omega})$;
\item[$\bullet$] for $\lambda = \gamma_1$, the problem \eqref{Plambda-ch3} has exactly one solution $u_{\gamma_1} \equiv u_0$;
\item[$\bullet$] for $\lambda > \gamma_1$, the problem \eqref{Plambda-ch3} has at least two solutions $u_{\lambda,1} \equiv u_0 \gg u_{\lambda,2} \in \mathcal{C}_0^1(\overline{\Omega})$.
\end{itemize}
\end{theorem} 

\begin{remark} $ $ 
\begin{itemize}
\item[a)] Under the assumption \eqref{A2-ch3},  for all $\lambda \in \R$,  every solution to \eqref{Plambda-ch3} belongs to $\mathcal{C}_0^1(\overline{\Omega})$. This was proved in \cite[Theorem 2.2]{DC_J_2017}.
\item[b)] At the expense of considering $\mu > 0$ constant instead of $\mu \in \Linfty$ with $\mu(x) \geq \mu_1 > 0$ in $\Omega$, Theorems \ref{th2-ch3}, \ref{th3-ch3} and  \ref{th-h-0-ch3} extend the main existence results of \cite{DC_J_2017} to the case where $c$ may change sign. Moreover, unlike \cite{DC_J_2017}, we do not assume global sign conditions on $u_0$ (solution to \eqref{P0-ch3}). Hence, even in the case where $c_{-} \equiv 0$, i.e. $c$ has a sign, our hypotheses are weaker than the corresponding ones in \cite{DC_J_2017}.
\item[c)] Theorem \ref{th2-ch3} removes the ``thick zero set'' condition on the support of $c_{\lambda}$ considered in  \cite[Theorem 1.2]{DC_F_J_2018-A2} and gives somehow a more precise information. In turn, here $\mu$ is constant and we require stronger regularity on the coefficient functions $c_{\lambda}$ and $h^{+}$.
\end{itemize}
\end{remark}
\medbreak

Finally, we give sufficient conditions in terms of $h$ ensuring that the hypotheses of Theorem \ref{th2-ch3},   \ref{th3-ch3} or \ref{th-h-0-ch3} are satisfied.

\begin{cor} \label{cor-ch3} Under the assumption \eqref{A2-ch3}, it follows that:
\begin{itemize}
\item[$\bullet$] If $h \gneqq 0$ and \eqref{P0-ch3} has a solution, then the conclusions of Theorem \ref{th2-ch3} hold.
\item[$\bullet$] If $h \lneqq 0$, then the conclusions of Theorem \ref{th3-ch3} hold.
\item[$\bullet$] If $h \equiv 0$, then the conclusions of Theorem \ref{th-h-0-ch3} hold.
\end{itemize}
\end{cor}

\begin{remark} \mbox{}
\begin{itemize}
\item [a)]
In case $h \leq 0$ (i.e. $h \lneqq 0$ or $h \equiv 0$), by Theorem \ref{necessary and sufficient-ch3},  the problem \eqref{P0-ch3} has always a solution.
\item[b)]
%{\color{green!45!black}
Let us consider the boundary value problem
\[
\left\{
\begin{aligned}
-&u''=\lambda c_+(x) u+ |u'|^2+h(x), \quad x\in ]-2\pi,2\pi[, \\
&u(-2\pi)=0,\ \, u(2\pi)=0,
\end{aligned}
\right.
\]
%$$
%\begin{array}{cl}
%-u''=\lambda c_+(x) u+ |u'|^2+h(x), &x\in ]-2\pi,2\pi[
%\\
%\end{array}
%$$
with 
\[ h(x) = 
\left\{
\begin{aligned}
&\cos x - |\sin x|^2 ,& \mbox{ for } x\in ]-2\pi, 0[, \\
&0, \quad & \mbox{ for } x \in  [0,2\pi[,
\end{aligned}
\right.
\]
%$$
%\begin{array}{rcll}
%h(x)&=&0,& \mbox{ for } x\in \,]0,2\pi[
%\\
%&=&\cos x - |\sin x|^2 ,& \mbox{ for } x\in ]-2\pi, 0[,
%\end{array}
%$$
which changes sign, and 
\[ c_{+}(x) = 
\left\{
\begin{aligned}
&0 ,& \mbox{ for } x\in ]-2\pi, 0[, \\
&\cos x + 1, \quad & \mbox{ for } x \in  [0,2\pi[.
\end{aligned}
\right.
\]
%$$
%\begin{array}{rcll}
%c_+(x)&=&\cos x + 1,& \mbox{ for } x\in [0,2\pi]
%\\
%&=&0,& \mbox{ for } x\in [-2\pi, 0]
%\end{array}
%$$
The unique solution to \eqref{P0-ch3} is given by
\[ u_0(x) = 
\left\{
\begin{aligned}
&\cos x-1 ,& \mbox{ for } x\in [-2\pi, 0[, \\
&0, \quad & \mbox{ for } x \in  \,[0,2\pi],
\end{aligned}
\right.
\]
%$$
%\begin{array}{rcll}
%u_0(x)&=&0,& \mbox{ for } x\in [0,2\pi]
%\\
%&=&\cos x - 1,& \mbox{ for } x\in [-2\pi, 0]
%\end{array}
%$$
and satisfies  $u_0\lneqq 0$ and $c_+u_0\equiv 0$. This example first shows that $u_0\lneqq 0$ does not imply $c_+u_0\lneqq 0$. It also shows that we can have $c_+u_0\equiv 0$ without having $h\equiv 0$ and finally that we can enter in the framework of Theorem  \ref{th-h-0-ch3} without having a sign on $h$.
\end{itemize}
\end{remark}

%\textcolor{red}{} \medskip

%\textcolor{blue}{J'ai retravaill\'e ici l'anglais mais surtout sur le plan math je ne suis pas convaincu quand vous parlez des Th 1.4 et 1.5. J'aime beaucoup l'exemple en tout cas !}

\medbreak
%
%\noindent {\color{blue!75!white} 
%\rule[3mm]{1 \textwidth}{0.1mm}
%
%\vspace{-0.3cm}
%\begin{enumerate}
%\item[1)] \textbf{Je pense qu'on doit vraiment r\'emarquer ici qu'on a des solutions du probl\`eme avec 
%$c_{+}u_0 \equiv 0$ que  ne sont pas $0$. Si je me rappelle bien vous m'avez montr\'e un example. 
%Une bonne remarque avec l'example que vous m'avez montr\'e doit donner assez int\^eret \'a le r\'esultat.
%\item[2)] Si on arrive \`a avoir le maximum et anti-maximum, il faut seulement \'eliminre $(A_3)$ et utiliser 
%$(A_2)$ directement. Le reste ne change rien. 
%\item[3)] La partie en rouge que viens c'est \`a refaire.}
%\end{enumerate}
%\noindent \rule[3mm]{1 \textwidth}{0.1mm}
%\medbreak}
%
%{\color{red!70!black}

We provide now some ideas of the proofs of Theorems \ref{th1-ch3}, \ref{th2-ch3} , \ref{th3-ch3} and \ref{th-h-0-ch3}. First of all we notice that, as $\mu$ is assumed to be a constant, we can perform % a by now known as 
a  Cole-Hopf change of variable %(even if it was already used by Ricatti {\color{red}a verifier}) 
and reduce \eqref{Plambda-ch3} to a semilinear problem. Considering 
\begin{equation} \label{cvI-ch3}
 v = \frac{1}{\mu} \Big( e^{\mu u} - 1 \Big)\,,
\end{equation}
one can check that $u$ is a solution to \eqref{Plambda-ch3} if, and only if, $v > -1/\mu$ is a solution to
\begin{equation} \label{p3-ch3}
-\Delta v = c_{\lambda}(x) g(v) + (1+\mu v) h(x)\,, \quad v \in \H\,,
\end{equation}
where $g$ is given by
\[ g(s) = \frac{1}{\mu} (1+\mu s)\ln(1+\mu s)\,, \quad \textup{ for } s > -1/\mu\,.\]
Hence, we need to control from below the solutions $v$ to \eqref{p3-ch3}.
This is one of the main difficulties we have to face when dealing with \eqref{p3-ch3}. More precisely,  we need to verify that every solution $v$ to \eqref{p3-ch3} satisfies $v > -1/\mu$.  
To that end, we truncate problem \eqref{p3-ch3} using a lower solution $\underline{u}_{\lambda}$ to \eqref{Plambda-ch3}. More precisely, we define 
\[ \alpha_{\lambda} = \frac{1}{\mu} \Big(e^{\mu \underline{u}_{\lambda}} - 1 \Big)\]
and introduce the problem
\[ \tag{$Q_{\lambda}$} \label{Qlambda-intro-ch3}
-\Delta v = f_{\lambda}(x,v)\,,\quad v \in \H\,,\]
where
\begin{equation*}
f_{\lambda}(x,s) = \left\{ \begin{aligned}
& c_{\lambda}(x)g(s) + (1+\mu s) h(x)\,, \quad & \textup{ if } s \geq \alpha_{\lambda}(x)\,,\\
& c_{\lambda}(x)g(\alpha_{\lambda}(x)) + (1+\mu \alpha_{\lambda}(x)) h(x)\,, \quad &\textup{ if } s \leq \alpha_{\lambda}(x)\,.
\end{aligned}
\right.
\end{equation*}
We then show that the solutions to \eqref{Qlambda-intro-ch3} satisfy $v\geq \alpha_{\lambda} > -1/\mu$ and so, they give solutions 
to \eqref{Plambda-ch3}.\medbreak

Under the assumptions of Theorem \ref{th2-ch3}, we easily see that $u_0$ is a lower solution to \eqref{Plambda-ch3} for all $\lambda > 0$ and we shall use it as $\underline{u}_{\lambda}$ in the definition of $\alpha_{\lambda}$. %Under the assumptions of  Theorem \ref{th2-ch3}, we can see that $u_0$ is a lower solution to \eqref{Plambda-ch3} for $\lambda>0$ and we can use it as $\underline{u}_{\lambda}$. 
In the other cases, we do not have an obvious lower solution at hand. However, in Section \ref{III-ch3}, we manage to construct a lower solution $\underline{u}_{\lambda}$ to \eqref{Plambda-ch3} below every upper solution to this problem. Using this lower solution $\underline{u}_{\lambda}$ in the definition of $\alpha_{\lambda}$, we obtain a problem  \eqref{Qlambda-intro-ch3} which is completely equivalent to \eqref{Plambda-ch3}. Let us point out that the fact that $c_{\lambda}$ has no sign causes several difficulties in this construction. We refer to Proposition \ref{propLowerSol-ch3} for more details.

 %With the use of this lower solution $\underline{u}_{\lambda}$,  we obtain a problem  \eqref{Qlambda-intro-ch3} which has the advantage of being completely equivalent   to \eqref{Plambda-ch3}.

%\textcolor{red}{} 

\medbreak
The main advantage of problem \eqref{Qlambda-intro-ch3} which respect to \eqref{Plambda-ch3} is that it admits a variational formulation. We shall then look for solutions to \eqref{Qlambda-intro-ch3} as critical points of the associated functional $I_{\lambda}: \H \to \R$ defined as
\[ I_{\lambda}(v) = \frac{1}{2} \int_{\Omega} |\gradu|^2 dx - \int_{\Omega} F_{\lambda}(x,v) dx \,, \]
where $G(s) = \int_0^s g(t) dt,$
\begin{equation*} 
F_{\lambda}(x,s) = c_{\lambda}(x)\,G(s) + \frac{1}{2\mu} (1+\mu s)^2 h(x)\,, \quad \textup{ if } s \geq \alpha_{\lambda}(x)\,,
\end{equation*}
and
\begin{equation*} 
\begin{aligned}
F_{\lambda}(x,s) = \big[ c_{\lambda}(x) g(\alpha_{\lambda}(x)) & + (1+\mu\alpha_{\lambda}(x))h(x) \big](s-\alpha_{\lambda}(x)) \\ & + c_{\lambda}(x)\, G(\alpha_{\lambda}(x)) + \frac{1}{2\mu}(1+\mu \alpha_{\lambda}(x))^2 h(x)\,, \quad \textup{ if } s \leq \alpha_{\lambda}(x)\,.
\end{aligned}
\end{equation*}
\medbreak
When $\lambda $ is positive, this functional is unbounded from below.
% and presents a concave-convex type geometry. 
Then, in trying to obtain critical points, we have to overcome several difficulties. First, we shall notice that $g$ is only slightly superlinear at infinity. Hence, $I_{\lambda}$ does not satisfies an Ambrosetti-Rabinowitz type condition. Moreover, the coefficient functions $c_{\lambda}$ and $h$ have no sign. In this context, to prove that the Palais-Smale sequences are bounded requires a special care. 
%The role of the lower solution 
%that we will construct in Proposition \ref{propLowerSol-ch3} 
%is again crucial. 
See Section \ref{Section-Auxiliary-funcional-ch3} for more details.
\medbreak
%\textcolor{red}{
%The role of the lower solution 
%that we will construct in Proposition \ref{propLowerSol-ch3} 
%is again crucial. 
%See Section \ref{Section-Auxiliary-funcional-ch3}.}
%\medbreak

Having at hand the Palais-Smale condition for $I_{\lambda}$ with $\lambda > 0$, we shall look for critical points which are either local minimum or of mountain-pass type. In Theorem \ref{th1-ch3}, we work mainly with variational techniques as in \cite{J_S_2013, J_RQ_2016}. Nevertheless, since our hypotheses are weaker than the corresponding ones in \cite{J_S_2013, J_RQ_2016}, to prove that the mountain-pass geometry holds becomes more involved. 
\medbreak

In Theorems \ref{th2-ch3}, \ref{th3-ch3} and \ref{th-h-0-ch3} we combine lower and upper solution with variational techniques. In all three theorems, a first solution is obtained throughout the existence of well-ordered lower and upper solutions. This solution is further proved to be a local minimum. Then, we obtain a second solution by a mountain-pass type argument.

%%%%%%%%%%%%%%%%%%%%%%%%%%%%%%%%%%%%%%%%%%%%%%%%%%%%%%%%%%%%%%%%

\medbreak

The rest of the paper is organized as follows. In Section \ref{II-ch3} we recall some auxiliary results that will be useful in the rest of the paper. Section \ref{coercive-ch3} is concerned with the proof of Theorem \ref{necessary and sufficient-ch3}. In Section \ref{Section-Auxiliary-funcional-ch3} we prove the Palais-Smale condition for the functional associated to 
\eqref{Qlambda-intro-ch3}.
 Section \ref{III-ch3} is devoted to the construction of  the strict lower solution $\underline{u}_{\lambda}$ to \eqref{Plambda-ch3} below every upper solution to this problem.
 The proofs of Theorems \ref{th1-ch3}, \ref{th2-ch3} and \ref{th3-ch3} are done respectively in Sections \ref{section 6}, \ref{section 7} and \ref{V-ch3}. In Section \ref{section 7}  we prove the first part of Theorem \ref{th-h-0-ch3} and of Corollary \ref{cor-ch3}, the rest of their proofs being postponed to Section \ref{V-ch3}.
 Finally, in Appendix \ref{Appendix-ch3}, we prove a Hopf's boundary point Lemma with unbounded lower order terms. %This permits to prove Theorem \ref{SMP-ch3}.

\medbreak

\noindent \textbf{Acknowledgments.}
The authors thank   warmly L. Jeanjean for helpful discussions.
\medbreak

\noindent \textbf{Notation.}
\begin{enumerate}{\small 
\item[1)] In  $\mathbb R^N$, we use the notations $|x|=\sqrt{x_1^2+\ldots+x_N^2}$ and $B_R(y)=\{x\in \mathbb R^N : |x-y|<R\}$.
\item[2)] We denote $\mathbb R^+=\,]0,+\infty[$ and $\mathbb R^-=\,]-\infty,0[$.}
\item[3)] For $v \in L^1(\Omega)$ we define $v^+= \max(v,0)$ and $v^- = \max(-v,0)$.
\item[4)] For $a$, $b \in L^1(\Omega)$ we denote $\{a \leq b\} = \{x\in \Omega:\, a(x)\leq b(x)\}\,.$
\item[5)] The space $\H$ is equipped with the
norm $\|u\|:=\big( \int_\Omega |\nabla u|^2\,dx\big)^{1/2}$.
\item[6)] 
For $p\in [1,+\infty[$, the norm $(\int_{\Omega}|u|^pdx)^{1/p}$
in $L^p(\Omega)$ is denoted by $\|\cdot\|_p$. We denote by $p^{\prime}$
the  conjugate exponent of $p$ and by $2^*$ the Sobolev critical exponent i.e. 
$2^*=2N/(N-2)$ if $N \geq 3$ and $2^{*} = + \infty$ in case $N = 2$. The norm in $L^{\infty}(\Omega)$ is $\|u\|_{\infty}=\mbox{esssup}_{x\in \Omega}|u(x)|$.

\end{enumerate}
\medbreak

\section{Preliminaries} \label{II-ch3}

This section presents some definitions and known results which are going to play an important role throughout the work. Let us start with some results on  lower and upper solution. We consider the boundary value problem
\begin{equation} \label{eqP1-ch3}
-\Delta u + H(x,u,\nabla u) = \xi(x)\,, \quad u \in \HLinfty,
\end{equation}
where $\xi$ belongs to $L^{1}(\Omega)$ and $H: \Omega \times \R \times \RN \to \R$ is a Carath\'eodory
function 
(i.e. for every $(s,\xi)\in \R \times \RN$,
$f(\cdot,s,\xi)$ is measurable on $\Omega$ and 
for a.e. $x\in \Omega$, $f( x, \cdot,\cdot)$ is
continuous on
$\R \times \RN$) .
%(see \cite[Definition I-3.1]{DC_H_2006} for the definition Carath\'eodory function).

\begin{definition}  \label{lowerUpperSolution-ch3}
We say that $\alpha \in H^1(\Omega)\cap L^{\infty}(\Omega)$ is a \textit{lower solution to \eqref{eqP1-ch3}} if 
$\alpha^{+} \in \H$ and, for all $\varphi \in \HLinfty$ with $\varphi \geq 0$, if follows that
\[ 
\int_{\Omega} \nabla \alpha \nabla \varphi\,dx + \int_{\Omega} H(x,\alpha, \nabla \alpha) 
\varphi\, dx \leq \int_{\Omega} \xi(x) \varphi\,dx\,.
\]
Similarly, $\beta \in H^1(\Omega)\cap L^{\infty}(\Omega)$ is an \textit{upper solution to \eqref{eqP1-ch3}} if 
$\beta^{-} \in \H$ and, for all $\varphi \in \HLinfty$ with $\varphi \geq 0$, if follows that
\[ 
\int_{\Omega}  \nabla \beta \nabla \varphi\,dx + \int_{\Omega} H(x,\beta, \nabla \beta) 
\varphi\, dx \geq \int_{\Omega} \xi(x) \varphi\,dx\,.
\]
\end{definition} 

\begin{theorem} \rm \cite[Theorems 3.1 and 4.2]{B_M_P_1988} \label{BMP1988-ch3}
\it Assume the existence of a non-decreasing function $b: \R^{+} \rightarrow \R^{+}$ and a function 
$k \in L^1(\Omega)$ such that
\[ 
|H(x,s,\xi)| \leq b(|s|)[ k(x) + |\xi|^2 ], \quad \textup{a.e. } x \in \Omega, \,\,
\forall (s,\xi) \in \R \times \RN\,.\]
If there exist a lower solution $\alpha$ and an upper solution $\beta$ to \eqref{eqP1-ch3} with $\alpha \leq \beta$, 
then there exists a solution $u$ to \eqref{eqP1-ch3} with $\alpha \leq u \leq \beta$.
Moreover, there exists $u_{min}$ (resp. $u_{max}$) minimum (resp. maximum) solution to \eqref{eqP1-ch3} with 
$\alpha \leq u_{min}\leq u_{max} \leq \beta$ and such that,  every solution $u$ to \eqref{eqP1-ch3} with
$\alpha \leq u \leq \beta$ satisfies $u_{min}\leq u \leq u_{max}$.
\end{theorem}

\begin{definition}
A \textit{lower solution} $\alpha \in \mathcal{C}^1(\overline{\Omega})$ is said to be \textit{strict}
if every solution $u$ to \eqref{eqP1-ch3} with $u\geq \alpha$ satisfies $u\gg\alpha$. Similarly, an \textit{upper solution} $\beta\in \mathcal{C}^1(\overline{\Omega})$  is said to be \textit{strict}
 if every solution $u$ to \eqref{eqP1-ch3} such that $u\leq \beta$ satisfies $u\ll\beta$.
\end{definition}

\medbreak
Now, we consider the boundary value problem
\begin{equation} \label{p2-ch3}
-\Delta v = f(x,v)\,, \quad v \in H_0^1(\Omega),
\end{equation}
being $f: \Omega \times \R \rightarrow \R$ an $L^p$-Carath\'eodory function for some $p > N$ (i.e. $f$ is a  Carath\'eodory function and, for every $R>0$, there exists $h_R$ with $f(x,s)\leq h_R(x)$, for a.e. $x\in \Omega$ and all $s\in  [-R,R]$)
%see \cite[Definition I-3.1]{DC_H_2006} for the definition of $L^p$-Carath\'eodory function) 
such that
\[ |f(x,s)| \leq C |s|^{\frac{N+2}{N-2}} + d(x),\]
for some $C > 0$ and $d \in L^{\frac{2N}{N+2}}(\Omega)$. This problem can be handled variationally. 
We consider the associated functional $J: \H \rightarrow \R$ 
defined by
\[ 
J(v) :=  \frac{1}{2} \int_{\Omega} |\gradv|^2\,dx - \int_{\Omega}F(x,v)\,dx\,,
\quad \mbox{ where }
\quad 
\frac{\partial}{\partial s} F(x,s) = f(x,s),
\]
and we recall the following results.

\begin{prop}\rm \cite[Theorem 6]{DF_S_1984} \cite[Chapter 1, Theorem 2.4]{Struwe}  \label{minimumInterval-ch3}
\it Assume that $\alpha$ and $\beta$ are respectively a lower and an 
upper solution to \eqref{p2-ch3} with 
$\alpha \leq \beta$  and consider the set
\[ M:= \big\{ v \in \H: \alpha \leq v \leq \beta \big\}.\]
Then the infimum of $J$ on $M$ is achieved at some $v$, and such $v$ is a solution to \eqref{p2-ch3}.
\end{prop}

%In order to define the notion of strict lower and upper solutions, we need the following order notion.

%\begin{definition}
%Let $u$, $v \in \mathcal{C}^{1}(\overline{\Omega})$. We say that $u \ll v$ in $\Omega$ if: 
%for all $x \in \Omega$, 
%$u(x) < v(x)$ and, for all $x \in \partial \Omega$, either $u(x) < v(x)$, or, $u(x) = v(x)$ and 
%$\frac{\partial u}{\partial \nu}(x) > \frac{\partial v}{\partial \nu}(x)$, where $\nu$ denotes the exterior unit 
%normal. 
%\end{definition}

%\begin{definition}
%We say that $\alpha \in \mathcal{C}^1(\overline{\Omega})$ is a \textit{strict lower solution} of \eqref{p1} if for 
%every solution $u$ of \eqref{p1} such that $\alpha \leq u$ in $\Omega$ it follows that $\alpha \ll u$ in $\Omega$.

%Similarly, $\beta\in \mathcal{C}^1(\overline{\Omega})$ is said to be a \textit{strict upper solution} of 
%\eqref{p1} if every solution $u$ of \eqref{p1} such that $\beta \geq u$ in $\Omega$ satisfies $\beta \gg u$ in 
%$\Omega$.
%\end{definition}

\begin{cor} \label{localMinimizer-ch3}
Assume that $\alpha$ and $\beta$ are strict lower and upper solutions to \eqref{p2-ch3} belonging to 
$\mathcal{C}^{1}(\overline{\Omega})$ and satisfying $\alpha \ll \beta$ and let $M$ be defined as in Proposition \ref{minimumInterval-ch3}. 
Then the minimizer $v$ of $J$ on $M$ is  a local minimizer of the functional $J$ in the $\mathcal{C}_0^1$-topology. 
Furthermore, this minimizer is a solution to  \eqref{p2-ch3} with $\alpha \ll v \ll \beta$.
\end{cor}

\begin{proof}
The proof follows as in \cite[Corollary 2.6]{DC_F_2018} using Proposition \ref{minimumInterval-ch3} and the fact that, as  $f$ is an $L^p$-Carath\'eodory function for some $p > N$, the classical regularity results imply that $v \in \mathcal{C}^{1}(\overline{\Omega})$.
%First of all observe that  Proposition \ref{minimumInterval-ch3} implies the existence of  $v \in \H$ solution to
%\eqref{p2-ch3}, which minimizes $J$ on $ M= \{ v \in \H: \alpha \leq v \leq \beta\}$.
%Moreover, as $f$ is an $L^{p}$-Carath\'eodory function for some $p > N,$ the classical regularity results imply that $v \in \mathcal{C}^{1}(\overline{\Omega})$.  Since the lower and the upper solutions are strict, it follows that $\alpha \ll v \ll \beta$ 
%and so, there is a 
%$\mathcal{C}_0^1$-neighbourhood of $v$ in $M$. Hence, it follows that $v$ minimizes locally $J$ in the 
%$\mathcal{C}_0^1$-topology. 
\end{proof}

\begin{prop}\rm \cite[Theorem 8]{DF_S_1984} \cite[Theorem 1]{B_N_1993}  \label{c01vsWop-ch3}
\it Assume that there exist $h \in L^p(\Omega)$ for some $p > N$ and $\sigma \leq \frac{2^{\ast}(p-1)}{p}- 1$ such that
\[ |f(x,s)| \leq h(x) (1 + |s|^{\sigma}), \quad \textup{a.e. }x \in \Omega\,,\ \forall\ s \in \R\,,\]
and let $v \in \H$ be a local minimizer of 
$J$ for the $\mathcal{C}_0^1$-topology. Then $v \in \mathcal{C}_0^{1}(\overline{\Omega})$ and it is a local minimizer of $J$ in the $H_0^1$-topology.
\end{prop}

%\begin{proof}
%The proof follows the line of \cite{B_N_1993} and can be obtained arguing as in 
%\cite[Theorem IV-2.1]{DC_H_2006}.
%\end{proof}

\medbreak
\begin{remark} $ $
If $f$ is an $L^{\infty}$-Carath\'eodory function the result holds under the growth condition
\[ |f(x,s)| \leq C (1+|s|^{\sigma}), \quad \textup{a.e. }x \in \Omega\,,\ \forall\ s \in \R\,,\]
for some $C > 0$ and $\sigma \leq 2^{\ast}-1$. In that case, we are exactly in the framework of \cite{B_N_1993}.
\end{remark}

Now, let us recall some abstract results in order to find critical points of $J$ other than local minima.

\begin{definition} \label{PSc-ch3}
Let $ (X , \| \cdot \|) $  be a real Banach space with dual space $ (X^{\ast}, \| \cdot \|_{\ast}) $ and let 
$\Phi: X \rightarrow \R$ be a $\mathcal{C}^1$ functional. The functional $ \Phi $ satisfies the \textit{Palais-Smale 
condition at level $c \in \R$} if, for any \textit{Palais-Smale sequence} at level $c \in \R$, i.e. for any sequence 
$\{x_n\} \subset X$ with
\begin{equation*}
\begin{aligned}
&\Phi(x_n) \to c \quad \textup{ and } \quad \|\Phi'(x_n)\|_{\ast} \to 0\,,
\end{aligned}
\end{equation*}
there exists a subsequence $\{x_{n_k}\}$ strongly convergent in $X$.
\end{definition}

\begin{theorem}\rm \cite[Theorem 2.1]{A_R_1973} \label{mpTheorem-ch3}
\it Let $ (X , \| \cdot \|) $  be a real Banach space. Suppose that $\Phi: X \rightarrow \R$ is a $\mathcal{C}^1$ 
functional. Take two points $e_1$, $e_2 \in X$ and define
\[ \Gamma := \{ \varphi \in \mathcal{C}([0,1],X): \varphi(0) = e_1, \,\varphi(1) = e_2\}\,,\]
and
\[ c := \inf_{\varphi \in \Gamma} \max_{t \in [0,1]} \Phi(\varphi(t))\,.\]
Assume that $\Phi$ satisfies the Palais-Smale condition at level $c$ and that
\[ c > \max \{ \Phi(e_1),\Phi(e_2)\}\,.\]
Then, there is a critical point of $\Phi$ at level $c$, i.e. there exists 
$x_0 \in X$ such that $\Phi(x_0) = c$ and $\Phi'(x_0) = 0$.
\end{theorem}

\begin{theorem}\rm \cite[Corollary 1.6]{G_1993} \label{characterizacionMinimizer-ch3}
\it Let $ (X , \| \cdot \|) $  be a real Banach space and let $\Phi: X \rightarrow \R$ be a $\mathcal{C}^{1}$ functional. 
Suppose that $u_0 \in X$ is a local minimum, i.e. there exists $\epsilon > 0$ such that
\[ \Phi(u_0) \leq \Phi(u), \quad \textup{ for } \|u-u_0\| \leq \epsilon\,,\]
and assume that $\Phi$ satisfies the Palais-Smale condition at any level $d \in \R$. Then:
\begin{itemize}
\item[i)] either there exists $0 < \gamma < \epsilon$ such that 
$ \inf \{\Phi(u): \|u-u_0\| = \gamma \} > \Phi(u_0),$
\item[ii)] or, for each $0 < \gamma < \epsilon$, $\Phi$ has a local minimum at a point $u_{\gamma}$ with 
$\|u_{\gamma}-u_0\| = \gamma$ and $\Phi(u_{\gamma}) = \Phi(u_0)$.
\end{itemize}
\end{theorem}

%{\color{green!45!black}

Another key ingredient in the proofs of Theorems \ref{th2-ch3}, \ref{th3-ch3} and \ref{th-h-0-ch3} is the following result that can be seen as a combination of the Strong maximum principle and the Hopf's Lemma with unbounded lower order coefficients. This can be obtained as a Corollary of \cite[Theorem 4.1]{rosales_2019}. Nevertheless, for the benefit of the reader, we provide a self-contained %simplified 
proof adapted to our setting in Appendix \ref{Appendix-ch3}. Under the assumption
\begin{equation} \label{hypSMP-ch3}
\left\{
\begin{aligned}
& \Omega \subset \RN,\ N \geq 2, \textup{ is a bounded domain with } \partial \Omega \textup{ of class } \mathcal{C}^{1,1},\\
& a  \textup{ belongs to } L^p(\Omega) \textup{ and } B = (B^1, \ldots, B^N) \textup{ belongs to }  (L^{p}(\Omega))^N \textup{ for some } p > N, \\
& a \geq 0,
\end{aligned}
\right.
\end{equation}
we obtain the following result.

\begin{theorem}
 \label{SMP-ch3}  
Assume  \eqref{hypSMP-ch3} and let $u \in \mathcal{C}^1(\overline{\Omega})$ be an upper solution to 
\begin{equation*}
-\Delta u + \langle B(x), \gradu \rangle + a(x) u = 0, \quad u \in H_0^1(\Omega).
\end{equation*}
Then, either $u \equiv 0$ or $u \gg 0$.
\end{theorem}

\begin{remark} $ $
The case where $B \in (L^{\infty}(\Omega))^N$ and $a \in \Linfty$ is nowadays classical and can be found for instance in \cite[Theorem 3.27]{T_1987}.
\end{remark}

\medbreak

We also need the following maximum and anti-maximum principles for non-selfadjoint second order operators with unbounded lower order coefficients. 
%The maximum principle can be deduced as in \cite[Proposition 2]{H_K_1980} from the results of 
%\cite{G_G_P_2007} and the anti-maximum principle can be found in \cite[Theorem 5.1]{G_G_P_2007}. 
%We recall both results together in the following proposition.

\begin{prop} \label{prop-maximum-anti-maximum-ch3}
Assume \eqref{hypSMP-ch3} and let $m\in L^p(\Omega)$ with $m\gneqq 0$.
%Let $B \in (\Linfty)^N$, $a, m \in \Linfty$ with $a \geq 0$ and $m \gneqq 0$ and $h \in L^p(\Omega)$ 
%for some $p > N$ with $h \gneqq 0$. 
Then, the problem
\begin{equation} \label{eq-maxim-anti-maxim-1}
-\Delta u + \langle B(x), \gradu \rangle + a(x) u = \gamma m(x) u, \quad u \in H_0^1(\Omega),
\end{equation}
has a unique principal eigenvalue $\gamma_1 > 0$ with corresponding eigenfunction $\varphi_1\gg 0$. Also, let $h\in L^p(\Omega)$ with $h\gneqq 0$. For the non-homogeneous problem
\begin{equation} \label{eq-maxim-anti-maxim-2}
-\Delta u + \langle B(x), \gradu \rangle + a(x) u = \gamma m(x) u + h(x), \quad u \in H_0^1(\Omega),
\end{equation}
 it follows that:
\begin{itemize}
\item[$\bullet$] For all $\gamma \in\, ]-\infty, \gamma_1[$, the solution $w$ to \eqref{eq-maxim-anti-maxim-2} satisfies $w \gg 0$.
\item[$\bullet$] For $\gamma = \gamma_1$ the problem \eqref{eq-maxim-anti-maxim-2} has no solution.
\item[$\bullet$] There exists $\delta = \delta(h) > 0$ such that, for all $\gamma \in\,  ]\gamma_1, \gamma_1 + \delta[$, the solution $w$ to \eqref{eq-maxim-anti-maxim-2} satisfies $w \ll 0$.
\end{itemize}
\end{prop}

\begin{proof}
Let us define $\overline{m}=\max\{m,1\}$ and consider the eigenvalue problem
\begin{equation} \label{eq-EVb}
-\Delta u + \langle B(x), \gradu \rangle + a(x) u-\gamma m(x) u = \mu \overline{m}(x) u, \quad u \in H_0^1(\Omega).
\end{equation}
It is clear that $\gamma$ is an eigenvalue of \eqref{eq-maxim-anti-maxim-1} if and only if $\mu=0$ is an eigenvalue of \eqref{eq-EVb}.
Note also that $a -\gamma m+\gamma^+\overline{m}\geq 0$. Then, we consider the operator
$$
K_{\gamma}: {\mathcal C}^1_0(\overline\Omega) \to {\mathcal C}^1_0(\overline\Omega): u \mapsto v
$$
where $v$ is the unique solution to
$$
-\Delta v + \langle B(x), \gradv \rangle + (a(x) -\gamma m(x) + \gamma^+ \overline{m}(x)) v =  \overline{m}(x) u, \quad v \in H_0^1(\Omega).
$$
By the compact embedding from $W^{2,p}(\Omega)\to {\mathcal C}^1_0(\overline\Omega)$, the choice of $\overline{m}$ and Theorem \ref{SMP-ch3}, we observe that $K_{\gamma}$ is compact and strongly positive (i.e. $u\gneqq 0 \Rightarrow K_{\gamma} u \gg 0$). Hence, applying the Krein-Rutman Theorem (see for instance \cite[Theorem 7.C and Corollary 7.27]{Zeidler}), we prove the existence of a unique eigenvalue $\mu_1(\gamma)$ of \eqref{eq-EVb} with eigenfunction $\varphi\gg 0$ 
%which is $\mu_1(\gamma)=r(K_{\gamma})^{-1}-\gamma^+$ 
and, for $h\in L^p(\Omega)$ with $p>N$ and $h\gneqq 0$, that the non-homogeneous problem 
$$
-\Delta u + \langle B(x), \gradu \rangle + a(x) u-\gamma m(x) u = r \overline{m}(x) u+h, \quad u \in H_0^1(\Omega)
$$
has a unique solution $u\gg0$ if $r<\mu_1(\gamma)$ and no solution $u\geq 0$ if $r\geq \mu_1(\gamma)$.
\medbreak
Now, arguing as in \cite{L-G_1996}, we prove that the function $\mu_1(\gamma)$ is concave and, following  \cite{G_G_P_2007}, that $\mu_1(0)>0$  and $\lim_{\gamma\to +\infty} \mu_1(\gamma)=-\infty$. This proves the existence of a unique $\gamma_1 > 0$ such that $\mu_1(\gamma_1)=0$, i.e. a unique principal eigenvalue $\gamma_1 > 0$  to problem \eqref{eq-maxim-anti-maxim-1}.

\medbreak
Finally, having at hand the existence and uniqueness of the principal eigenvalue $\gamma_1>0$, the rest of the proof follows exactly as in \cite{G_G_P_2007}.
\end{proof}

%{\color{red!70!black}

\section{Solving the limit coercive case} \label{coercive-ch3}

This section is devoted to the proof of Theorem \ref{necessary and sufficient-ch3}. Let us first recall some of the notation introduced in Section \ref{I-ch3}. For a function $d \in L^q(\Omega)$ for some $q > N/2$ we recall that 
\[
\begin{aligned}
 W_d := &\{ w \in H_0^1(\Omega): d(x)w(x) = 0 \textup{ a.e. in } \Omega,\ \|w\|  = 1\} \end{aligned}
\]
and
\begin{equation}\label{md-ch3}
m_d:= \left\{
\begin{aligned}
& \inf_{u \in W_{d}} \int_{\Omega} \big( |\nabla u|^2 - \mu h(x) u^{2}\big)\, dx\,, \quad & \textup{if} \quad W_{d} \neq \emptyset,\\
& + \infty\,, & \quad \textup{if} \quad W_{d} = \emptyset.
\end{aligned}
\right.
\end{equation}
Let us emphasize that 
\[
 W_0 = \{ w \in H_0^1(\Omega):\, \|w\|  = 1\}
 % \quad \textup{ and } \quad W_0^2 := \{ w \in H_0^1(\Omega): \, \|w\|_2 = 1\},
\]
and immediately observe that $W_d \subseteq W_0$.

\begin{remark}  \label{remark-normalization-ch3}
Note that we could have chosen a different normalization in the definition of $W_d$. In fact, if we define
\[
\begin{aligned}
\widetilde W_d := &\{ w \in H_0^1(\Omega): d(x)w(x) = 0 \textup{ a.e. in } \Omega,\ \|w\|_2  = 1\} \end{aligned}
\]
and
\begin{equation}  \label{md-ch3}
\widetilde m_d:= \left\{
\begin{aligned}
& \inf_{u \in \widetilde W_{d}} \int_{\Omega} \big( |\nabla u|^2 - \mu h(x) u^{2}\big)\, dx\,, \quad & \textup{if} \quad \widetilde W_{d} \neq \emptyset,\\
& + \infty\,, & \quad \textup{if} \quad \widetilde W_{d} = \emptyset,
\end{aligned}
\right.
\end{equation}
we can prove that 
\[
m_d>0 \quad  \Leftrightarrow \quad \widetilde m_d>0. 
\]
\end{remark}

\medbreak
\begin{proof}[\textbf{Proof of Theorem \ref{necessary and sufficient-ch3}}]
By \cite[Theorem 1.1]{DC_F_2018} we know that $m_d > 0$ is a sufficient condition to ensure that \eqref{Pcoercive-ch3} has a solution. Hence, we just have to prove that the existence of a solution to \eqref{Pcoercive-ch3} implies that $m_d > 0$. If $W_d = \emptyset$, the result is obviously true. Hence, we just consider the case where $W_d \neq \emptyset$. In the case where $d \equiv 0$, the result follows from \cite[Proposition 7.1]{DC_F_2018}. Thus, we may assume that $d \gneqq 0$. 
\medbreak
Assume that \eqref{Pcoercive-ch3} has a solution $u \in \HLinfty$. Then, it follows that 
\begin{equation} \label{phi2-test-ch3}
\int_{\Omega} \left( \gradu \nabla (\phi^2) + d(x) u \phi^2 - \mu |\gradu|^2 \phi^2 - h(x)\phi^2 \right) dx = 0, \quad \forall\ \phi \in \mathcal{C}_0^{\infty}(\Omega).
\end{equation}
Now, by Young's inequality, observe that
\begin{equation} \label{young-ch3}
\int_{\Omega} \gradu \nabla (\phi^2) dx \leq \int_{\Omega} \big( \mu |\gradu|^2 \phi^2 + \frac{1}{\mu} |\nabla \phi|^2 \big)\, dx, \quad \forall\ \phi \in \mathcal{C}_0^{\infty}(\Omega).
\end{equation}
Hence, gathering  \eqref{phi2-test-ch3}-\eqref{young-ch3} and using the density of $\mathcal{C}_0^{\infty}(\Omega)$ in $H_0^1(\Omega)$, we have that
\begin{equation}
\label{eq 3.5}
\int_{\Omega} \big(\frac{1}{\mu} |\nabla \phi|^2 + d(x)u \phi^2  - h(x)\phi^2 \big)\, dx \geq 0, \quad \forall\ \phi \in H_0^1(\Omega).
\end{equation}
Next, since for any $\phi \in W_d$, %we have that
\begin{equation} \label{remark1-NS-ch3}
\int_{\Omega} d(x)u \phi^2 dx = 0,
\end{equation}
and $W_d \subseteq W_0$, we obtain by \eqref{eq 3.5} that
\begin{equation} \label{ineq-NS-ch3}
\begin{aligned}
\inf_{\phi \in W_d} \int_{\Omega} \big(\frac{1}{\mu} |\nabla \phi|^2 - h(x)\phi^2 \big)\, dx 
& 
= \inf_{\phi \in W_d} \int_{\Omega} 
\big( \frac{1}{\mu} |\nabla \phi|^2 + d(x)u\phi^2  - h(x)\phi^2 \big)\, dx \\
& 
\geq \inf_{\phi \in W_0} \int_{\Omega}
 \big( \frac{1}{\mu} |\nabla \phi|^2 + d(x)u\phi^2  - h(x)\phi^2 \big)\, dx \geq 0.
\end{aligned}
\end{equation}
%Hence, if we have that
%\[ \inf_{\phi \in W_d} \int_{\Omega} \left(\frac{1}{\mu} |\nabla \phi|^2 - h(x)\phi^2 \right) dx = 0,\]
%it follows that
%\[  \inf_{\phi \in W_0} \int_{\Omega} \left( \frac{1}{\mu} |\nabla \phi|^2 
%+ d(x)\phi^2 u - h(x)\phi^2 \right) dx = 0.\]
Assume by contradiction that
\begin{equation*}
m_d = \mu \inf_{\phi \in W_d} \int_{\Omega} \big(\frac{1}{\mu} |\nabla \phi|^2 - h(x)\phi^2 \big)\, dx = 0.
\end{equation*}
Then, by standard arguments there exists $\phi_0 \in W_d \subseteq W_0$ non-negative such that 
\[ \int_{\Omega} \big(\frac{1}{\mu} |\nabla \phi_0|^2 - h(x)\phi_0^2 \big) \, dx = 0.\]
Thus, by Remark \ref{remark-normalization-ch3}, \eqref{remark1-NS-ch3} and \eqref{ineq-NS-ch3}, we have that
%\[ \int_{\Omega} \left(\frac{1}{\mu} |\nabla \phi_0|^2 +  d(x)u\phi_0^2 - h(x)\phi_0^2 \right) dx 
%= \int_{\Omega} \left(\frac{1}{\mu} |\nabla \phi_0|^2 - h(x)\phi_0^2 \right) dx = 0.\] 
%This implies that
\[  
\begin{aligned}
\inf_{\phi \in \widetilde W_0}
 \int_{\Omega}  \big( \frac{1}{\mu} |\nabla \phi|^2 + d(x)u\phi^2  - h(x)\phi^2 \big) \,dx 
 & 
 = \inf_{\phi \in W_0} \int_{\Omega} \big( \frac{1}{\mu} |\nabla \phi|^2 + d(x)u\phi^2  - h(x)\phi^2 \big) \,dx 
 \\
& = \int_{\Omega} \big(\frac{1}{\mu} |\nabla \phi_0|^2 +  d(x)u\phi_0^2 - h(x)\phi_0^2 \big)\, dx = 0,
\end{aligned} \]
and so, that $\phi_0$ is an eigenfunction associated to the first eigenvalue (which is then equal to zero) of the eigenvalue problem
\[ -\div\left( \frac{\nabla \phi}{\mu} \right) + (d(x)u - h(x)) \phi = \lambda \phi, \quad \phi \in H_0^1(\Omega).\]
Applying then \cite[Theorem 8.20]{G_T_2001_S_Ed} %and arguing as in \cite[Proposition 3.2]{C_2001} 
we deduce that $\phi_0 > 0$ in $\Omega$. Since $d \gneqq 0$, this contradicts that $\phi_0 \in W_d$ and the result follows.
\end{proof}
%\bigbreak
%
%\begin{proof}[\textbf{Proof of Corollary \ref{characterization P0-ch3}}] It follows directly from
% Theorem \ref{necessary and sufficient-ch3} taking $d \equiv c_{-}$.
%\end{proof}

\section{An auxiliary variational problem and its Palais-Smale condition} \label{Section-Auxiliary-funcional-ch3}

As explained in the introduction, to
%solve the problem of the needed 
control from below the solutions to \eqref{p3-ch3}, we modify the problem using a lower solution to problem \eqref{Plambda-ch3}.
The aim of this section is to introduce and study this auxiliary problem.
% which, together with the construction of a suitable lower solution to \eqref{Plambda-ch3}, will allow us to find 
% multiple solutions to \eqref{Plambda-ch3} when $\lambda > 0$.
First of all, inspired by \cite[Section 5]{DC_F_2018}, let us define
\begin{equation}  \label{g-ch3}
g(s) = \left\{ 
\begin{aligned}
& \frac{1}{\mu}(1+\mu s) \ln(1+\mu s)\,, \quad & s > -1/\mu\,,\\
& 0\,, & s \leq -1/\mu\,,
\end{aligned}
\right.
\qquad \textup{ and }  \qquad G(s) = \int_0^s g(t)dt.
\end{equation}
In the following lemma we gather some already known properties of these functions.

\begin{lemma} \label{gProperties-ch3} $ $ 
\begin{itemize}
\item[i)] The function $g$ is continuous on $\R$, satisfies  $g > 0$ on $\R^{+}$ and there exists $D > 0$ 
with $-D \leq g \leq 0$ on $\R^{-}$. Moreover, $G \geq 0$ on $\R$.
\item[ii)] For any $\delta > 0$, there exists $\overline{c} = \overline{c}(\delta,\mu) > 0$ such that, 
 for any $s > \frac{1}{\mu}$, 
$g(s) \leq \overline{c}\, s^{1+\delta}$.
\item[iii)] $\lim_{s\to +\infty}g(s)/s =+ \infty$ and $\lim_{s\to+\infty}G(s)/s^2 = + \infty$.
\item[iv)] For any $s \in \R$, it follows that $g(s)-s \geq 0$.
%\item[iv)] There exists $R > 0$ such that the function $H$ satisfies $H(s) \leq \big(\frac{s}{t}\big)^{p-1}H(t)$, 
%for $R \leq s \leq t$.
%\item[v)] The function $H$ is bounded on $\R^{-}$.
\end{itemize}
\end{lemma}

\begin{proof}
See \cite[Lemma 5.1]{DC_F_2018} for i), ii) and iii). See \cite[Lemma 7]{J_S_2013} for iv).
\end{proof}

Now, let $\underline{u}_{\lambda} \in H^1(\Omega) \cap \Linfty$ be a lower solution to \eqref{Plambda-ch3}. We define the function 
\begin{equation} \label{truncFunction-ch3}
\alpha_{\lambda} = \frac{1}{\mu} \left( e^{\mu \underline{u}_{\lambda}} -1 \right) \in H^1(\Omega) \cap \Linfty\,,
\end{equation}
and, before going further, observe that $ \alpha_{\lambda} \geq -1/\mu + \epsilon$ for some $\epsilon > 0$. We consider then the auxiliary truncated problem
\[ \label{Qlambda-ch3} \tag{$Q_{\lambda,\, \alpha_{\lambda}}$} -\Delta v = f_{\lambda, \, \alpha_{\lambda}}(x,v)\,, \quad v \in H_0^1(\Omega)\,,\]
where
\begin{equation} \label{flambda-ch3}
f_{\lambda, \, \alpha_{\lambda}}(x,s) 
= 
\left\{ \begin{aligned}
& c_{\lambda}(x)g(s) + (1+\mu s) \, h(x)\,, \quad & \textup{ if } s \geq \alpha_{\lambda}(x)\,,\\
& c_{\lambda}(x)g(\alpha_{\lambda}(x)) + (1+\mu \alpha_{\lambda}(x)) \, h(x)\,, \quad &\textup{ if } s \leq \alpha_{\lambda}(x)\,,
\end{aligned}
\right.
\end{equation}
and $g$ is defined in \eqref{g-ch3}. Following \cite[Lemma 5.2]{DC_F_2018} one can obtain the next lemma.
%{\color{green!45!black}
\begin{lemma} \label{equivProblems-ch3}
Assume \eqref{A1-ch3}, let $\underline{u}_{\lambda} \in H^1(\Omega) \cap \Linfty$ be a lower solution to \eqref{Plambda-ch3} and define $\alpha_{\lambda}$ by \eqref{truncFunction-ch3}.
Then, it follows that:
\begin{itemize}
\item[i)] Every $v$ solution to \eqref{Qlambda-ch3} belongs to $\Linfty$ and satisfies $v \geq \alpha_{\lambda}$.
%\item[ii)] Every solution $v$ to \eqref{Qlambda-ch3} satisfies $v \geq \alpha_{\lambda}$.
\item[ii)] If $u$ is a solution to \eqref{Plambda-ch3} such that $u \geq \underline{u}_{\lambda}$ then $v = \frac{1}{\mu} \big( e^{\mu u} - 1 \big)$ is a solution to \eqref{Qlambda-ch3}.
\item[iii)] If $v$ is a solution to \eqref{Qlambda-ch3} then $u = \frac{1}{\mu} \ln (1+\mu v)$ is a solution to \eqref{Plambda-ch3}.
\end{itemize}
\end{lemma}

\begin{proof}
See \cite[Lemma 5.2]{DC_F_2018}.
\end{proof}

\begin{remark} \label{equivProblemsLowerUpper-ch3}
ii) and iii) also hold if the solutions are replaced by lower or upper solutions.
\end{remark}

One of the main advantages of problem \eqref{Qlambda-ch3} is that it admits a variational formulation. Its solutions in $\H$ can be obtained as critical points of the functional $I_{\lambda,\, \alpha_{\lambda}}: \H \to \R$ defined as
\begin{equation} \label{Ilambda-ch3}
I_{\lambda,\, \alpha_{\lambda}}(v) = \frac{1}{2}\int_{\Omega} |\gradv|^2 dx - \int_{\Omega} F_{\lambda,\, \alpha_{\lambda}}(x,v) dx\,,
\end{equation}
where 
\begin{equation} \label{FLambdaPos-ch3}
F_{\lambda,\, \alpha_{\lambda}}(x,s) = c_{\lambda}(x)\,G(s) + \frac{1}{2\mu} (1+\mu s)^2 h(x)\,, \quad \textup{ if } s \geq \alpha_{\lambda}(x)\,,
\end{equation}
and
\begin{equation} \label{FLambdaNeg-ch3}
\begin{aligned}
F_{\lambda,\, \alpha_{\lambda}}(x,s) = \big[ c_{\lambda}(x) g(\alpha_{\lambda}(x)) & + (1+\mu\alpha_{\lambda}(x))h(x) \big](s-\alpha_{\lambda}(x)) \\ & + c_{\lambda}(x)\, G(\alpha_{\lambda}(x)) + \frac{1}{2\mu}(1+\mu \alpha_{\lambda}(x))^2 h(x)\,, \quad \textup{ if } s \leq \alpha_{\lambda}(x)\,.
\end{aligned}
\end{equation}
Note that, under the assumption \eqref{A1-ch3}, since $g$ has subcritical growth (see Lemma \ref{gProperties-ch3}), it is standard to show that $I_{\lambda, \, \alpha_{\lambda}} \in \mathcal{C}^{1}(\H,\R)$.

\medbreak
Now, we are going to show that, for any $\lambda > 0$, the functional $I_{\lambda, \, \alpha_{\lambda}}$ satisfies the Palais-Smale condition (see Definition \ref{PSc-ch3}).
First we show that the Palais-Smale sequences are bounded. Actually, we prove a slightly more general result.  Our proof is inspired by \cite{J_RQ_2016}. However, since we do not assume  $h \gneqq 0$ and truncate the nonlinearity, it is significantly more involved.
%The role of the lower solution is crucial. 
Let us define
\begin{equation} \label{mclambda-ch3}
m_{c_{\lambda}}:= \left\{
\begin{aligned}
& \inf_{u \in W_{c_{\lambda}}} \int_{\Omega} \big( |\gradu|^2 - \mu h(x) u^{2}\big)\, dx\,, \quad & \textup{if} \quad  W_{c_{\lambda}} \neq \emptyset\,,\\
& + \infty\,, & \quad \textup{if} \quad W_{c_{\lambda}} = \emptyset\,,
\end{aligned}
\right.
\end{equation}
where
\[
W_{c_{\lambda}} = \left\{ w \in \H\,,\ c_{\lambda}(x)w(x) = 0 \textup{ a.e. } x \in \Omega\,,\ w \geq 0\,,\ \|w\| = 1  \right\}.
\]
\medbreak
\begin{remark}
Observe that $W_{c_{\lambda}} \subseteq W_{c_{-}}$ and so that $m_{c_{\lambda}} \geq m_{c_{-}}$.
\end{remark}
\medbreak

\begin{lemma} \label{boundPS-ch3}
Fixed $\lambda > 0$ arbitrary, assume \eqref{A1-ch3} and suppose that $m_{c_{\lambda}} > 0$. Then, 
for all $A\in \R$, 
every sequence $\{v_n\} \subset H_0^1(\Omega)$ with $I_{\lambda,\, \alpha_{\lambda}}(v_n)\leq A$ and 
$\|I_{\lambda,\, \alpha_{\lambda}}'(v_n)\|\to 0$ is bounded. In particular, for all $d \in \R$, the Palais-Smale sequences for $I_{\lambda,\alpha_{\lambda}}$ at level $d$ are bounded. 
%the Palais-Smale sequences for $I_{\lambda,\, \alpha_{\lambda}}$ at  level $d$ are bounded.
\end{lemma}

\begin{proof}
Let $\{v_n\} \subset H_0^1(\Omega)$ be such a sequence.
%a Palais-Smale sequence for $I_{\lambda,\, \alpha_{\lambda}}$ at level $d \in \R$. 
First we claim that the sequence $\{v_n^{-}\}$ is bounded.
Indeed, as $\|I_{\lambda,\, \alpha_{\lambda}}'(v_n)\|\to 0$
%$\{v_n\}$ is a Palais-Smale sequence, 
there exists a sequence $\epsilon_n \to 0$ such that
\begin{equation} \label{ineq1-vnminus-ch3}
 -\epsilon_n \|v_n^{-}\| \leq \langle I_{\lambda,\, \alpha_{\lambda}}'(v_n),v_n^{-}\rangle \leq \epsilon_n \|v_n^{-}\|.
\end{equation}
Also, since $f_{\lambda,\, \alpha_{\lambda}}(x,s)$ is bounded in $\Omega \times \R^{-}$, there exist $D_1$, $D_2 > 0$ such that
\begin{equation} \label{ineq2-vnminus-ch3}
\langle I_{\lambda, \alpha_{\lambda}}'(v_n),v_n^{-}\rangle \leq -\|v_n^{-}\|^2 + D_1 \|v_n^{-}\| + D_2.
\end{equation}
Gathering \eqref{ineq1-vnminus-ch3} and \eqref{ineq2-vnminus-ch3} we deduce that 
\[  0 \leq -\|v_n^{-}\|^2 + (D_1 + \epsilon_n) \|v_n^{-}\| + D_2,\]
and the claim follows. To prove that $\{v_n^{+}\}$ is also bounded we assume by contradiction that $\|v_n\| \to \infty$ and introduce the sequence $\{w_n\} \subset \H$ given by $w_n = \frac{v_n}{\|v_n\|}$. Observe that $\{w_n\}$ is bounded in $\H$. Hence, up to a subsequence, it follows that $w_n \rightharpoonup w$ weakly  in $\H$, $w_n \to w$ strongly in $L^r(\Omega)$ for $1 \leq r < 2^{\ast}$ and $w_n \to w$ a.e. in $\Omega$. We split the rest of the proof into several steps: 
\bigbreak

%%%%%%%%%%%%%%%%%%%%%%%%%%%%%%%%%%%%%%%%%%%%%%%%%%%%%%%%%%%%%%%%%%

\noindent \textbf{Step 1:} $w \equiv 0$.
\smallbreak
As $\|v_n^{-}\|$ is bounded and by assumption $\|v_n\| \to \infty$, clearly $w^{-} \equiv 0$. It then remains to prove that $w^{+} \equiv 0$. We first prove that $c_{\lambda}w \equiv 0$. Assume by contradiction that $c_{\lambda}w^{+} \not \equiv 0$ and observe that for every $\varphi \in \H$, we can write
\begin{equation} \label{w equiv 0 - eq1-ch3}
\begin{aligned}
\langle I_{\lambda, \, \alpha_{\lambda}}'(v_n),\varphi\rangle = \int_{\Omega} \nabla v_n \nabla \varphi dx & - \int_{\{v_n \geq \alpha_{\lambda}\}} \mu v_n h(x)\varphi dx  - \int_{\{v_n \geq \alpha_{\lambda}\}} c_{\lambda}(x) g(v_n) \varphi dx - \int_{\{v_n \geq \alpha_{\lambda}\}}  h(x) \varphi dx \\
& - \int_{\{v_n < \alpha_{\lambda}\}} f_{\lambda, \alpha_{\lambda}}(x,v_n) \varphi\, dx
\end{aligned}
\end{equation}
Hence, using that $f_{\lambda, \, \alpha_{\lambda}}(x,s)$ is bounded in 
$\Omega \times \,]-\infty, \|\alpha_{\lambda}^{+}\|_{\infty}]$ and the convergence of $w_n$, it follows that
\begin{equation*} 
\frac{\langle I_{\lambda, \, \alpha_{\lambda}}'(v_n),\varphi\rangle}{\|v_n\|} 
=
 \int_{\Omega} \nabla w \nabla \varphi\, dx 
 -
  \int_{\Omega} \mu w^+ h(x)\varphi\,dx 
  - 
  \int_{\{v_n \geq \alpha_{\lambda}\}} c_{\lambda}(x) g(v_n) \frac{\varphi}{\|v_n\|} dx + o(1), \ \ \forall\ \varphi \in \H\,.
\end{equation*}
Actually, using that $g$ is bounded on $]-\infty, \|\alpha_{\lambda}^{+}\|_{\infty}]$ and that $w^{-} \equiv 0$, we obtain that
\begin{equation*}
\int_{\Omega} c_{\lambda}(x) g(v_n) \frac{\varphi}{\|v_n\|} dx = \int_{\Omega} \big( \nabla w \nabla \varphi - \mu w h(x) \varphi \big) \, dx + o(1)\,, \quad \forall\ \varphi \in \H\,.
\end{equation*}
Equivalently, we deduce that, for every $\varphi \in \H$,
\begin{equation}  \label{PSContradiction1-ch3}
\int_{\Omega} c_{\lambda}(x) \frac{g(v_n)-v_n}{\|v_n\|} \varphi\, dx = \int_{\Omega} \Big( \nabla w \nabla \varphi - (c_{\lambda}(x) + \mu h(x)) w \varphi \Big) \, dx + o(1).
\end{equation}
Since  $c_{\lambda}w^{+} \not \equiv 0$, we may choose $\varphi \in \H$ and a measurable subset $\Omega_{\varphi} \subset \Omega$ such that
\[ |\Omega_{\varphi}| > 0\,, \quad c_{\lambda}w^{+}\varphi > 0 \textup{ in } \Omega_{\varphi} \subset \Omega \quad \textup{and} \quad c_{\lambda}w^{+}\varphi \equiv 0 \textup{ in } \Omega \setminus \Omega_{\varphi}\,.\]
As $g(s) - s \geq 0$ on $\R$ (see Lemma \ref{gProperties-ch3}), it follows that
\[ c_{\lambda}(x) \frac{g(v_n)-v_n}{\|v_n\|} \varphi \geq 0 \quad \textup{ a.e. in } \Omega_{\varphi}\,.\]
Moreover, observe that
\[ \liminf_{n \to \infty} c_{\lambda}(x) \frac{g(v_n)-v_n}{\|v_n\|} \varphi = 
\liminf_{n \to \infty} c_{\lambda}(x) w_n \frac{g(w_n \|v_n\|)-w_n\|v_n\|}{w_n\|v_n\|} \varphi = + \infty \quad \textup{a.e. in } \Omega_{\varphi}\,.\]
Hence, applying Fatou's Lemma we deduce that
\[ \liminf_{n \to \infty} \int_{\Omega_{\varphi}} c_{\lambda}(x) \frac{g(v_n)-v_n}{\|v_n\|} \varphi\, dx = + \infty\,,\]
which yields a contradiction with \eqref{PSContradiction1-ch3}. Thus, we conclude that $c_{\lambda} w \equiv 0$. Now, we take $\varphi = w$ in \eqref{w equiv 0 - eq1-ch3} and divide by $\|v_n\|$. Using that $c_{\lambda} w \equiv 0$ and that $\|I'_{\lambda,\alpha_{\lambda}}(v_n)\| \to 0$, we get that
% $\{v_n\}$ is a Palais-Smale sequence,  
\[ \int_{\Omega} \nabla w_n \nabla w dx - \int_{\Omega} \mu w_n w h(x) \,dx \to 0,\]
and so, since $w_n \rightharpoonup w$  weakly in $\H$ and $w_n \to w$  strongly in $L^r(\Omega)$, for $1 \leq r < 2^{\ast}$, that
\[ \int_{\Omega} \big( |\nabla w|^2 - \mu h(x) w^2 \big)\, dx = 0\,.\]
\medbreak
\noindent By this last identity and the facts that  $w \geq 0$ and $c_{\lambda} w \equiv 0$, the condition $m_{c_{\lambda}} > 0$ implies that $w \equiv 0$.
\bigbreak

%%%%%%%%%%%%%%%%%%%%%%%%%%%%%%%%%%%%%%%%%%%%%%%%%%%%%%%%%%%%%%%%%%

\noindent \textbf{Step 2:} $ \displaystyle \frac{1}{\|v_n\|^2} \int_{\{v_n > \alpha_{\lambda}\}} c_{\lambda}(x) g(v_n)v_n dx  \to 1$.
\medbreak
First of all, observe that 
\[
\begin{aligned}
 \frac{\langle I_{\lambda, \, \alpha_{\lambda} }'(v_n),v_n \rangle}{\|v_n\|^2}
  = 
  1 - \frac{1}{\|v_n\|^2} \Big[ \mu \int_{\{v_n > \alpha_{\lambda}\}} v_n^2 h(x) \, dx & 
  + \int_{\{v_n > \alpha_{\lambda}\}} c_{\lambda}(x) g(v_n)v_n \, dx 
  \\
& + \int_{\{v_n > \alpha_{\lambda}\}} v_n h(x) \, dx 
+\int_{\{v_n \leq \alpha_{\lambda}\}} f_{\lambda, \alpha_{\lambda}}(x,v_n) \,dx \Big] 
\to 0\,.
\end{aligned}
\]
Hence, using that $w \equiv 0$ and that $f_{\lambda, \, \alpha_{\lambda}}(x,s)$ is bounded in 
$\Omega \times \,]-\infty, \|\alpha_{\lambda}^{+}\|_{\infty}]$, we deduce that
\[ 1 - \frac{1}{\|v_n\|^2} \int_{\{v_n > \alpha_{\lambda}\}} c_{\lambda}(x) g(v_n)v_n\, dx \to 0\,.\]
%Moreover, observe that
%\[ 
%\begin{aligned}
%\frac{1}{\|v_n\|^2}  \int_{\{v_n > \alpha_{\lambda}\}} c_{\lambda}(x) g(v_n)v_n dx = 
%& \ \ \frac{1}{\|v_n\|^2} \int_{\Omega}  c_{\lambda}(x) g(v_n^{+})v_n^{+} dx 
%\\ 
%&+ \frac{1}{\|v_n\|^2} \int_{\{\alpha_{\lambda} \leq v_n \leq 0\}} c_{\lambda}(x) g(v_n)v_n dx  
%-  \frac{1}{\|v_n\|^2} \int_{\{0 \leq v_n \leq \alpha_{\lambda}\}} c_{\lambda}(x) g(v_n)v_n dx,
%\end{aligned}
%\]
%and
%\[ \frac{1}{\|v_n\|^2} \int_{\{\alpha_{\lambda} \leq v_n \leq 0\}} c_{\lambda}(x) g(v_n)v_n dx \to 0, 
%\qquad \frac{1}{\|v_n\|^2} \int_{\{0 \leq v_n \leq \alpha_{\lambda}\}} c_{\lambda}(x) g(v_n)v_n \to 0.\]
%Thus, we can conclude that
%\[ 1 - \int_{\Omega} c_{\lambda}(x) \frac{g(v_n^{+})}{v_n^{+}}(w_n^{+})^2 dx 
%= 1 - \frac{1}{\|v_n\|^2} \int_{\Omega} c_{\lambda}(x) g(v_n^{+})v_n^{+} dx \to 0\,.\]

\medbreak

%%%%%%%%%%%%%%%%%%%%%%%%%%%%%%%%%%%%%%%%%%%%%%%%%%%%%%%%%%%%%%%%%%

\noindent \textbf{Step 3:} $ \displaystyle \ln (\|v_n\|) \int_{\{v_n > \alpha_{\lambda}\}} c_{\lambda}(x) w_n^2 \, dx + \int_{\{v_n > \alpha_{\lambda}\}} c_{\lambda}(x) w_n^2 \ln\big(\mu w_n + \frac{1}{\|v_n\|} \big) \, dx \to 1$.
% $\displaystyle 1 - \int_{\Omega} c_{\lambda}(x) (w_n^{+})^2 \ln(1+\mu \|v_n\| w_n^+) dx \to 0$.
\medbreak
By the definition of $g$ (see \eqref{g-ch3}) and Step 2, we have
\begin{equation}
\label{eq Step 3 PS}
 \frac{1}{\|v_n\|^2} \int_{\{v_n > \alpha_{\lambda}\}} c_{\lambda}(x) \frac{\ln(1+\mu v_n)}{\mu}\, v_n\, dx +  \frac{1}{\|v_n\|^2} \int_{\{v_n > \alpha_{\lambda}\}} c_{\lambda}(x)  \ln(1+\mu v_n)\, v_n^2\, dx
   \to 1.
 \end{equation}
%
%\[ \frac{g(v_n)}{v_n^+} = \frac{1}{\mu v_n^+} \ln(1+\mu v_n^+) + \ln (1+ \mu v_n^+) 
%= \frac{\ln(1+\mu \|v_n\| w_n^+)}{\mu \|v_n\| w_n^+} + \ln(1+\mu \|v_n\| w_n^+) \,.\]
%Now, observe that
%\[ \left| c_{\lambda}(x) \frac{\ln(1+\mu \|v_n\| w_n^+)}{\mu \|v_n\|w_n^+} \right| \leq |c_{\lambda}(x)|, \]%\qquad \textup{and} \qquad \lim_{n \to \infty} \frac{1}{\mu} 
%\frac{\ln(1+\mu \|v_n\| w_n^+)}{\mu \|v_n\| w_n^+} c_{\lambda}(x) = 0\,.\]
Now, as there exists $D>0$ such that 
\[
0\leq \frac{\ln(1+\mu s)}{\mu}\, s\leq {s}^2 +D, \quad \forall\ s > \alpha_{\lambda}(x), \quad \textup{ a.e. } x \in \Omega,  
\]
we deduce that 
\[ 
\frac{1}{\|v_n\|^2} \int_{\{v_n > \alpha_{\lambda}\}} \Big| c_{\lambda}(x) \frac{\ln(1+\mu v_n)}{\mu}\, v_n\Big|\, dx
\leq 
 \int_{\{v_n > \alpha_{\lambda}\}} |c_{\lambda}(x)|\Big( w_n^2 +\frac{D}{\|v_n\|^2}\Big)\, dx
\]
and, since $w_n\to w \equiv 0$ strongly in $L^r(\Omega)$ for all $1 \leq r < 2^{\ast}$, we obtain that
\[ 
\frac{1}{\|v_n\|^2} \int_{\{v_n > \alpha_{\lambda}\}}  c_{\lambda}(x) \frac{\ln(1+\mu v_n)}{\mu}\, v_n\, dx
 \to 0\,.\]
The claim then follows from \eqref{eq Step 3 PS} and the property of the logarithm that implies  
\[
\ln(1+\mu v_n)=\ln(\|v_n\|)+ \ln \big(\frac{1}{\|v_n\|}+\mu w_n\big).
\]
%  
%Applying the previous step, we conclude that
%\begin{equation} \label{PSLimit6-ch3}
%\displaystyle \int_{\Omega} c_{\lambda}(x) (w_n^{+})^2 \ln(1+\mu \|v_n\| w_n^+) dx \to 1.
%\end{equation}
%Observe that
%\[ \ln(1+\mu \|v_n\| w_n^{+}) =  \ln\left( \|v_n\| \Big(\mu w_n^{+} + \frac{1}{\|v_n\|} \Big) \right) 
%=  \ln(\|v_n\|) + \ln \Big( \mu w_n^{+} + \frac{1}{\|v_n\|} \Big)\,. \]
%Thus, substituting in \eqref{PSLimit6-ch3}, we conclude the Step 3.

\medbreak

%%%%%%%%%%%%%%%%%%%%%%%%%%%%%%%%%%%%%%%%%%%%%%%%%%%%%%%%%%%%%%%%%%

\noindent \textbf{Step 4:} $ \displaystyle \limsup_{n\to\infty} \Big(\ln (\|v_n\|) \int_{\{v_n > \alpha_{\lambda}\}} c_{\lambda}(x) w_n^2 \, dx\Big) \leq 0$.
\smallbreak

First of all, defining
$H(s) = \frac{1}{2}g(s)s - G(s)$, observe that
\begin{equation*}
\begin{aligned}
A + \epsilon_n \|v_n\| + o(1) & \geq I_{\lambda, \, \alpha_{\lambda}}(v_n) - \frac{1}{2}\langle I_{\lambda, \, \alpha_{\lambda}}'(v_n),v_n \rangle \\
& =  \int_{\{v_n >\alpha_{\lambda}\}} c_{\lambda}(x)H(v_n)\, dx - \frac{1}{2\mu} \int_{\{v_n > \alpha_{\lambda}\}} (1+\mu v_n) h(x)\, dx\\ 
& \hspace{40mm} - \int_{\{v_n \leq \alpha_{\lambda}\}} \big[ F_{\lambda, \, \alpha_{\lambda}}(x,v_n) 
   - \frac{1}{2}f_{\lambda, \, \alpha_{\lambda}}(x,v_n)v_n\big]\, dx %\\
%& = \int_{\Omega} c_{\lambda}(x) H(v_n^{+})dx 
%+ \int_{\{\alpha_{\lambda} \leq v_n \leq 0\}} c_{\lambda}(x) H(v_n) %dx - \int_{\{0 \leq v_n 
%\leq \alpha_{\lambda}\}} c_{\lambda}(x) H(v_n) dx 
%\\
%& \hspace{39mm} - \frac{1}{2\mu} \int_{\{v_n > \alpha_{\lambda}\}} %(1+\mu v_n) h(x) dx \\ 
%& \hspace{39mm} - \int_{\{v_n \leq \alpha_{\lambda}\}} \big[ F_{\lambda, \alpha_{\lambda}}(x,v_n) 
%- \frac{1}{2}f_{\lambda, \alpha_{\lambda}}(x,v_n)v_n\big] dx  \\ 
%& \leq A + \epsilon_n \|v_n\| + o(1)\,,
\end{aligned}
\end{equation*}
or equivalently
\begin{equation} \label{PSStep7-ch3}
\begin{aligned}
\int_{\{v_n >\alpha_{\lambda}\}} c_{\lambda}(x)H(v_n) \, dx \leq A
&
+
 \frac{1}{2\mu} \int_{\{v_n > \alpha_{\lambda}\}} (1+\mu v_n) h(x) \,dx
\\
&+
\int_{\{v_n \leq \alpha_{\lambda}\}} \big[ F_{\lambda, \, \alpha_{\lambda}}(x,v_n) 
   - \frac{1}{2}f_{\lambda, \, \alpha_{\lambda}}(x,v_n)v_n\big]\, dx
%
%
%\int_{\Omega} c_{\lambda}(x) H(v_n^{+}) dx = d 
%&- \int_{\{\alpha_{\lambda} \leq v_n \leq 0\}} c_{\lambda}(x) H(v_n) dx  + \int_{\{0 \leq v_n 
%\leq \alpha_{\lambda}\}} c_{\lambda}(x) H(v_n) dx \\
%&+ \frac{1}{2\mu} \int_{\{v_n > \alpha_{\lambda}\}} (1+\mu v_n) h(x) dx 
%\\ 
%&  +\int_{\{v_n \leq \alpha_{\lambda}\}} \big[ F_{\lambda, \alpha_{\lambda}}(x,v_n) 
%- \frac{1}{2}f_{\lambda, \alpha_{\lambda}}(x,v_n)v_n\big]dx 
 + \epsilon_n \|v_n\|  +  o(1)\,.
\end{aligned}
\end{equation}
Now, using that for every $s >-\frac{1}{\mu}$,
\[ H(s) = \frac{s^2}{4} + \frac{s}{2\mu}\big(1-\ln(1+\mu s)\big)-\frac{1}{2\mu^2}\ln(1+\mu s)\,,\]
substituting in \eqref{PSStep7-ch3} and multiplying by $\frac{4\ln(\|v_n\|)}{\|v_n\|^2}$, we deduce that
%\begin{equation*}
%\begin{aligned}
%\frac{1}{4} \int_{\Omega} c_{\lambda}(x) (v_n^{+})^2 dx = d &  
%- \int_{\{\alpha_{\lambda} \leq v_n \leq 0\}} c_{\lambda}(x) H(v_n) dx  
%+ \int_{\{0 \leq v_n \leq \alpha_{\lambda}\}} c_{\lambda}(x) H(v_n) dx \\
%&+ \frac{1}{2\mu} \int_{\{v_n > \alpha_{\lambda}\}} (1+\mu v_n) h(x) dx\\
%&  +\int_{\{v_n \leq \alpha_{\lambda}\}} \big[ F_{\lambda, \alpha_{\lambda}}(x,v_n) 
%- \frac{1}{2}f_{\lambda, \alpha_{\lambda}}(x,v_n)v_n\big]dx 
%- \frac{1}{2\mu} \int_{\Omega} c_{\lambda}(x) v_n^{+} (1- \ln(1+\mu v_n^{+})) dx \\
%& + \frac{1}{2\mu^2} \int_{\Omega} c_{\lambda}(x) \ln(1+\mu v_n^{+}) dx+ \epsilon_n \|v_n\| + o(1)\,.
%\end{aligned}
%\end{equation*}
%As a consequence, we obtain that
\begin{equation*}
\begin{aligned}
\ln(\|v_n\|) \int_{\{v_n > \alpha_{\lambda}\}} c_{\lambda}(x) w_n^2 \,dx \leq  \frac{4 \ln (\|v_n\|)}{\|v_n\|^2} \Bigg[ A & 
+
 \frac{1}{2\mu} \int_{\{v_n > \alpha_{\lambda}\}} (1+\mu v_n) h(x) \,dx
\\
&+
\int_{\{v_n \leq \alpha_{\lambda}\}} \big[ F_{\lambda, \, \alpha_{\lambda}}(x,v_n) 
   - \frac{1}{2}f_{\lambda, \, \alpha_{\lambda}}(x,v_n)v_n\big]\, dx
   \\
& - \frac{1}{2\mu} \int_{\{v_n > \alpha_{\lambda}\}} c_{\lambda}(x) v_n \big(1- \ln(1+\mu v_n)\big)\, dx \\
& + \frac{1}{2\mu^2} \int_{\{v_n > \alpha_{\lambda}\}}  c_{\lambda}(x) \ln(1+\mu v_n) \, dx 
+ \epsilon_n \|v_n\| \Bigg] +  o(1)\,. 
\end{aligned}
\end{equation*}
We easily deduce that each term of the right hand side goes to zero
%. Thus, 
%we can conclude that 
 and the Claim follows.
%\[ \ln(\|v_n\|) \int_{\Omega}c_{\lambda}(x) (w_n^{+})^2 dx \to 0\,.\]
\medbreak

%%%%%%%%%%%%%%%%%%%%%%%%%%%%%%%%%%%%%%%%%%%%%%%%%%%%%%%%%%%%%%%%%%

\noindent \textbf{Step 5:} \textit{Conclusion.}
\smallbreak
Considering together Steps 3 and 4, we deduce that
\[
\liminf_{n\to\infty} \int_{\{v_n > \alpha_{\lambda}\}} c_{\lambda}(x) w_n^2 \ln\big(\mu w_n + \frac{1}{\|v_n\|} \big) \, dx 
\geq 1.
\]
which clearly contradicts the fact that $w \equiv 0$. Since we have a contradiction, we conclude that $\|v_n\|$ is bounded, as desired.
\end{proof}

%Having at hand the boundedness of the Palais-Smale sequences, it is classical to show that the 
%Palais-Smale condition holds.

\begin{prop} \label{strongPS-ch3}
Fixed $\lambda > 0$ arbitrary, assume \eqref{A1-ch3} and suppose that $m_{c_{\lambda}} > 0$. Then,
 for all $A\in \R$, 
every sequence $\{v_n\} \subset H_0^1(\Omega)$ with $I_{\lambda,\, \alpha_{\lambda}}(v_n)\leq A$ and 
$\|I_{\lambda,\, \alpha_{\lambda}}'(v_n)\|\to 0$ admits a strongly convergent subsequence. In particular,  
 $I_{\lambda, \alpha_{\lambda}}$ satisfies the Palais-Smale condition at any level $d \in \R$. 
\end{prop}

\begin{proof}
Thanks to Lemma \ref{boundPS-ch3}, we know that such sequences
%the Palais-Smale sequences for $I_{\lambda, \alpha_{\lambda}}$ at any level $d \in \R$ 
are bounded. The strong convergence follows in a standard way. See \cite[Lemma 11]{J_S_2013} or \cite[Lemma 5.2]{DC_F_2018} for two different approaches adapted to this setting.
\end{proof}

\section{A lower solution to \eqref{Plambda-ch3} below every upper solution}\label{III-ch3}

The aim of this section is to construct a lower solution to \eqref{Plambda-ch3} below every upper solution to the problem. The construction of this lower solution relies on the following a priori lower bound proved by the first author and L. Jeanjean in \cite{DC_J_2017}. Let us consider the boundary value problem
\begin{equation} \label{equationLB-ch3}
-\Delta u = d(x)u + \mu|\gradu|^2 + f(x)\,, \quad u \in \HLinfty\,,
\end{equation}
under the assumption
\begin{equation} \label{assumptionLB-ch3}
\left\{
\begin{aligned}
& \Omega \subset \RN\,, N \geq 2\,, \textup{ is a bounded domain with } \partial \Omega \textup{ of class } \mathcal{C}^{0,1}\,,\\
& d \textup{ and } f \textup{ belong to } L^q(\Omega) \textup{ for some } q > N/2\,, \\
& \mu > 0. % \in \Linfty \textup{ satisfies } 0 < \mu_1 \leq \mu(x) \leq \mu_2\,.
\end{aligned}
\right.
\end{equation}

\begin{lemma} \rm \cite[Lemma 3.1]{DC_J_2017} \label{lowerBound-ch3} \it
Assume \eqref{assumptionLB-ch3}. Then, there exists a constant $M > 0$ with $M := M(N,q,|\Omega|,$ $\mu,\|d^{+}\|_q,\|f^{-}\|_q) > 0$ such that, every $u \in H^{1}(\Omega) \cap \Linfty$ upper solution to \eqref{equationLB-ch3} satisfies 
\[ \min_{\overline{\Omega}} u > -M\,.\]
\end{lemma}

\begin{remark} 
The lower bound does not depend on $f^+$ and $d^{-}$.  \label{remuniflb-ch3}
\end{remark}

Having at hand this lower bound, we construct the desired lower solution to \eqref{Plambda-ch3}. The proof is inspired by \cite[Lemma 4.2]{DC_J_2017} and \cite[Proposition 4.2]{DC_F_2018}. Nevertheless, since $c_{\lambda} = \lambda c_{+} - c_{-}$ may change sign, several new ideas are needed.

\begin{prop} \label{propLowerSol-ch3}
%{\color{green!45!black}
Under the assumption \eqref{A1-ch3}, for any $\lambda \in \R$, 
there exists $\underline{u}_{\lambda} \in H^{1}(\Omega) \cap \Linfty$ lower solution to 
\eqref{Plambda-ch3} such that, for every $\beta$ upper solution to \eqref{Plambda-ch3}, 
it follows that $\min_{\overline\Omega}\big(\min\{0,\beta\}-\underline{u}_{\lambda}\big)>0$.
\end{prop}

\begin{proof}
We shall consider separately the cases $\lambda \leq 0$ and $\lambda > 0$. The case $\lambda \leq 0$ can be obtained exactly as in \cite[Proposition 4.2]{DC_F_2018}. We turn then to study the case $\lambda > 0$. Fixed $\lambda > 0$ arbitrary, let us denote by $M_{\lambda,1} > 0$ the constant given by Lemma \ref{lowerBound-ch3} applied to \eqref{Plambda-ch3} and let us introduce the auxiliary problem
\begin{equation} \label{aux1-ch3}
\left\{
\begin{aligned}
 -\Delta u & = \lambda c_{+}(x)u + \mu|\gradu|^2- h^{-}(x)-\lambda M_{\lambda,1}c_{+}(x)-1\,, \quad & \textup{ in } \Omega,\\
u & = 0\,, & \textup{ on } \partial \Omega.
\end{aligned}
\right.
\end{equation}
Thanks to Lemma \ref{lowerBound-ch3}, there exists $M_{\lambda,2}> 0$ such that, for every $\beta_1 \in H^{1}(\Omega) \cap \Linfty$ upper solution to \eqref{aux1-ch3}, we have $\beta_1 > - M_{\lambda,2}$. Now, 
%for $k > M_{\lambda,2}$, we 
introduce the problem
\begin{equation} \label{aux2-ch3}
\left\{
\begin{aligned}
 -\Delta u & = - \lambda M_{\lambda,2} c_{+}(x)- h^{-}(x)-\lambda M_{\lambda,1}c_{+}(x)-1\,, \quad & \textup{ in } \Omega,\\
u & = 0\,, & \textup{ on } \partial \Omega,
\end{aligned}
\right.
\end{equation}
and denote by $\alpha$ its solution. Since $- \lambda M_{\lambda,2} c_{+}(x)- h^{-}(x)-\lambda M_{\lambda,1} c_{+}(x)-1 < 0$, the weak maximum principle implies that $\alpha \leq 0$. Then, observe that, for every $\beta_1$ upper solution to \eqref{aux1-ch3}, we have that
\[ -\Delta \beta_1 \geq  \lambda c_{+}(x) \beta_1 + \mu |\nabla \beta_1|^2 - h^{-}(x) - \lambda M_{\lambda,1} c_{+}(x) - 1 \geq - \lambda M_{\lambda,2} c_{+}(x) - h^{-}(x) - \lambda M_{\lambda,1} c_{+}(x) - 1 = - \Delta \alpha\,, \quad \textup{in } \Omega\,. \]
Consequently, it follows that
\[ \left\{
\begin{aligned}
-\Delta \beta_1 & \geq - \Delta \alpha\,, \quad & \textup{ in } \Omega, \\
\beta_1 & \geq \alpha\,, & \textup{ on } \partial \Omega,
\end{aligned}
\right.
\]
and, by the comparison principle, that $\beta_1 \geq \alpha$. 
\medbreak
Now, we introduce the problem
\begin{equation} \label{aux3-ch3}
\left \{
\begin{aligned}
-\Delta u & = \lambda c_{+}(x)\widetilde{T}(u) + \mu |\gradu|^2 - h^{-}(x) - \lambda M_{\lambda,1} c_{+}(x) -1, 
& \textup{ in } \Omega, \\
u & = 0, & \textup{ on } \partial \Omega,
\end{aligned}
\right.
\end{equation}
where
\begin{equation*}
\widetilde{T}(s) = \left\{
\begin{aligned}
 -M_{\lambda,2}\,,  \qquad&\textup{ if } s \leq -M_{\lambda,2}\,,\\
 s\,,  \qquad&\textup{ if } s > -M_{\lambda,2}\,.
\end{aligned}
\right.
\end{equation*}
Observe that $\beta_1$ and $0$ are upper solutions to \eqref{aux3-ch3}. Recalling that the minimum of two upper solutions is an upper solution, it follows that $\overline{\beta} = \min\{0,\beta_1\}$ is an upper solution to \eqref{aux3-ch3}. As $\alpha$ is a lower solution to \eqref{aux3-ch3} with $\alpha \leq \overline{\beta}$, applying Theorem \ref{BMP1988-ch3}, we conclude the existence of $v_{\lambda}$ minimal solution to \eqref{aux3-ch3} with $\alpha \leq v_{\lambda} \leq \overline{\beta} = \min\{0,\beta_1\}$.
\medbreak
Now, observe that $v_{\lambda}$ is an upper solution to \eqref{aux1-ch3}. Hence, it follows that 
$v_{\lambda}  > -M_{\lambda,2}$ and so, that $v_{\lambda}$ is a solution to \eqref{aux1-ch3}.
\medbreak
Finally, let us introduce $\underline{u}_{\lambda} = v_{\lambda} - M_{\lambda,1}$ and observe that
\begin{equation*}
\begin{aligned}
-\Delta \underline{u}_{\lambda} & = \lambda c_{+}(x) v_{\lambda} + \mu |\nabla v_{\lambda}|^2 - h^{-}(x) - \lambda M_{\lambda,1}c_{+}(x) - 1 
\\ 
%\leq \lambda c_{+}(x) \underline{u}_{\lambda} + \mu|\nabla \underline{u}_{\lambda}|^2 + h(x) 
&  \leq (\lambda c_{+}(x) - c_{-}(x))\, \underline{u}_{\lambda} + \mu|\nabla \underline{u}_{\lambda}|^2 + h(x), \quad \textup{ in } \Omega.
\end{aligned}
\end{equation*}
Hence, we have that $\underline{u}_{\lambda}$ is a lower solution to \eqref{Plambda-ch3} with 
$\underline{u}_{\lambda} \leq - M_{\lambda,1}$. 
Thus, since every $\beta$ upper solution to \eqref{Plambda-ch3} satisfies $\min_{\overline \Omega}\beta> - M_{\lambda,1}$,  we have that $\underline{u}_{\lambda}$ is the desired lower solution. % to \eqref{Plambda-ch3}.
%a lower solution to \eqref{Plambda-ch3} with $\underline{u}_{\lambda} \leq \beta$ for every $\beta$ upper solution to \eqref{Plambda-ch3}.
\end{proof}

\begin{remark}
The constant $\mu > 0$ can be replaced by a function $\mu \in \Linfty$ with $\mu(x) \geq \mu_1 > 0$ a.e. in $\Omega$ and the result still holds true.
\end{remark}

\section{Proof of Theorem \ref{th1-ch3}}
\label{section 6}

The aim of this section is to prove Theorem \ref{th1-ch3}. For every $\lambda \in \R$, let us denote by $\underline{u}_{\lambda} \in H^1(\Omega) \cap \Linfty$ the lower solution to \eqref{Plambda-ch3} constructed in Proposition \ref{propLowerSol-ch3}. We choose this $\underline{u}_{\lambda}$ as lower solution in \eqref{truncFunction-ch3} and 
%the definition of \eqref{Qlambda-ch3} and 
we consider, for the sake of simplicity, the more compact notation $I_{\lambda}:= I_{\lambda,\, \alpha_{\lambda}}$ for the functional $I_{\lambda,\, \alpha_{\lambda}}$ defined in \eqref{Ilambda-ch3}. First of all, we are going to prove that $I_{\lambda}$ has a mountain-pass geometry for $\lambda > 0$ small enough. We begin proving some auxiliary estimates.

\begin{lemma}  \label{mountainPassGeometry0-ch3}
Fixed $\Lambda_1 > 0$ arbitrary. There exist $D_1>0$ and $D_2>0$ such that, for every $\lambda \in [0,\Lambda_1]$ and any $v \in H_0^1(\Omega)$, it follows that
\begin{equation} \label{mpg0-1-ch3b}
I_{\lambda}(-v^{-}) \geq \frac{1}{2} \|v^{-}\|^2  - D_1 \|v^-\|-D_2
\end{equation}
\end{lemma}

\begin{proof}
First of all, observe that for all $v \in H_0^1(\Omega)$ we can write
\begin{equation*}
\begin{aligned}
I_{\lambda}(-v^{-})  = \frac{1}{2}\|v^{-}\|^2
& - \int_{ \{0 \geq v > \alpha_{\lambda} \} } \left( (\lambda c_{+}(x) - c_{-}(x)) G(v) + \frac{1}{2\mu} (1+\mu v)^2 h(x) \right) \,dx \\
& - \int_{ \{ v \leq \alpha_{\lambda}\} } \Big[ (\lambda c_{+}(x) - c_{-}(x)) g(\alpha_{\lambda}) + (1+\mu \alpha_{\lambda}) h(x) \Big] (v - \alpha_{\lambda}) \,dx \\
& - \int_{ \{v \leq \alpha_{\lambda}\} } \left( (\lambda c_{+}(x) - c_{-}(x)) G(\alpha_{\lambda}) + \frac{1}{2\mu} (1+\mu \alpha_{\lambda})^2 h(x) \right)\, dx.
\end{aligned}
\end{equation*}
Hence, using that, for all $\lambda\in \mathbb R$, $\alpha_{\lambda} \in [-1/\mu, 0]$ and Lemma \ref{gProperties-ch3}, i), we have that, for all $\lambda\in [0,\Lambda_1]$,
\begin{equation*}
\begin{aligned}
\displaystyle
I_{\lambda}(-v^{-}) \geq \frac{1}{2}\|v^{-}\|^2 & - \int_{\Omega} \left( \lambda c_{+}(x) \max_{ [-1/\mu,0] } G + \frac{1}{2\mu}  h^+(x) \right) dx \\
& - \int_{ \{ v \leq \alpha_{\lambda}\} } \Big[ \lambda c_{+}(x) g(\alpha_{\lambda}) - (1+\mu \alpha_{\lambda}) h^-(x) \Big] (v - \alpha_{\lambda}) dx \\
\geq \frac{1}{2}\|v^{-}\|^2 & - \int_{\Omega} \left( \Lambda_1 c_{+}(x) \max_{ [-1/\mu,0] } G + \frac{1}{2\mu} h^+(x) \right) dx  - \int_{\Omega} \left( \Lambda_1 c_{+}(x) \max_{ [-1/\mu,0] } |g| + h^-(x) \right) v^{-}\, dx.
\end{aligned}
\end{equation*}
The estimate \eqref{mpg0-1-ch3b} follows immediately from the Sobolev inequality.
\end{proof}

\begin{lemma} \label{mountainPassGeometry1-ch3}
Assume \eqref{A1-ch3} and suppose that $m_{c_{-}}> 0$. Then, there exist constants $\Lambda > 0$ and $R > 0$ such that, if $0 \leq \lambda \leq \Lambda$, then $I_{\lambda}(v) \geq  I_{\lambda}(0)+\frac12$ for all $v \in \partial D$ with $D:= \{ v \in H_0^1(\Omega): \|v^{+}\| < R\}$.
\end{lemma}

\begin{proof}
Let us begin with some preliminary observations. First of all, we fix $\Lambda_1 > 0$ arbitrary  and, by Lemma \ref{mountainPassGeometry0-ch3}, we know that, for all $\lambda \in [0,\Lambda_1]$ and all $v \in H_0^1(\Omega)$, \eqref{mpg0-1-ch3b} holds. This
%\begin{equation} \label{mpg1-1-ch3}
%I_{\lambda}(-v^{-}) \geq \frac{1}{2} \|v^{-}\|^2  - D_1 \|v^-\|-D_2.
%\end{equation}
 implies the existence of $D_3$ (independent of $\lambda$) such that,  for all $\lambda \in [0,\Lambda_1]$ and all $v \in H_0^1(\Omega)$,
\begin{equation} \label{mpg1-2-ch3}
I_{\lambda}(-v^{-}) \geq - D_3. 
\end{equation}
%for some constant $D_1 > 0$ (independent of $\lambda$). 
Now, by the definition of $I_{\lambda}$ and Lemma \ref{gProperties-ch3}, observe that, for any $\lambda \geq 0$ and any $\delta > 0$, there exists $D_4 > 0$ (independent of $\lambda$) such that, for all $v \in H_0^1(\Omega)$,
\begin{equation} \label{mpg1-3-ch3}
I_{\lambda}(v^{+}) \geq I_0(v^{+}) - \lambda D_4 (1+\|v^{+}\|^{2+\delta}). % \quad \forall\ v \in H_0^1(\Omega).
\end{equation}
Also, since $m_{c_{-}} > 0$ by hypothesis, we know that $I_0$ is coercive (see \cite[Proposition 6.1]{DC_F_2018}) and so, that there exists $R > 0$ such that, for all $v \in H_0^1(\Omega)$ with $\|v^{+}\| = R$, 
\begin{equation} \label{mpg1-4-ch3}
I_0(v^{+}) \geq -\frac{1}{\mu}\int_{\Omega}h(x)dx + 1 + D_3. % \quad \forall\ v \in H_0^1(\Omega) \quad \textup{ with } \quad \|v^{+}\| = R_1.
\end{equation}
Gathering \eqref{mpg1-3-ch3} and \eqref{mpg1-4-ch3} we deduce the existence of $0 < \Lambda \leq \Lambda_1$ such that, for all $\lambda \in [0,\Lambda]$ and all $v \in H_0^1(\Omega)$ with $\|v^{+}\| = R$, the following inequality holds
\begin{equation} \label{mpg1-5-ch3}
I_{\lambda}(v^{+}) \geq -\frac{1}{\mu}\int_{\Omega}h(x)dx + \frac{1}{2} + D_3. % \quad \forall\ v \in H_0^1(\Omega) \quad \textup{ with } \quad \|v^{+}\| = R_1.
\end{equation}
Now, for the constants $\Lambda > 0$ and $R > 0$ previously given, we define $D:= \{ v \in H_0^1(\Omega): \|v^{+}\| < R\}$ and consider an arbitrary $\lambda \in [0,\Lambda]$. In order to finish the proof, we are going to show that 
\[ I_{\lambda}(v) \geq I_{\lambda}(0) + \frac{1}{2}, \quad \forall\ v \in \partial D.\]
Let $v \in \partial D$ fixed but arbitrary. By \eqref{mpg1-2-ch3}, \eqref{mpg1-5-ch3} and the fact that $I_{\lambda}(0) = -\frac{1}{2\mu} \int_{\Omega} h(x) dx$, we directly obtain that
\begin{equation*}
\begin{aligned}
I_{\lambda}(v) = I_{\lambda}(v^{+}) + I_{\lambda}(-v^{-}) - I_{\lambda}(0) \geq -\frac{1}{\mu} \int_{\Omega} h(x) dx + \frac{1}{2} + D_3 - D_3 + \frac{1}{2\mu} \int_{\Omega} h(x) dx = I_{\lambda}(0) + \frac{1}{2}.
\end{aligned}
\end{equation*}
\end{proof}

\begin{lemma} \label{mountainPassGeometry2-ch3} 
Assume \eqref{A1-ch3}. For any $\lambda > 0$, $M > 0$ and $R > 0$, there exists $w \in \H$ such that $\|w^{+}\| > R$ and $I_{\lambda}(w) \leq -M.$
\end{lemma}

\begin{proof}
Since $c_{+} \not \equiv 0$, we can choose $v \in \mathcal{C}_0^{\infty}(\Omega)$ such that $v \geq 0$, $c_{+} v \not \equiv 0$ and $c_{-}v \equiv 0$. Moreover, let us take $t \in \R^{+}$, $t \geq 1$. 
As $\alpha_{\lambda} \leq 0$, observe that
\[ \begin{aligned}
I_{\lambda}(tv) 
&
 \leq 
 \frac{1}{2} t^2 \int_{\Omega} 
\big( |\gradv|^2 - \mu h(x) v^2 \big) \, dx - \lambda t^2  \int_{\Omega} c_{+}(x) v^2 \frac{G(tv)}{t^2 v^2} dx + \frac{t}{\mu} \|1+\mu v\|_{\infty} \|h^{-}\|_1 
\\
& = t^2 \left[  \frac{1}{2} \int_{\Omega} \big( |\gradv|^2 - \mu h(x) v^2 \big) \,dx - \lambda \int_{\Omega} c_{+}(x) v^2 \frac{G(tv)}{t^2 v^2} dx + \frac{1}{t} \frac{1}{\mu} \|1+\mu v\|_{\infty} \|h^{-}\|_1 \right]\,.
\end{aligned}
\]
Now, since by Lemma \ref{gProperties-ch3}, we have 
\[ \lim_{t \to \infty} \lambda \int_{\Omega} c_{+}(x) v^2 \frac{G(tv)}{t^2 v^2} dx = +\infty\,,\]
we deduce that $\displaystyle \lim_{t \to \infty} I_{\lambda}(tv) = -\infty$ and the lemma follows.
\end{proof}

Gathering Lemmas \ref{mountainPassGeometry1-ch3} and \ref{mountainPassGeometry2-ch3} we deduce that, for $\lambda > 0$ small enough, $I_{\lambda}$ possess a mountain-pass geometry. Once this is proved, we first show the existence of a local minimum of $I_{\lambda}$ and then we prove Theorem \ref{th1-ch3}.

\begin{prop} \label{localMinimumVariational-ch3}
Assume \eqref{A1-ch3} and suppose that $m_{c_{-}} > 0$ and that $\lambda \geq 0$ is small enough in order to ensure that the conclusion of Lemma \ref{mountainPassGeometry1-ch3} holds. Then, $I_{\lambda}$ possesses a critical point $v$ with $I_{\lambda}(v) \leq I_{\lambda}(0)$, which is a local minimum.
\end{prop}

\begin{proof}
From Lemma \ref{mountainPassGeometry1-ch3}, we know that there exist $R > 0$ such that
\[ m:= \inf_{v \in D} I_{\lambda}(v) \leq I_{\lambda}(0) \qquad \textup{ and } \qquad 
I_{\lambda}(v) \geq I_{\lambda}(0)+ \frac{1}{2} \,\,\textup{ if }\,\,  v \in \partial D,\]
where $D:= \{ v \in H_0^1(\Omega): \|v^{+}\| < R\}$. Let $\{v_n\} \subset D$ be such that $I_{\lambda}(v_n) \to m$. By the definition of $D$ and \eqref{mpg0-1-ch3b}, we deduce that $\{v_n\}$ is bounded and so, up to a subsequence, it follows that $v_n \rightharpoonup v \in H_0^1(\Omega)$. By the weak lower semicontinuity of the norm and of the functional $I_{\lambda}$, we have
\[ \|v^{+}\| \leq \liminf_{n \to \infty} \|v_n^{+}\| \leq R \quad \textup{ and } \quad I_{\lambda}(v) \leq \liminf_{n \to \infty} I_{\lambda}(v_n) = m \leq I_{\lambda}(0)\,.\]
Finally, since, by Lemma \ref{mountainPassGeometry1-ch3}, we know that $I_{\lambda}(v) \geq  I_{\lambda}(0)+  \frac{1}{2}$ if $v \in \partial D$, we deduce that $v \in D$ is a local minimum of $I_{\lambda}$.
\end{proof}

\begin{proof}[\textbf{Proof of  Theorem \ref{th1-ch3}}]
By Corollary \ref{characterization P0-ch3}, we know that $m_{c_-}>0$.
%First of all, as we have already pointed out several times (Remark \ref{muPositive1}), we can assume $\mu > 0$ 
%and our results will hold for $\mu \in \R \setminus\{0\}$. 
Assume that 
$\lambda > 0$ is small enough in order to ensure that the conclusion of Lemma \ref{mountainPassGeometry1-ch3} holds. 
% and  \ref{geometrySmall2}. 
By Proposition \ref{localMinimumVariational-ch3} we have  a first critical 
point $v_1$, which is a local minimum of $I_{\lambda}$ and satisfies $I_{\lambda}(v_1) \leq I_{\lambda}(0)$. On the other hand, since the Palais-Smale condition holds, in view of Lemmas \ref{mountainPassGeometry1-ch3}.
 and  \ref{mountainPassGeometry2-ch3}, we can apply Theorem \ref{mpTheorem-ch3} and obtain a second critical point of 
$I_{\lambda}$ at the mountain-pass level $c\geq  I_{\lambda}(0)+\frac12$. This gives two different solutions to  $(Q_{\lambda,\,\alpha_{\lambda}})$. % in $\Wop$. 
Finally, applying Lemma \ref{equivProblems-ch3}, we obtain two different solutions to our original problem \eqref{Plambda-ch3}.%, as desired. 
\end{proof}

\section{Proof of Theorem \ref{th2-ch3} and first part of Theorem \ref{th-h-0-ch3}}
\label{section 7}

This section is devoted to the proofs of Theorem \ref{th2-ch3}, the first part of Theorem \ref{th-h-0-ch3} and the first part of Corollary \ref{cor-ch3}. Let us recall that, under the assumption \eqref{A2-ch3}, every solution to \eqref{Plambda-ch3} belongs to $\mathcal{C}_0^1(\overline{\Omega})$ (see \cite[Theorem 2.2]{DC_J_2017}). We first prove Theorem \ref{th2-ch3} and the first part of Corollary \ref{cor-ch3}.

\begin{lemma} \rm \cite[Lemma 2.2]{A_DC_J_T_2014} \label{compPrincipleP0-ch3} \it
Assume \eqref{A1-ch3}. If $u_1,\, u_2 \in H^1(\Omega) \cap W_{loc}^{1,N}(\Omega) \cap \mathcal{C}(\overline{\Omega})$ are respectively a lower and an upper solution to \eqref{P0-ch3},
 then $u_1 \leq u_2\,.$
\end{lemma}

%\begin{proof}
%It follows directly from \cite[Lemma 2.2]{A_DC_J_T_2014}.
%\end{proof}

%{\color{green!45!black}
\begin{lemma} \label{lemmaStrictLowerAndUpperPositive-ch3}
Assume \eqref{A2-ch3} and suppose that \eqref{P0-ch3} has a solution $u_0$ with $c_{+}u_0 \gneqq 0$. Then, it follows that:
\begin{itemize}
\item[i)] Every  upper solution $\beta \in \mathcal{C}_0^1(\overline{\Omega})$  to \eqref{Plambda-ch3} with $\lambda > 0$ and $c_{+} \beta \geq 0$ satisfies $\beta \gg u_0$.
\item[ii)] For all $\lambda>0$, $u_0$ is a strict lower solution to \eqref{Plambda-ch3}.
\item[iii)] There exists $\lambda_0 > 0$ such that, for all $\lambda \in (0,\lambda_0)$, \eqref{Plambda-ch3} has a solution $u$ with $c_{+} u \geq 0$.
\item[iv)] The problem \eqref{Plambda-ch3} has no solution $u$ with $c_{+}u \geq 0$ for $\lambda > 0$ large.
% \item[$\bullet$] There exists $\overline{\lambda} \in \, ]0, + \infty[$ such that, 
%for all $0 < \lambda < \overline{\lambda}$, \eqref{Plambda-ch3} has a strict upper solution $\beta$ 
%with $\beta\gg u_0$ and, for all $\lambda > \overline{\lambda}$, \eqref{Plambda-ch3} has no solutions 
%$u$ with $c_{+} u \geq 0$.
\end{itemize}
\end{lemma}

\begin{proof} 
%Let us split the proof into several steps: 
%\medbreak
%\noindent \textbf{Step 1:} \textit{There exists $\lambda >$}
%\medbreak
%\noindent \textbf{Step 1:} \textit{Every $u$ solution to %\eqref{Plambda-ch3} with $\lambda > 0$ and $c_{+}u \geq 0$ satisfies $u \gg u_0\,.$}
i) Let $\beta \in \mathcal{C}_0^1(\overline{\Omega})$  be an upper solution to \eqref{Plambda-ch3} with $\lambda> 0$ and $c_{+}\beta \geq 0$. 
%We easily observe that $\beta \in \mathcal{C}_0^1(\overline{\Omega})$ is an upper solution to \eqref{P0-ch3} 
It is clear that $\beta$ is an upper solution to \eqref{P0-ch3}. Hence, by Lemma \ref{compPrincipleP0-ch3}, we have  $\beta \geq u_0$. 
Now, let us prove that 
%Now, arguing as in Corollary \ref{strictLower-ch3}, we are going to show that 
$\beta \gg u_0$. To that end, we define $w = \beta-u_0$ and observe that
%\[ 
%-\Delta w = \lambda c_{+}(x) u_0 - c_{-}(x)w + \mu \langle \nabla u + \nabla u_0, \nabla w \rangle\,, \quad \textup{ in }  \Omega\,,\]
%or equivalently
\[ \left\{ \begin{aligned}
- \Delta w - \mu \langle \nabla \beta + \nabla u_0, \nabla w \rangle + c_{-}(x) w & \geq \lambda c_{+}(x) u_0\,, \quad & \textup{ in } \Omega\,, \\
w & = 0\,, & \textup{ on } \partial \Omega\,.
\end{aligned}
\right.
\]
Applying then Theorem \ref{SMP-ch3}, we deduce that $w \gg 0$ and so that $\beta \gg u_0$.
\medbreak

%\noindent \textbf{Step 2:} \textit{For all $\lambda>0$, $u_0$ is a %strict lower solution to  \eqref{Plambda-ch3}.}}
%\medbreak
ii) As $c_{+}u_0 \gneqq 0$, we easily observe that $u_0$ is a lower solution to \eqref{Plambda-ch3}. Moreover, 
if $u$ is a solution to  \eqref{Plambda-ch3} with $u \geq u_0$ then $c_+u\geq c_+u_0\gneqq 0$. By i), we have that $u\gg u_0$. This proves that  $u_0$ is a strict lower solution to \eqref{Plambda-ch3}.
\medbreak

iii) Define  $v_0:= \frac{1}{\mu} \big( e^{\mu u_0} - 1 \big) \in \mathcal{C}_0^1(\overline{\Omega})$ and introduce the more compact notation $J_{\lambda} := I_{\lambda,\,v_0}$. Observe that $v_0$ %and, since $u_0-\epsilon$ and $u_0 + \epsilon$ are respectively strict lower and upper solution to \eqref{P0-ch3}, observe that $v_0$ 
is the minimum of $J_{0}$. The proof then follows the lines of
Lemma \ref{mountainPassGeometry1-ch3} and Proposition \ref{localMinimumVariational-ch3}.
\medbreak

%After that, let us prove the existence of $r > 0$ such that 
%\begin{equation} \label{contr_existence_lambda0}
%\inf_{v \in \partial B(v_0,r)} J_{0}(v) > J_{0}(v_0).
%\end{equation}
%We assume by contradiction that \eqref{contr_existence_lambda0} does not hold. 
%Then, by Theorem \ref{characterizacionMinimizer-ch3}, for all $n\in \mathbb N$, there exists $w_n$ with 
%$\|w_n-v_0\|=\frac{1}{n}$, $\,J_{0}(w_n)=J_{0}(v_0)$ and $w_n$ solution to $(Q_{0,\,v_0})$. 
%Since the solution to $(P_0)$ is unique, by Lemma \ref{equivProblems-ch3}, we obtain a contradiction. 
%Hence, we deduce that \eqref{contr_existence_lambda0} holds. 
%Having at hand \eqref{contr_existence_lambda0}, it is now easy to prove the existence of $\lambda_0 > 0$ 
%such that
%\begin{equation}\label{contr_existence_lambda0-2}
%\inf_{v \in \partial B(v_0,r)} J_{\lambda}(v) > J_{\lambda}(v_0),
%\end{equation}
%for all $\lambda \in (0,\lambda_0)$. Once \eqref{contr_existence_lambda0-2} is proved, i) follows arguing as 
%in the proof of Proposition \ref{localMinimumVariational-ch3}.
%
%
%
%\noindent \textbf{Step 3:} \textit{The problem \eqref{Plambda-ch3} has no solution $u$ with $c_{+}u \geq 0$ for $\lambda > 0$ large.}
%\medbreak
iv) We argue by contradiction. Suppose that $u$ is a solution to \eqref{Plambda-ch3} with $c_{+}u \geq 0$ and let $\gamma_1 > 0$ be the first eigenvalue and $\varphi_1 > 0$ be the first eigenfunction to the eigenvalue problem
\begin{equation*}
-\Delta v + c_{-}(x) v = \gamma c_{+}(x) v\,, \quad v \in \H\,.
\end{equation*}
Multiplying \eqref{Plambda-ch3} by $\varphi_1$ and integrating it follows that
\[ \begin{aligned}
\gamma_1 \int_{\Omega}c_{+}(x) u \varphi_1 dx & =  \int_{\Omega} \big( \gradu \nabla \varphi_1   + c_{-}(x) u  \varphi_1 \big)   dx \\
& = \lambda \int_{\Omega} c_{+}(x) u \varphi_1 dx + \mu \int_{\Omega} |\gradu|^2 \varphi_1 dx + \int_{\Omega} h(x) \varphi_1 dx\,,
\end{aligned}\]
or equivalently
\[ 0 = (\lambda - \gamma_1) \int_{\Omega} c_{+}(x) u \varphi_1 dx + \mu \int_{\Omega} |\gradu|^2 \varphi_1 dx + \int_{\Omega} h(x) \varphi_1 dx\,. \]
Hence, as by i) we know that  $u \geq u_0$, if $\lambda > \gamma_1$ it follows that
\begin{equation} \label{contStep2Lemma42-ch3}
0 \geq (\lambda - \gamma_1) \int_{\Omega} c_{+}(x) u_0 \varphi_1 dx +  \int_{\Omega} h(x) \varphi_1 dx\,.
\end{equation}
Since $c_{+} u_0 \gneqq 0$ and $\varphi_1 > 0\,,$ \eqref{contStep2Lemma42-ch3} gives a contradiction for $\lambda$ large enough. 
\end{proof}

\begin{proof}[\textbf{Proof of Theorem \ref{th2-ch3}}]%{\color{green!45!black}
Having at hand Lemma \ref{lemmaStrictLowerAndUpperPositive-ch3}, we define $v_0:= \frac{1}{\mu} \big( e^{\mu u_0} - 1 \big) \in \mathcal{C}_0^1(\overline{\Omega})$ and use the more compact notation $J_{\lambda}:= I_{\lambda, v_0}$. 
%\medbreak
%By Lemma  \ref{lemmaStrictLowerAndUpperPositive-ch3} we know that every $u$ solution to \eqref{Plambda-ch3} with $\lambda > 0$ and $c_{+} u \geq 0$ satisfies $u \gg u_0$.
Then, let us define 
%as in Lemma \ref{lemmaStrictLowerAndUpperPositive-ch3} and Proposition \ref{propLocalMinimumPositive-ch3}
\[  \overline{\lambda} := \sup \left\{ \lambda: \eqref{Plambda-ch3} \textup{ has a solution } u \textup{ with } c_{+}u \geq 0 \right\},\]
and observe that, by Lemma  \ref{lemmaStrictLowerAndUpperPositive-ch3}, we have $ 0 < \overline{\lambda}<\infty$ and, by the definition of  $ \overline{\lambda}$, for all $\lambda > \overline{\lambda}$, \eqref{Plambda-ch3} has no solution $u$ with $c_{+} u \geq 0$. Also by Lemma \ref{lemmaStrictLowerAndUpperPositive-ch3}, we know that every solution to \eqref{Plambda-ch3} with $\lambda > 0$ and $c_{+}u \geq 0$ satisfies $u \gg u_0$.
\medbreak

Now, we split the proof into several steps:
\medbreak
%\noindent \textbf{Step 1:} \textit{Every $u$ solution to \eqref{Plambda-ch3} with $c_{+}u \geq 0$ satisfies $u \gg u_0$.}
%\medbreak
%This follows directly from Lemma \ref{lemmaStrictLowerAndUpperPositive-ch3}.

%%%%%%%%%%%%%%%%%%%%%%%%%%%%%%%%%%%%%%%%%%%%%%%%%%%%%%%%%%%%%%%%%%

\noindent \textbf{Step 1:} \textit{For all $0 < \lambda < \overline{\lambda}$, \eqref{Plambda-ch3} has a strict upper solution $\beta_{\lambda} \gg u_0$.}
\medbreak
By the definition of $\overline{\lambda}$, 
we can find $\widetilde{\lambda} \in\, ]\lambda, \overline{\lambda}[$ and $u_{\widetilde{\lambda}}$ 
solution to $(P_{\widetilde{\lambda}})$ with $c_{+} u_{\widetilde{\lambda}} \geq 0$. Then, observe that 
$u_{\widetilde{\lambda}} \in \mathcal{C}_0^1(\overline{\Omega})$ is an upper solution to \eqref{P0-ch3} and so, 
by Lemma  \ref{lemmaStrictLowerAndUpperPositive-ch3}, $u_{\widetilde{\lambda}} \gg u_0$. Finally, in order to prove that $u_{\widetilde{\lambda}}$ is a  strict  upper solution to \eqref{Plambda-ch3}, let us consider a solution $u$  to 
\eqref{Plambda-ch3} with $u \leq u_{\widetilde{\lambda}}$ and introduce $w = u_{\widetilde{\lambda}} - u$. %Arguing as in Corollary \ref{strictLower-ch3}, % we are going to show that $w \gg 0$. 
%Using Lemma \ref{lemmaStrictLowerAndUpperPositive-ch3}, 
Observe that
\[ \left\{
\begin{aligned}
-\Delta w - \mu \langle \nabla u + \nabla u_{\widetilde{\lambda}}, \nabla w \rangle + c_{-}(x) w & \geq (\widetilde{\lambda} - \lambda) c_{+}(x) u_0\,, \quad & \textup{ in } \Omega\,, \\
w & = 0\,,& \textup{ on } \partial \Omega\,.
\end{aligned}
\right.
\]
Hence, applying Theorem \ref{SMP-ch3}, we obtain that $w \gg 0$, and so that $\beta_{\lambda}= u_{\widetilde{\lambda}}$ is the required  strict upper solution.

\medbreak
\noindent \textbf{Step 2:} \textit{For every $0 < \lambda < \overline{\lambda}$, there exists 
$v \in \mathcal{C}_0^1 (\overline{\Omega})$ with $v \gg v_0$ which is a local minimum of $J_{\lambda}$ 
in the $H_0^1$-topology and a solution to $(Q_{\lambda,\, v_0})$.}
\medbreak

First of all, by Lemma \ref{lemmaStrictLowerAndUpperPositive-ch3}, we easily observe that $v_0 \in \mathcal{C}^1(\overline{\Omega})$ is a strict lower solution to 
$(Q_{\lambda, v_0})$. On the other hand, by Step 1, 
for all $0 < \lambda < \overline{\lambda}$, there exists $\beta_{\lambda} \in \mathcal{C}_0^1(\overline{\Omega})$  
strict  upper solution to \eqref{Plambda-ch3} with $\beta_{\lambda} \gg u_0$. We introduce then
\[  w_{\lambda} := \frac{1}{\mu} \Big( e^{\mu \beta_{\lambda}} - 1 \Big)\,.\]
and, by  Lemma \ref{equivProblems-ch3} and Remark \ref{equivProblemsLowerUpper-ch3}, it follows that 
$v_0$, $w_{\lambda}\in  \mathcal{C}^1(\overline{\Omega})$ are a couple of well ordered strict lower and upper 
solutions to $(Q_{\lambda, v_0})$. Hence, applying Corollary \ref{localMinimizer-ch3} and Proposition 
\ref{c01vsWop-ch3} we have the existence of $v \in  \mathcal{C}_0^{1}(\overline{\Omega})$ minimizer 
of $J_{\lambda}$ on 
$M:= \big\{ v \in \H: v_0 \leq v \leq w_{\lambda}\big\}$, local minimum of $J_{\lambda}$ in the $H^1_0$-topology 
and solution to  $(Q_{\lambda,v_0})$. 
%Finally, by Lemma \ref{equivProblems-ch3}, ii), we have that $v \geq v_0$.

\medbreak

\noindent \textbf{Step 3:} \textit{For every $\lambda \in\,  ]0,\overline{\lambda}[$, \eqref{Plambda-ch3} has 
at least two solutions $u_{\lambda,1}$, $u_{\lambda,2} \in \mathcal{C}_0^1(\overline{\Omega})$ such that 
$u_0 \ll u_{\lambda,1} \ll u_{\lambda,2}$.}
\medbreak
By Step 2, we have the existence of  
%Proposition \ref{propLocalMinimumPositive-ch3}, for $\lambda < \overline{\lambda}$, there exists 
a first critical point $v_{\lambda,1} \in \mathcal{C}_0^{1}(\Omega)$, which is a local minimum of 
$J_{\lambda}$ and satisfies $v_{\lambda,1} \gg v_0$. Since the Palais-Smale condition at any level 
$d \in \R$ holds, by Theorem \ref{characterizacionMinimizer-ch3}, we have two options. 
If we are in the first case, then together with Lemma \ref{mountainPassGeometry2-ch3}, 
we see that $J_{\lambda}$ has the mountain-pass geometry and by Theorem \ref{mpTheorem-ch3}, 
we have the existence of a second solution to $(Q_{\lambda,\,v_0})$. In the second case, 
we have directly the existence of a second solution to  $(Q_{\lambda,\,v_0})$. 
Then, by Lemma \ref{equivProblems-ch3}, we conclude the existence of at least two solutions  
$u_{\lambda,1}$,  $u_{\lambda,2}$ to \eqref{Plambda-ch3}
with  $u_{\lambda,i} \geq u_0$ for $i=1,2$. 
%As $u_0$ is strict, we obtain that $u_{\lambda,i} \gg u_0$ for $i=1,2$.
\medbreak

Moreover, without loss of generality we can choose the first solution $u_{\lambda,1}$ as the minimal solution with $u_{\lambda,1} \geq u_0$. Hence, we have $u_{\lambda,1} \lneqq u_{\lambda,2}$ as otherwise there exists a solution with $u_0 \leq u \leq \min\{u_{\lambda,1},u_{\lambda,2}\}$ which contradicts the minimality of $u_{\lambda,1}$. As  $u_0$ is strict, we obtain that $u_{\lambda,1} \gg u_0$.
\medbreak

Now observe that, by convexity of $y \mapsto |y|^2$, the function $\beta = \frac{1}{2} (u_{\lambda,1} + u_{\lambda,2})$ is an upper solution to \eqref{Plambda-ch3} which is not a solution. In fact, arguing as in \cite[Proof of Theorem 1.3]{DC_J_2017} and Step 1, we can see that $\beta$ is a strict upper solution to \eqref{Plambda-ch3}. As $u_{\lambda,1} \lneqq \beta \lneqq u_{\lambda,2}$ and $\beta$ is strict, we deduce that $u_{\lambda,1} \ll u_{\lambda,2}$.

\medbreak

%%%%%%%%%%%%%%%%%%%%%%%%%%%%%%%%%%%%%%%%%%%%%%%%%%%%%%%%%%%%%%%%%%

\noindent \textbf{Step 4:} \textit{$\lambda_1 < \lambda_2$ implies $u_{\lambda_1,1} \ll u_{\lambda_2,1}$.}
\medbreak
Directly observe that $u_{\lambda_2,1}$ is an upper solution to $(P_{\lambda_1})$ which is not a solution. Hence, as $u_{\lambda_1,1}$ is the minimal solution to $(P_{\lambda_1})$ with $u_{\lambda_1,1} \geq u_0$, we deduce that $u_{\lambda_1,1} \leq u_{\lambda_2,1}$. 
%We define then $w = u_{\lambda_1,2} - u_{\lambda_1,1}$ and 
Arguing as in 
%Lemma \ref{lemmaStrictLowerAndUpperPositive-ch3} 
Step 1,  we conclude that  
$u_{\lambda_1,1} \ll u_{\lambda_2,1}$, as desired.
\medbreak

%%%%%%%%%%%%%%%%%%%%%%%%%%%%%%%%%%%%%%%%%%%%%%%%%%%%%%%%%%%%%%%

\noindent \textbf{Step 5:} \textit{Existence of solution to $(P_{\overline{\lambda}})$.}
\medbreak

Let $\{\lambda_n\}$ be a sequence with $0 <\lambda_n < \overline{\lambda}$ and $\lambda_n \to \overline{\lambda}$ and let $\{v_n\}$ be the corresponding sequence of minimum of $J_{\lambda_n}$ obtained in Step 2. This implies that $\langle J_{\lambda_n}'(v_n),\varphi \rangle = 0$ for all $\varphi \in \H$. Moreover, as $v_n$ is the minimum of $J_{\lambda_n}$ on $\{v\in H^1_0(\Omega) : v_0\leq v \leq w_{\lambda_n}\}$,
%
% by the construction of the lower and upper solutions, arguing as in Lemma \ref{mountainPassGeometry0-ch3}, 
we obtain that
\[ J_{\lambda_n}(v_n) \leq J_{\lambda_n} (v_0) \leq J_{\overline{\lambda}} (v_0) + D({\overline{\lambda}}-\lambda_n) \leq A, \]
for some $D>0$ and $A > 0$.
%$D_{\overline{\lambda}} > 0$. 
Then, by Lemma \ref{boundPS-ch3}
and Proposition \ref{strongPS-ch3}, we prove the existence of $v_{\overline{\lambda}} \in \H$ such that 
$v_n \rightarrow v_{\overline{\lambda}}$ in $\H$ with $v_{\overline{\lambda}}$ a solution to  
$(Q_{ \overline{\lambda},\,v_0})$.
% for $\lambda = \overline{\lambda}$. 
Moreover, as $v_n \geq v_0$ for all $n \in \N$, we deduce that $v_{\overline{\lambda}}\geq v_0$.
\medbreak

Applying Lemma \ref{equivProblems-ch3} we conclude that $u_{\overline{\lambda}} = \frac{1}{\mu} \ln(1+\mu v_{\overline{\lambda}})$ is a solution to  $(P_{\overline{\lambda}})$. Moreover, by construction and as $u_0$ is a strict lower solution, we have  $u_{\overline{\lambda}} \gg u_0$;
 
\medbreak

%%%%%%%%%%%%%%%%%%%%%%%%%%%%%%%%%%%%%%%%%%%%%%%%%%%%%%%%%%%%%%%%%%
%
%\noindent \textbf{Step 5:} \textit{For $\lambda > \overline{\lambda}$ the problem \eqref{Plambda-ch3} has no solutions $u$ with $c_{+}u \geq 0$. }
%\medbreak
%This follows directly from Lemma \ref{lemmaStrictLowerAndUpperPositive-ch3}.
%
%
%\medbreak

%%%%%%%%%%%%%%%%%%%%%%%%%%%%%%%%%%%%%%%%%%%%%%%%%%%%%%%%%%%%%%%

\noindent \textbf{Step 6:} \textit{Uniqueness of solution to $(P_{\overline{\lambda}})$.}
\medbreak

Assume by contradiction the existence of two solutions $u_1$, $u_2$ to $(P_{\overline{\lambda}})$ with $c_{+} u_1 \geq 0$ and $c_{+}u_2 \geq 0$. As in Step 3, we can assume that 
 $u_0 \ll u_1 \ll u_2$. By strict convexity of the nonlinearity and Theorem \ref{SMP-ch3}, it is easy to prove that, for all 
 $\lambda \in \,]0,1[$, $\beta_{\lambda}=\lambda u_1+(1-\lambda) u_2$ is a strict upper solution to $(P_{\overline{\lambda}})$. This implies in particular that there is no solution  $u_3$  to $(P_{\overline{\lambda}})$ with 
 $$
 u_1\lneqq u_3\ll u_2,
 $$
 as otherwise, define $\tilde\lambda=\sup\{\lambda \in [0,1]\mid  \lambda u_1+(1-\lambda) u_2-u_3\gg0\}$ and observe that $u_3$ is a solution to $(P_{\overline{\lambda}})$ with $u_3\leq \beta_{\tilde \lambda}$ but not $u_3\ll \beta_{\tilde \lambda}$, which contradicts the fact that $\beta_{\tilde \lambda}$ is a strict upper solution to $(P_{\overline{\lambda}})$.
\medbreak 
Now, let us define $v_i= \frac{1}{\mu} \left( e^{\mu u_i} -1 \right)$ for $i=1$, $2$, the corresponding solution to
$(Q_{\overline{\lambda},\,v_0})$. By Lemma \ref{equivProblems-ch3}, we deduce from the above argument that the problem $(Q_{\overline{\lambda},\,v_0})$ has no solution $v_3$ with
\begin{equation}
 \label{eq *}
 v_1\lneqq v_3\ll v_2.
 \end{equation}
\medbreak

As $v_0$ and $\beta_{1/2}$ are strict lower and upper solutions to  $(Q_{\overline{\lambda},\,v_0})$ and $v_1$ is the unique solution of $(Q_{\overline{\lambda},\,v_0})$  with $v_0\lneqq v_1\lneqq \beta_{1/2}$,  we deduce from Corollary \ref{localMinimizer-ch3} and Proposition \ref{c01vsWop-ch3} that $v_1$ is a local minimum of $J_{\overline{\lambda}}$. Let us prove the existence of $r>0$ such that 
   \begin{equation}
 \label{eq **}
\inf_{v\in \partial B(v_1,r)} J_{\overline{\lambda}}(v) > J_{\overline{\lambda}}(v_1).
 \end{equation}
Otherwise, by Theorem \ref{characterizacionMinimizer-ch3}, for all $n\in \mathbb N$, there exists $w_n$ with $\|w_n-v_1\|=\frac{1}{n}$, $J_{\overline{\lambda}}(w_n)=J_{\overline{\lambda}}(v_1)$ and $w_n$ is a solution to $(Q_{\overline{\lambda},\,v_0})$. Moreover,  as $v_1$ is the minimum solution to $(Q_{\overline{\lambda},\,v_0})$, we have also $w_n\gneqq v_1$. By a bootstrap argument, we prove that $\|w_n-v_1\|_{C^1}\to 0$ and hence, for $n$ large enough, 
$$
v_1\lneqq w_n\ll v_2
$$
which contradicts \eqref{eq *}.
\medbreak

By \eqref{eq **}, denoting $\delta= \inf_{v\in \partial B(v_1,r)} J_{\overline{\lambda}}(v) - J_{\overline{\lambda}}(v_1)$, we have $D>0$ such that, for $\lambda > \overline{\lambda}$,
\[
 \inf_{v\in \partial B(v_1,r)} J_{\lambda}(v) 
 \geq  \inf_{v\in \partial B(v_1,r)} J_{\overline{\lambda}}(v) 
- D (\lambda-\overline{\lambda})= 
J_{\overline{\lambda}}(v_1) + \delta - D (\lambda-\overline{\lambda})
\geq
J_{\lambda}(v_1) + \delta - 2D (\lambda-\overline{\lambda}).
\]   
This implies  the existence of $\epsilon>0$ such that, for all $\lambda\in [\overline{\lambda}, \overline{\lambda}+\epsilon]$, 
$$
\inf_{v\in \partial B(v_1,r)} J_{\lambda}(v) > J_{\lambda}(v_1).
$$
%Arguing as in the proof of Proposition \ref{localMinimumVariational-ch3}, we see that 
Hence, arguing as in the proof of Proposition \ref{localMinimumVariational-ch3}, we deduce that, for some $\lambda > \overline{\lambda}$, the problem $(Q_{\lambda,\,v_0})$
has a solution which is a local minimum of $J_{\lambda}(v)$. This contradicts the definition of 
$\overline{\lambda}$.
\end{proof}

As a consequence of this result we prove the first part of Corollary \ref{cor-ch3}.

\begin{proof}[\textbf{Proof of the first part of Corollary \ref{cor-ch3}}]
Let $u_0$ be the unique solution to \eqref{P0-ch3}. Since \eqref{A2-ch3} holds, we know that $u_0 \in \mathcal{C}_0^1(\overline{\Omega})$. Now, observe that
\[ 
\left\{
\begin{aligned}
-\Delta u_0 - \mu \langle \nabla u_0, \nabla u_0 \rangle + c_{-}(x) u_0 & \geq h(x), \quad & \textup{ in } \Omega, \\
u_0 & = 0, \quad & \textup{ on } \partial \Omega.
\end{aligned}
\right.
\]
Hence, since $h \gneqq 0$, by Theorem \ref{SMP-ch3}, it follows that that $u_0 \gg 0\,,$ and so, in particular that $c_{+}u_0 \gneqq 0.$ The corollary then follows immediately from Theorem \ref{th2-ch3}. 
\end{proof}

We end this section  with the proof of the first part of Theorem \ref{th-h-0-ch3}.

\begin{lemma}  \label{strict-upper-cu0 0 positive-ch3}
Assume \eqref{A2-ch3}, suppose that \eqref{P0-ch3} has a solution $u_0$ with $c_{+} u_0 \equiv 0$ and let $\gamma_1 > 0$ be the principal eigenvalue of \eqref{linearized-eigenvalue-problem-ch3}. Then, for all $0 < \lambda < \gamma_1$, there exists $\beta \in \mathcal{C}_0^1(\overline{\Omega})$ strict upper solution to \eqref{Plambda-ch3} satisfying $\beta \gg u_0$.
\end{lemma}

\begin{proof}
First of all, by Proposition \ref{prop-maximum-anti-maximum-ch3} applied to $L_{u_0}$ (defined in \eqref{linearized-operator-ch3}) with $m=c_+$ and $h \equiv 1$, we know that, for $0<\lambda<\gamma_1$ there exists $w \gg 0$ solution to
\[ L_{u_0}(w) = \lambda c_{+}(x) w + 1, \quad w \in H_0^1(\Omega).\]
We consider then $\beta := u_0 + \epsilon w$ with $\epsilon > 0$ and we are going to prove that, for $\epsilon > 0$ small enough, $\beta$ is an upper solution to \eqref{Plambda-ch3}. Directly observe that, for $\epsilon > 0$ small enough,
\[ 
\begin{aligned}
-\Delta \beta & = c_{\lambda}(x) \beta + \mu |\nabla u_0|^2 + 2\mu \langle \nabla u_0, \nabla (\epsilon w) \rangle + \epsilon + h(x) 
\\
& \gneqq c_{\lambda}(x) \beta + \mu |\nabla u_0|^2 + 2\mu \langle \nabla u_0, \nabla (\epsilon w) \rangle + \mu \epsilon^2 |\nabla w |^2  + h(x) \\
& = c_{\lambda}(x) \beta + \mu |\nabla \beta|^2 + h(x), \quad \textup{ in } \Omega.
\end{aligned}
\]
Hence, since $\beta = 0$ on $\partial \Omega$ and $w \gg 0$, we have that $\beta$ is an upper solution to \eqref{Plambda-ch3} with $\beta\gg u_0$.
\medbreak

To prove that $\beta$ is strict, let $u$ be a solution to \eqref{Plambda-ch3} with $u\leq \beta$. Define $\varphi=\beta-u$ and observe that
\[ \left\{
\begin{aligned}
-\Delta \varphi - \mu \langle \nabla u + \nabla \beta, \nabla \varphi \rangle + c_{-}(x) \varphi &= \lambda c_+(x) \varphi\gneqq 0, \quad & \textup{ in } \Omega, \\
w & = 0\,,& \textup{ on } \partial \Omega\,.
\end{aligned}
\right.
\]
Applying Theorem \ref{SMP-ch3}, we obtain that $\varphi \gg 0$, and so that $\beta$ is strict.
\end{proof}
%
%\begin{prop} \label{propLocalMinimum0-ch3}
%Assume \eqref{A3-ch3}, suppose that \eqref{P0-ch3} has a solution $u_0$ with $c_{+} u_0 \equiv 0$ and let $\gamma_1 > 0$ be the principal eigenvalue of \eqref{linearized-eigenvalue-problem-ch3}. Then, for every $\lambda \in ]0, \gamma_1[$, there exists $v \in \mathcal{C}_0^1(\overline{\Omega})$ with $v \geq v_0$ which is a local minimum of $J_{\lambda}$ in the $H_0^1$-topology and a solution to $(Q_{\lambda,\,v_0})$.
%\end{prop}
%
%\begin{proof}
%The result follows from Lemma \ref{strict-upper-cu0 0 positive-ch3} arguing as in the proof of Proposition \ref{lemmaStrictLowerAndUpperPositive-ch3}.
%\end{proof}

\begin{proof}[\textbf{Proof of the first part of Theorem \ref{th-h-0-ch3}}] Let us split the proof into two steps:
\medbreak
\noindent \textbf{Step 1:} \textit{For every $\lambda \in\,  ]0, \gamma_1[$, there exists $v \in \mathcal{C}_0^1(\overline{\Omega})$ with $v \geq v_0$ which is a local minimum of $J_{\lambda}$ in the $H_0^1$-topology and a solution to $(Q_{\lambda,\,v_0})$.}
\medbreak

The result follows from Lemma \ref{strict-upper-cu0 0 positive-ch3} arguing as in Step 2 of the proof of Theorem \ref{th2-ch3} choosing as strict lower solution $v_0-1$ as now $v_0$ is a solution to $(Q_{\lambda,\,v_0})$.
% and hence is still  a lower solution but not more a strict lower solution  to $(Q_{\lambda,\,v_0})$.
\medbreak
\noindent \textbf{Step 2:} \textit{Conclusion.}
\medbreak

As in Step 3 of the proof of Theorem \ref{th2-ch3}, we have the existence of two solutions 
to \eqref{Plambda-ch3}.
%By Step 1, for $\lambda < \gamma_1$, there exists a first critical point 
%$v_{\lambda,1} \in \mathcal{C}_0^{1}(\Omega)$, which is a local minimum of $J_{\lambda}$ and satisfies 
%$v_{\lambda,1} \geq v_0$. Since the Palais-Smale condition at any level $d \in \R$ holds, 
%by Theorem \ref{characterizacionMinimizer-ch3}, we have two options. If we are in the first case, 
%then together with Lemma \ref{mountainPassGeometry2-ch3}, we see that $J_{\lambda}$ has the mountain-pass 
%geometry and by Theorem \ref{mpTheorem-ch3}, we have the existence of a second solution to  
%$(Q_{\lambda,\,v_0})$. In the second case, we have directly the existence of a second solution to  
%$(Q_{\lambda,v_0})$. Then, by Lemma \ref{equivProblems-ch3}, we conclude the existence of two solutions 
%to \eqref{Plambda-ch3}. By the construction of $v_{\lambda,1}$, it follows that 
%$u_{\lambda,1} = \frac{1}{\mu} \ln(1+\mu v_{\lambda,1}) \geq u_0$. 
%Let us stress that we cannot prove that $u_{\lambda,1}  \not \equiv u_0$ and $u_0$ is always a solution 
%to \eqref{Plambda-ch3}. Hence, we will consider that $u_{\lambda,1} \equiv u_0$. 
As $u_0$ is a solution to  \eqref{Plambda-ch3}, now the minimal solution above $u_0$ is $u_0$. Hence the two solutions satisfy  $u_0 \equiv u_{\lambda,1} \lneqq u_{\lambda,2}$. Then, arguing as in the proof of Theorem \ref{th2-ch3}, Step 3, we deduce that $u_0 \equiv u_{\lambda,1} \ll u_{\lambda,2}$ and the result follows.
\end{proof}

\section{Proof of Theorem \ref{th3-ch3} and second part of Theorem  \ref{th-h-0-ch3}} \label{V-ch3}

The aim of this Section is to prove Theorem \ref{th3-ch3} and the rest Theorem \ref{th-h-0-ch3} and of Corollary \ref{cor-ch3}. As in Section \ref{section 6}, we choose $\underline{u}_{\lambda} \in H^1(\Omega) \cap \Linfty$, 
the lower solution to \eqref{Plambda-ch3} constructed in Proposition \ref{propLowerSol-ch3},  as lower solution in \eqref{truncFunction-ch3} and 
%the definition of \eqref{Qlambda-ch3} and 
we use the notation $I_{\lambda}:= I_{\lambda,\, \alpha_{\lambda}}$ for the functional $I_{\lambda,\, \alpha_{\lambda}}$ defined in \eqref{Ilambda-ch3}. We begin proving the uniqueness of solutions with $c_{+}u \leq 0$ under suitable assumptions.

\medbreak

\begin{prop} \label{uniquenessNonPositiveSolution-ch3}
Assume \eqref{A2-ch3} and suppose that \eqref{P0-ch3} has a solution $u_0$ with $c_{+}u_0 \lneqq 0$. Then,  for every $\lambda \in \R$, the problem \eqref{Plambda-ch3} has at most one solution $u$ with $c_{+}u \leq 0$.
\end{prop}

\begin{proof} First of all, observe that for $\lambda \leq 0$ the result follows from \cite[Theorem 1.1]{A_DC_J_T_2014}. Hence, we just consider the case $\lambda > 0$. Let us split the proof into three steps:
\medbreak
\noindent \textbf{Step 1:} \textit{If $u$ is a lower solution to \eqref{Plambda-ch3} with $c_{+}u \leq 0$ then $u \ll u_0$.}
\medbreak
Let $u$ be solution to \eqref{Plambda-ch3} with $c_{+}u \leq 0$. We easily observe that $u$ is a lower solution to \eqref{P0-ch3}. Hence, by Lemma \ref{compPrincipleP0-ch3}, it follows that $u \leq u_0$. Now, let us introduce $w = u_0 - u$ and observe that
\[ \left\{
\begin{aligned}
-\Delta w - \mu \langle \nabla u_0 + \nabla u, \nabla w \rangle + c_{-}(x) w & \geq -\lambda c_{+}(x) u_0\,, \quad & \textup{ in } \Omega\,, \\
w & = 0\,,& \textup{ on } \partial \Omega\,.
\end{aligned}
\right.
\]
Hence, by Theorem \ref{SMP-ch3}, we deduce that $w \gg 0$ and so, that $u \ll u_0$, as desired.
\medbreak
\noindent \textbf{Step 2:} \textit{If we have $u_1$ and $u_2$ two solutions to \eqref{Plambda-ch3} with $c_{+}u_1 \leq 0$ and $c_{+}u_2 \leq 0$ then we have two ordered solutions $\tilde{u}_1 \lneqq \tilde{u}_2 \leq u_0\,.$}
\medbreak
By Step 1, we have  $u_1\ll u_0$ and $u_2\ll u_0$. In case $u_1$ and $u_2$ are not ordered, using that $u_0$ is an upper solution to \eqref{Plambda-ch3} and the maximum of two lower solutions is a lower solution, by Theorem \ref{BMP1988-ch3}, we deduce the existence of $u_3$ solution to \eqref{Plambda-ch3} such that $\max\{u_1,u_2\} \leq u_3 \leq u_0$. We conclude taking $\tilde{u}_1 = u_1$ and $\tilde{u}_2 = u_3$.
\bigbreak
\noindent \textbf{Step 3:} \textit{Conclusion.}
\smallbreak
We assume by contradiction the existence of $u_1$ and $u_2$ solutions to \eqref{Plambda-ch3} such that $c_{+}u_1 \leq 0$ and $c_{+}u_2 \leq 0$. By Steps 1 and 2, we can suppose without loss of generality  that $u_1 \lneqq u_2\ll u_0$. As $u_0- u_2\gg 0$, observe that the set $\{v\in C^1_0(\overline\Omega)\mid v\leq  u_0- u_2\}$ is an open neighbourhood of $0$ and so, that the set  $\{\varepsilon>0\mid u_2-u_1\leq \varepsilon (u_0- u_2)\}$ is not empty. We define 
$$
\bar\varepsilon: =\inf\{\varepsilon>0\mid u_2-u_1\leq \varepsilon  (u_0- u_2)\}.
$$
Observe that  $0<\bar\varepsilon <\infty$ and 
\begin{equation}
\label{eq min}
\bar\varepsilon=\min\{\varepsilon>0\mid u_2-u_1\leq \varepsilon  (u_0- u_2)\}.
\end{equation}
Now, we define
 $$
w_{\bar\varepsilon}=\frac{(1+\bar\varepsilon)u_2-u_1}{\bar\varepsilon},
$$
and, arguing as in \cite[Proposition 4.3]{DC_J_2017}, we deduce that $w_{\overline{\epsilon}}$ is a lower solution to \eqref{Plambda-ch3}. By the choice of ${\bar\varepsilon}>0$, $w_{\bar\varepsilon}\leq u_0$ and so $c_+(x)w_{\bar\varepsilon}\leq c_+(x) u_0\lneqq 0$. Hence, by Step 1, it follows that $w_{\bar\varepsilon}\ll u_0$, which  contradicts the definition of $\overline{\varepsilon}$ given in \eqref{eq min}.
\end{proof}

\begin{proof}[\textbf{Proof of Theorem \ref{th3-ch3}}] As in the proof of Theorem \ref{th2-ch3}, we define $v_0:= \frac{1}{\mu}\left( e^{\mu u_0} -1 \right) \in \mathcal{C}_0^1(\overline{\Omega})$ and we use the compact notation $I_{\lambda}:= I_{\lambda,\, \alpha_{\lambda}}$. Then, we split the proof into several steps:

\medbreak
\noindent \textbf{Step 1:}\textit{ For every $\lambda > 0$, there exists $v \in \mathcal{C}_0^1(\overline{\Omega})$ with $v \ll v_0$ which is a local minimum of $I_{\lambda}$ and a solution to $(Q_{\lambda,\, \alpha_{\lambda}})$.}
\medbreak
By Proposition \ref{propLowerSol-ch3} we have the existence of a strict lower solution $\alpha_{\lambda}$ with $\alpha_{\lambda} \ll \beta$ for every $\beta$ upper solution to \eqref{Plambda-ch3}. On the other hand, arguing as in Lemma \ref{lemmaStrictLowerAndUpperPositive-ch3}, we prove that $u_0$ is a strict upper solution to \eqref{Plambda-ch3} for all $\lambda > 0$. Then, reasoning as in the proof of Theorem \ref{th2-ch3}, we deduce the existence of $v \in \mathcal{C}_0^1(\overline{\Omega})$ local minimum of $I_{\lambda}$ and solution to $(Q_{\lambda,\, \alpha_{\lambda}})$. By its construction $v \ll v_0$.% = \frac{1}{\mu} \big(e^{\mu u_0} - 1 \big)$.
\medbreak
\noindent \textbf{Step 2:} \textit{For every $\lambda > 0$ the exists at least two solutions $u_{\lambda,1}$, $u_{\lambda,2} \in \mathcal{C}_0^1(\overline{\Omega})$ to \eqref{Plambda-ch3} such that}
\[ u_{\lambda,1} \ll u_{\lambda,2}, \quad u_{\lambda,1} \ll u_0 \quad \textup{ and } \quad  
c_+u_{\lambda,2}\not\leq 0. 
%\max_{\overline{\Omega}} u_{\lambda,2} > 0.
\] 
\smallbreak
By Step 1 and Lemma \ref{equivProblems-ch3}, there exists a first solution $u_{\lambda,1} \in \mathcal{C}_0^1(\overline{\Omega})$ such that $u_{\lambda,1} \ll u_0$. Then, arguing as in Step 3 of the proof of Theorem \ref{th2-ch3}, we deduce the existence of a second solution $u_{\lambda,2} \in \mathcal{C}_0^1(\overline{\Omega})$ with $u_{\lambda,2} \gg u_{\lambda,1}$. Finally, by Proposition \ref{uniquenessNonPositiveSolution-ch3}, we deduce that 
$c_+u_{\lambda,2}\not\leq 0$.
%$\max_{\overline{\Omega}} u_{\lambda,2} > 0$.
\medbreak
\noindent \textbf{Step 3:} \textit{$\lambda_1 < \lambda_2$ implies $u_{\lambda_1,1} \gg u_{\lambda_2,1}$}.
\smallbreak
The proof follows as in Step 4 of the proof of Theorem \ref{th2-ch3}.
\end{proof}

%\begin{prop} \label{propLocalMinimumNegative-ch3}
%Assume \eqref{A2-ch3} and suppose that \eqref{P0-ch3} has a solution $u_0$ with $c_{+}u_0 \lneqq 0$. Then, for every $\lambda > 0$, there exists $v \in \mathcal{C}_0^1 (\overline{\Omega})$ with $v \ll v_0 :=  \frac{1}{\mu} (e^{\mu u_0} - 1)$ which is a local minimum of $I_{\lambda}$ in the $H_0^1$-topology and a solution to $(Q_{\lambda,\, \alpha_{\lambda,1}})$.
%\end{prop}

%\begin{proof}

%\end{proof}

%\begin{proof}[\textbf{Proof of Theorem \ref{th3-ch3}}]
%The result follows from Propositions \ref{propLocalMinimumNegative-ch3} and \ref{uniquenessNonPositiveSolution-ch3}  arguing as in the proof of Theorem \ref{th2-ch3}.
%\end{proof}

\begin{proof}[\textbf{Proof of the second part of Corollary \ref{cor-ch3}}]
Since $h \lneqq 0$ and \eqref{A2-ch3} holds, the problem \eqref{P0-ch3} has always a solution $u_0 \in \mathcal{C}_0^1(\overline{\Omega})$. Moreover, observe that $0$ is an upper solution to \eqref{P0-ch3}. Hence, by Lemma \ref{compPrincipleP0-ch3}, it follows that $u_0 \leq 0$. Now, as $u_0$ satisfies
\[ -\Delta u_0 - \mu \langle \nabla u_0, \nabla u_0 \rangle + c_{-}(x) u_0 = h(x) \lneqq 0\,, \quad \textup{ in } \Omega\,,\]
we deduce by Theorem \ref{SMP-ch3} that $u_0 \ll 0$ and, in particular that $c_{+}u_0 \lneqq 0$. Thus, the corollary follows immediately from Theorem \ref{th3-ch3}.
\end{proof}

\begin{lemma}  \label{strict-upper-cu0 0 negative-ch3}
Assume \eqref{A2-ch3}, suppose that \eqref{P0-ch3} has a solution $u_0$ with $c_{+} u_0 \equiv 0$ and let $\gamma_1 > 0$ be the principal eigenvalue of \eqref{linearized-eigenvalue-problem-ch3}. Then, for all $\lambda > \gamma_1$, there exists $\beta \in \mathcal{C}_0^1(\overline{\Omega})$ strict upper solution to \eqref{Plambda-ch3} satisfying $\beta \ll u_0$.
\end{lemma}

\begin{proof}
Let $\delta > 0$ be given by Proposition \ref{prop-maximum-anti-maximum-ch3} applied to $L_{u_0}$ (defined in \eqref{linearized-operator-ch3}) with  $m=c_+$ and $h \equiv 1$ and let $\lambda_0 \in\, ]\gamma_1, \min\{ \lambda, \gamma_1 + \delta\}[$. By Proposition \ref{prop-maximum-anti-maximum-ch3} we know that there exists $w \ll 0$ solution to 
\[ L_{u_0}(w) = \lambda_0 c_{+}(x) w + 1, \quad w \in H_0^1(\Omega).\]
We consider $\beta:= u_0 + \epsilon w$ with $\epsilon > 0$ and, arguing as in Lemma \ref{strict-upper-cu0 0 positive-ch3}, we deduce that, for $\epsilon > 0$ small enough, $\beta$ is a strict upper solution to \eqref{Plambda-ch3} satisfying $\beta \ll u_0$.
\end{proof}
%\begin{prop} \label{prop2-u00-c3}
%Assume \eqref{A3-ch3}, suppose that \eqref{P0-ch3} has a solution $u_0$ with $c_{+} u_0 \equiv 0$ and let $\gamma_1 > 0$ be the principal eigenvalue of \eqref{linearized-operator-ch3}. Then, for every $\lambda > \gamma_1$, there exists $v \in \mathcal{C}_0^1(\overline{\Omega})$ with $v \ll v_0 := \frac{1}{\mu} \big( e^{\mu u_0} - 1 \big)$ which is a local minimum of $I_{\lambda}$ in the $H_0^1$-topology and a solution to $(Q_{\lambda,\, \alpha_{\lambda,1}})$.
%\end{prop}

%\begin{proof}
%The proof follows from Lemma \ref{strict-upper-cu0 0 negative-ch3} %arguing as in Proposition \eqref{propLocalMinimumNegative-ch3}.
%\end{proof}

\begin{proof}[\textbf{Proof of the rest of Theorem \ref{th-h-0-ch3}}]  Let us define $v_0:=  \frac{1}{\mu} \big( e^{\mu u_0} - 1 \big)$ and split the proof into three steps:%
\medbreak
\noindent \textbf{Step 1:} \textit{For every $\lambda > \gamma_1$, there exists $v \in \mathcal{C}_0^1(\overline{\Omega})$ with $v \ll v_0$ which is a local minimum of $I_{\lambda}$ in the $H_0^1$-topology and a solution to $(Q_{\lambda,\, \alpha_{\lambda}})$}
\medbreak
The proof follows 
%from Lemma \ref{strict-upper-cu0 0 negative-ch3}
 arguing as in Step 1 of the proof of Theorem \ref{th3-ch3} replacing $u_0$ by the strict upper solution constructed in Lemma \ref{strict-upper-cu0 0 negative-ch3}.
\medbreak
\noindent \textbf{Step 2:} \textit{For $\lambda > \gamma_1$, the problem \eqref{Plambda-ch3} has at least two solutions $u_{\lambda,1} \equiv u_0 \gg u_{\lambda,2} \in \mathcal{C}_0^1(\overline{\Omega})$.}
\medbreak
Since $c_{+} u_0 \equiv 0$, it is clear that $u_{\lambda,1}\equiv u_0$ is a solution for every $\lambda \in \R$. On the other hand, by Step 1, we deduce the existence of a second solution $u_{\lambda,2} \in \mathcal{C}_0^1(\overline{\Omega})$ with $u_{\lambda,2} \ll u_{\lambda,1} \equiv u_0$.

\medbreak
\noindent \textbf{Step 3:} \textit{For $\lambda = \gamma_1$, the unique solution to \eqref{Plambda-ch3} is $u_{\gamma_1} \equiv u_0$.}
%For $\lambda > \gamma_1$ the result follows from Proposition \ref{prop2-u00-c3} arguing as in the proof of the first part of Theorem \ref{th-h-0-ch3}. Let us then just consider the case $\lambda =  \gamma_1$. We are going to prove that for $\lambda = \gamma_1$ the unique solution to \eqref{Plambda-ch3} is $u_{\gamma_1} \equiv u_0$. 
\medbreak
We assume by contradiction that there exists $u$ solution to $(P_{\gamma_1})$ with $u \not \equiv u_0$. Let us then define $w := u - u_0$ and observe $w$ solves 
\[ -\Delta w = c_{\gamma_1}(x) w + \mu |\nabla u|^2 - \mu |\nabla u_0|^2, \quad w \in H_0^1(\Omega).\]
Equivalently, we have that $w \not \equiv 0$ is a solution to
\[ L_{u_0}(w) = \gamma_1 c_{+}(x) w + \mu |\nabla w|^2, \quad w \in H_0^1(\Omega).\]
Applying Proposition \ref{prop-maximum-anti-maximum-ch3} with $h = \mu |\nabla w|^2 \gneqq 0$ we obtain a contradiction and the result follows.
\end{proof}

%%%%%%%%%%%%%%%%%%%%%%%%%%%%%%%%%%%%%%%%%%%%%%%%%%%%%%%%%%%%%%%%%

\begin{proof}[\textbf{Proof of the rest of Corollary \ref{cor-ch3}}]
For $h \equiv 0$, the unique solution to \eqref{P0-ch3} is $u_0 \equiv 0$. Hence, the result follows from Theorem \ref{th-h-0-ch3}.
\end{proof}

\appendix
\section{Hopf's Lemma ans SMP with unbounded lower order terms} \label{Appendix-ch3}
 
In this section we prove Theorem \ref{SMP-ch3} which can be seen as a combination of the Strong maximum principle and the Hopf's Lemma.  As said in Section \ref{II-ch3}, our proof is inspired by \cite{rosales_2019}. Let us begin with some preliminary results that will be needed to prove Theorem \ref{Hopf-ch3}. Throughout the appendix we assume $N \geq 2$. 
 
\begin{lemma} \label{SMP-lemma1-ch3} \rm \cite[Lemma 4.2]{K_S_2019} \it Let $\Omega \subset \RN$ be a bounded domain, $\beta \in (L^N(\Omega))^N$ and $\xi \in L^{N/2}(\Omega)$ with $\xi \geq 0$. Then, for every $F \in H^{-1}(\Omega)$, there exists a unique  solution $u \in  H_0^1(\Omega)$ to 
\begin{equation*}
\left\{
\begin{aligned}
-\Delta u + \langle \beta(x), \gradu \rangle + \xi(x)u & = F, \quad & \textup{ in } \Omega,\\
u & = 0, \quad & \textup{ on } \partial \Omega.
\end{aligned}
\right.
\end{equation*}
\end{lemma}
 
As a consequence of the previous lemma we obtain an existence result with inhomogeneous boundary conditions. % Under the assumption
%\begin{equation} \label{hypSMP-App-ch3}
%\left\{
%\begin{aligned}
%& \omega \subset \RN,\ N \geq 2, \textup{ is a bounded domain with } %\partial \omega \textup{ of class } \mathcal{C}^{1,1},\\
%& \xi  \textup{ belongs to } L^p(\omega) \textup{ and } \beta = (\beta^1, \ldots, \beta^N) \in (L^{p}(\omega))^N \textup{ for some } p > N, \\
%& \xi \geq 0,
%\end{aligned}
%\right.
%\end{equation}
\vspace{-0.1cm}
\begin{cor} \label{SMP-corollary2-ch3}
Let $\omega := B_1(0) \setminus \overline{B_{1/2}(0)}$, $\beta \in (L^p(\omega))^N$ and $\xi \in L^p(\omega)$ for some $p > N$ and assume that $\xi \geq 0$. Then, there exists a unique  solution $u \in \mathcal{C}^{1,\tau}(\overline{\omega})$ for some $\tau > 0$ to
\begin{equation} \label{SMP-corollary2-equation-ch3}
\left\{
\begin{aligned}
-\Delta u + \langle \beta(x), \gradu \rangle + \xi(x)u & = 0, \quad & \textup{ in } \omega,\\
u & = 0, \quad & \textup{ on } \partial B_1(0),\\
u & = 1, \quad & \textup{ on } \partial B_{1/2}(0),
\end{aligned}
\right.
\end{equation}
such that $0 \leq u \leq 1$ in $\overline{\omega}$.
\end{cor}

\begin{proof}
Let us consider $\varphi \in C^{\infty}(\RN)$ given by 
$ \varphi(x) = \frac{4}{3}(1-|x|^2),$ and observe that $\varphi(x) = 0$ for all $ x \in \partial B_1(0)$ and  $\varphi(x) = 1$ for all $x \in \partial B_{1/2}(0)$. Moreover, by direct computations it follows that
\[ -\Delta \varphi + \langle \beta(x), \nabla \varphi \rangle + \xi(x) \varphi = \frac{8N}{3} - \frac{8}{3} \langle \beta(x), x \rangle + \frac{4}{3} \xi(x) (1-|x|^2) =: - F \in H^{-1}(\omega).\]
By Lemma \ref{SMP-lemma1-ch3} we know that there exists a unique solution   $w \in H_0^1(\omega)$   to
\begin{equation*}
\left\{
\begin{aligned}
-\Delta w + \langle \beta(x), \nabla w \rangle + \xi(x)w & = F, \quad & \textup{ in } \omega,\\
w & = 0, \quad & \textup{ on } \partial \omega.
\end{aligned}
\right.
\end{equation*}
Then, we define $u = w + \varphi$ and we observe that $u \in H^1(\omega)$ is a solution to \eqref{SMP-corollary2-equation-ch3}. Next, by \cite[Proposition 3.10]{K_S_2019} we deduce that $0 \leq u \leq 1$ in $\overline{\omega}$ and, by \cite[Theorem II-15.1]{L_U_1968} we obtain that $u \in \mathcal{C}^{1,\tau}(\overline{\omega})$ for some $\tau > 0$. Finally, the uniqueness follow again from \cite[Proposition 3.10]{K_S_2019}.
\end{proof}

\begin{lemma}  \label{SMP-lemma3-ch3}
Let $\omega := B_1(0) \setminus \overline{B_{1/2}(0)}$, $\epsilon \in (0,1/4)$, $x_0 \in \partial B_1(0)$ and $T: \RN \to \RN$ given by $T(x) = \epsilon^{-1}( x-x_0) + x_0$. Then, it follows that $\omega \subset T(\omega):= \{T(x): x \in \omega\}$.
\end{lemma}

\begin{proof}
First of all, observe that
\begin{equation} \label{TO-1-ch3}
T(\omega) = B_{1/\epsilon} \left(\big(1-\frac{1}{\epsilon}\big)x_0 \right) \setminus \overline{B_{1/2\epsilon} \left(\big(1-\frac{1}{\epsilon}\big)x_0 \right)} = \left\{ x \in \RN: \frac{1}{2\epsilon} < \big| x - \big( 1 - \frac{1}{\epsilon}\big)x_0 \big| < \frac{1}{\epsilon} \right\}. 
\end{equation}
Now, observe that, for all $x \in \omega$ and all $\epsilon \in (0,1/4)$, it follows that
\begin{equation} \label{TO-2-ch3}
\big| x - \big( 1 - \frac{1}{\epsilon}\big)x_0 \big| \leq |x| + \big| 1 - \frac{1}{\epsilon}\big|\,|x_0| = |x| + \frac{1}{\epsilon}-1 < 1 + \frac{1}{\epsilon} - 1 = \frac{1}{\epsilon},
\end{equation}
and 
\begin{equation} \label{TO-3-ch3}
\big| x - \big( 1 - \frac{1}{\epsilon}\big)x_0 \big| \geq \left| |x| - \big| 1 - \frac{1}{\epsilon}\big|\,|x_0| \right| \geq  \frac{1}{\epsilon} - 1 - |x| > \frac{1}{\epsilon} - 2 > \frac{1}{2\epsilon}.
\end{equation}
Hence, the result follows from \eqref{TO-1-ch3}-\eqref{TO-3-ch3}.
\end{proof} 
 
\begin{lemma} \label{SMP-lemma4-ch3}
Let $\omega := B_1(0) \setminus \overline{B_{1/2}(0)}$, $B = (B^1, \ldots, B^N) \in (L^p(\omega))^N$ and $a \in L^p(\omega)$, for some $p > N$, $\epsilon \in [0,1/4]$, $B_{\epsilon}(y) = (B_\epsilon^1, \ldots, B_{\epsilon}^N) := \epsilon B(\epsilon(y-x_0) + x_0)$ and $a_{\epsilon}(y) := \epsilon^2 a(\epsilon(y-x_0) + x_0).$
Then, it follows that
\begin{itemize}
\item[$\bullet$] $\|B_{\epsilon}^i \|_{L^p(\omega)} \leq \epsilon^{1 - \frac{N}{p}} \|B^i\|_{L^p(\omega)},\qquad \forall\  i = 1, \ldots, N$;
\item[$\bullet$] $\|a_{\epsilon}\|_{L^p(\omega)} \leq \epsilon^{2 - \frac{N}{p}} \|a\|_{L^p(\omega)}$.
\end{itemize}
\end{lemma} 
 
\begin{proof}
Let $i \in \{1, \ldots, N\}$. We directly observe that
\begin{equation} \label{ABC-1-ch3}
\|B_{\epsilon}^{i}\|_{L^p(\omega)}^p = \int_{\omega} |B_{\epsilon}^i(y)|^p dy = \epsilon^p \int_{\omega} |B^{i}(\epsilon(y-x_0)+x_0)|^p dy = \epsilon^{p-N} \int_{S(\omega)} |B^i(z)|^p dz,
\end{equation}
where $z = S(y) = \epsilon(y-x_0) + x_0$. Then, arguing as in Lemma \ref{SMP-lemma3-ch3}, we obtain that $S(\omega) \subset \omega$, and so, taking into account \eqref{ABC-1-ch3}, we deduce that
\[ \|B_{\epsilon}^{i}\|_{L^p(\omega)}^p \leq \epsilon^{p-N} \int_{\omega} |B^{i}(z)|^p dz = \epsilon^{p-N} \|B^{i}\|_{L^p(\omega)}^p .\]
The estimate for $a_{\epsilon}$ follows arguing on the same way.
\end{proof} 

Using the rescaled functions $B_{\epsilon}$ and $a_{\epsilon}$  defined in Lemma \ref{SMP-lemma4-ch3}, we introduce the auxiliary boundary value problem
\begin{equation} \label{Pepsilon-ch3} \tag{$P_{\epsilon}$}
\left\{
\begin{aligned}
-\Delta u + \langle B_{\epsilon}(x), \gradu \rangle + a_{\epsilon}(x)u & = 0, \quad & \textup{ in } B_1(0) \setminus \overline{B_{1/2}(0)},  \\
u & = 0, \quad & \textup{ on } \partial B_1(0),\\
u & = 1, \quad & \textup{ on } \partial B_{1/2}(0),
\end{aligned}
\right.
\end{equation}
and we prove the following uniform a priori bound that will be crucial in the proof of Theorem  \ref{Hopf-ch3}.
 
\begin{lemma} \label{SMP-lemma5-ch3}
Let $\omega:= B_1(0) \setminus \overline{B_{1/2}(0)}$, $B = (B_1, \ldots, B_N) \in (L^p(\omega))^N$ and $a \in L^p(\omega)$ for some $p > N$. Then, there exists $M > 0$ such that, for all $\epsilon \in [0,1/4]$, 
any solution $u$ to \eqref{Pepsilon-ch3} satisfies $\|u\|_{\mathcal{C}^1(\overline{\omega})} \leq M$.
\end{lemma} 

\begin{proof}
We argue by contradiction. Assume the existence of sequences $\{\epsilon_n\} \subset [0,1/4]$ and $\{u_n\}$ solutions to \eqref{Pepsilon-ch3} with $\epsilon = \epsilon_n$ such that
\[ \|u_n\|_{\mathcal{C}^1(\overline{\omega})} \to + \infty, \quad \textup{ as }  n \to \infty.\]
Without loss of generality (up to a subsequence if necessary) we may assume that
\[ 1 \leq \|u_n\|_{\mathcal{C}^1(\overline{\omega})}, \quad \forall\ n \in \N.\]
We consider then $v_n:= \frac{u_n}{\|u_n\|_{\mathcal{C}^1(\overline{\omega})}}$ and  observe that $v_n$ solves
\begin{equation*} 
\left\{
\begin{aligned}
-\Delta v_n + \langle B_{\epsilon_n}(x), \nabla v_n \rangle + a_{\epsilon_n}(x)v_n & = 0, \quad & \textup{ in }  \omega,  \\
v_n & = 0, \quad & \textup{ on } \partial B_1(0),\\
v_n & = \frac{1}{\|v_n\|_{ \mathcal{C}^1(\overline{\omega}) }}, \quad & \textup{ on } \partial B_{1/2}(0).
\end{aligned}
\right.
\end{equation*}
Now, for all $n \in \N$, let us define 
\[ \xi_n = \frac{4}{3\|u_n\|_{ \mathcal{C}^1(\overline{\omega}) }} (1-|x|^2) \in \mathcal{C}^{\infty}(\RN) \quad \textup{ and } \quad  w_n = v_n - \xi_n,\]
and observe that $w_n$ solves
\begin{equation*}
\left\{
\begin{aligned}
-\Delta w_n & = - \langle B_{\epsilon_n}(x), \nabla v_n \rangle - a_{\epsilon_n} (x) v_n - \frac{8N}{3\|u_n\|_{\mathcal{C}^1(\overline{\omega})}}, \quad & \textup{ in } \omega, \\
w_n & = 0, & \textup{ on } \partial \omega.
\end{aligned}
\right.
\end{equation*}
Then, by \cite[Lemma 9.17]{G_T_2001_S_Ed}, there exists $C (\omega, N) > 0$ such that
\[ \|w_n\|_{W^{2,p}(\omega)} \leq C\, \left\| \langle B_{\epsilon_n}(x), \nabla v_n \rangle + a_{\epsilon_n} (x) v_n + \frac{8N}{3\|u_n\|_{\mathcal{C}^1(\overline{\omega})}} \right\|_{L^p(\omega)},\]
and so, since $\|v_n\|_{\mathcal{C}^1(\overline{\omega})} = 1$ for all $n  \in \N$, by Lemma \ref{SMP-lemma4-ch3}, there exists $C_1 = C_1 (\omega, N, \|B\|_{(L^p(\omega))^N},$ $ \|a\|_{L^p(\omega)}) > 0$ such that
\[ \|w_n\|_{W^{2,p}(\omega)} \leq C_1.\]
From the definition of $w_n$, we deduce the existence of $C_2 > 0$ (independent of $n$) such that 
\[ \|v_n\|_{W^{2,p}(\omega)} \leq C_2.\]
Since $p > N$, by the Sobolev compact embedding, we have that, up to a subsequence $v_n \to v$ in $\mathcal{C}^1(\overline{\omega})$ for some $v \in\mathcal{C}^1(\overline{\omega})$. Moreover, by  Lemma \ref{SMP-lemma4-ch3}, we have $\overline{a} \in L^p(\omega)$ (resp. $\overline{B} \in (L^p(\omega))^N$) with $\overline{a}\geq 0$ such that $a_{\epsilon_n} \rightharpoonup \overline{a}$ weakly in $L^p(\omega)$ (resp. $B_{\epsilon_n} \rightharpoonup \overline{B}$ weakly in $(L^p(\omega))^N$).
 This implies that $v$ is a weak solution to 
\begin{equation*}
\left\{
\begin{aligned}
-\Delta v + \langle \overline{B}(x), \nabla v \rangle + \overline{a}(x) v & = 0, \quad & \textup{ in } \omega,\\
v & = 0, & \textup{ on } \partial \omega.
\end{aligned}
\right.
\end{equation*}
By \cite[Proposition 3.10]{K_S_2019}, we deduce that $v \equiv 0$. This contradicts the fact that $v_n \to v$ in $\mathcal{C}^1(\overline{\omega})$ and the result follows.
\end{proof}

Having at hand all the needed ingredients, we prove the Hopf's Lemma with unbounded lower order terms.

\begin{theorem} \label{Hopf-ch3} \rm (\textbf{Hopf's Lemma}) \it
For $z \in \RN$ and $R > 0$, let $B \in (L^p(B_R(z))^N$ and $a \in L^p(B_R(z))$ for some $p > N$ such that $a \geq 0$. Let $x_0 \in \partial B_R(z)$ and let $u \in \mathcal{C}^{1}(\overline{B_R(z)})$ be an upper solution to
\begin{equation*} %\label{Hopf-equation-ch3}
\left\{
\begin{aligned}
-\Delta u + \langle B(x), \gradu \rangle + a(x)u & = 0, \quad & \textup{ in } B_R(z),\\
u & = 0, \quad & \textup{ on } \partial B_R(z),
\end{aligned}
\right.
\end{equation*}
such that $u(x) > u(x_0) = 0$ for all $x \in B_R(z)$. Then $\frac{\partial u}{\partial \nu}(x_0) < 0$, where $\nu$ denotes the exterior unit normal.
\end{theorem}

\begin{proof} 
As it is well known, by the change of variable $y = T(x) = \frac{1}{R}(x-z)$, there is no loss of generality to consider the problem on $B_1(0)$ i.e. to assume that  $x_0 \in \partial B_1(0)$ and that  $u \in \mathcal{C}^{1}(\overline{B_1(0)})$ is an upper solution to
\begin{equation} \label{Hopf-equation-ch3}
\left\{
\begin{aligned}
-\Delta u + \langle B(x), \gradu \rangle + a(x)u & = 0, \quad & \textup{ in } B_1(0),\\
u & = 0, \quad & \textup{ on } \partial B_1(0).
\end{aligned}
\right.
\end{equation}

Let us fix $\omega:= B_1(0) \setminus \overline{B_{1/2}(0)}$ and split the proof into several steps:

\medbreak
\noindent \textbf{Step 1:} \textit{Auxiliary regular barrier $\varphi$}
\medbreak
Let us consider the problem
\begin{equation} \label{hl1-1-ch3}
\left\{
\begin{aligned}
-\Delta \varphi & = 0, \quad & \textup{ in } \omega,\\
\varphi & = 0, \quad & \textup{ on } \partial B_1(0),\\
\varphi & = 1, \quad & \textup{ on } \partial B_{1/2}(0).
\end{aligned}
\right.
\end{equation}
By \cite[Theorem 6.14]{G_T_2001_S_Ed} we know that there exists $\varphi \in \mathcal{C}^{2,\tau}(\overline{\omega})$ for some $\tau > 0$ solution to \eqref{hl1-1-ch3}. Moreover, by \cite[Lemma 3.4 and Theorem 3.5]{G_T_2001_S_Ed}, we know that
\begin{equation}
\label{eq A13}
0 < \varphi(x) < 1, \quad \forall\ x \in \omega, \quad \textup{ and } \quad \frac{\partial \varphi}{\partial \nu} (x_0) < 0.
\end{equation}
\medbreak
\noindent \textbf{Step 2:} \textit{Let $M > 0$ given by Lemma \ref{SMP-lemma5-ch3}. For every $\epsilon \in (0,1/4)$ there exists $\varphi_{\epsilon} \in \mathcal{C}^{1,\tau}(\overline{\omega})$ for some $\tau > 0$ solution to \eqref{Pepsilon-ch3} such that $\|\varphi_{\epsilon}\|_{\mathcal{C}^1(\overline{\omega})} \leq M.$}
\medbreak
The existence follows from Corollary \ref{SMP-corollary2-ch3} and the uniform bound from Lemma \ref{SMP-lemma5-ch3}.

\medbreak
\noindent \textbf{Step 3:} \textit{Let $\varphi_{\epsilon}$ the solution to \eqref{Pepsilon-ch3} given by Step 2. There exists $\overline{\epsilon} \in (0,1/4)$ such that, for all $\epsilon \in (0, \overline{\epsilon})$, it follows that $\frac{\partial \varphi_{\epsilon}}{\partial \nu} (x_0) < 0$.}
\medbreak
Let us define $\psi_{\epsilon} := \varphi_{\epsilon} - \varphi$ and observe that $\psi_{\epsilon} \in \mathcal{C}^{1,\tau}(\overline{\omega})$ for some $\tau > 0$  solves
\begin{equation*}
\left\{
\begin{aligned}
-\Delta \psi_{\epsilon} & = - \langle B_{\epsilon}(x), \nabla \varphi_{\epsilon} \rangle - a_{\epsilon}(x) \varphi_{\epsilon}, \quad & \textup{ in } \omega,\\
\psi_{\epsilon} & = 0, & \textup{ on } \partial\omega.
\end{aligned}
\right.
\end{equation*}
Then, by \cite[Lemma 9.17]{G_T_2001_S_Ed} and Lemma \ref{SMP-lemma4-ch3}, there exists $C = C(\omega, N)> 0$ such that
\begin{equation*}
\begin{aligned}
\|\psi_{\epsilon}\|_{W^{2,p}(\omega)} &\leq C \left\|  \langle B_{\epsilon}(x), \nabla \varphi_{\epsilon} \rangle + a_{\epsilon}(x) \varphi_{\epsilon} \right\|_{L^p(\omega)} \\
& \leq C \|\varphi_{\epsilon}\|_{\mathcal{C}^{1}(\overline{\omega})} \epsilon^{1-\frac{N}{p}} \left( \sum_{i=1}^N \|B^{i}\|_{L^p(\omega)} + \epsilon \|a\|_{L^p(\omega)} \right) \\
& \leq \epsilon^{1-\frac{N}{p}} C M \left( \sum_{i=1}^N \|B^{i}\|_{L^p(\omega)} +  \|a\|_{L^p(\omega)} \right) =: \epsilon^{1-\frac{N}{p}} C_2
\end{aligned}
\end{equation*}
for some $C_2$ independent of $\epsilon$. Hence, by the Sobolev  embedding, there exists $C_3 > 0$ independent of $\epsilon$ such that
\[ \|\psi_{\epsilon}\|_{\mathcal{C}^{1,\tau}(\overline{\omega})} \leq  \epsilon^{1-\frac{N}{p}} C_3.\]
We conclude that
\[ \lim_{\epsilon \to 0} \left| \frac{\partial \varphi_{\epsilon}}{\partial \nu}(x_0) - \frac{\partial \varphi}{\partial \nu}(x_0) \right| \leq \lim_{\epsilon \to 0} \|\psi_{\epsilon}\|_{\mathcal{C}^{1,\tau}(\overline{\omega})} = 0,\]
and the Step 3 follows by \eqref{eq A13}.

\medbreak
\noindent \textbf{Step 4:} \textit{Conclusion}
\medbreak

Let $u \in \mathcal{C}^1(\overline{B_1(0))}$ be an upper solution to \eqref{Hopf-equation-ch3} such that $u(x) > u(x_0) = 0$ for all $x \in B_1(0)$. We fix $\epsilon > 0$ small enough to ensure that the Step 3 holds and define 
\[u_{\epsilon}(y) = u( \epsilon(y-x_0) + x_0).\]
Since we know that $\omega \subset T(\omega)$ by Lemma \ref{SMP-lemma3-ch3}, we have that $u_{\epsilon}$ is an upper solution to
\begin{equation*}
\left\{
\begin{aligned}
-\Delta u_{\epsilon} + \langle B_{\epsilon}(x), \nabla u_{\epsilon} \rangle + a_{\epsilon}(x) u_{\epsilon} & = 0, \quad & \textup{ in } \omega, \\
u_{\epsilon} & = 0, & \textup{ on } \partial \omega.
\end{aligned}
\right.
\end{equation*}
Then, we define $\overline{u}_{\epsilon} = u_{\epsilon} - \theta_{\epsilon} \varphi_{\epsilon}$ with 
\[ \theta_{\epsilon} = \inf_{\partial B_{1/2}(0)} u_{\epsilon} > 0,\]
and we have that $\overline{u}_{\epsilon}$ is an upper solution to
\begin{equation*}
\left\{
\begin{aligned}
-\Delta \overline{u}_{\epsilon} + \langle B_{\epsilon}(x), \nabla \overline{u}_{\epsilon} \rangle + a_{\epsilon}(x) \overline{u}_{\epsilon} & = 0, \quad & \textup{ in } \omega, \\
\overline{u}_{\epsilon} & = 0, & \textup{ on } \partial \omega.
\end{aligned}
\right.
\end{equation*}
Applying then \cite[Proposition 3.10]{K_S_2019} we deduce that $u_{\epsilon} - \theta_{\epsilon} \varphi_{\epsilon} \geq 0,$ in $\overline{\omega},$ and so, by Step 3, we conclude that
\[ \frac{\partial u }{\partial \nu}(x_0) = \frac{1}{\epsilon} \frac{\partial u_{\epsilon}}{\partial \nu} (x_0) \leq \frac{\theta_{ \epsilon}}{\epsilon} \frac{\partial \varphi_{\epsilon}}{\partial \nu}(x_0) < 0,\]
as desired.
\end{proof}

%
%\begin{cor} \label{SMP-corollary7-ch3}
%For $z \in \RN$ and $R > 0$, let $\beta \in (L^p(B_R(z))^N$ and $\xi \in L^p(B_R(z))$ for some $p > N$ such that $\xi \geq 0$ a.e. in $B_R(z)$. Let $x_0 \in  \partial B_R(z)$ and let $u \in \mathcal{C}^1(\overline{B_R(z)})$ be an upper solution to
%\begin{equation} \label{SMP-corollary7-equation-ch3}
%\left\{
%\begin{aligned}
%-\Delta u + \langle \beta(x), \gradu \rangle + \xi(x)u & = 0, \quad & \textup{ in } B_R(z),\\
%u & = 0, \quad & \textup{ on } \partial B_R(z),
%\end{aligned}
%\right.
%\end{equation}
%such that $u(x) > u(x_0) = 0$ for all $x \in B_R(z)$. Then $\frac{\partial u}{\partial \nu}(x_0) < 0$, where $\nu$ denotes the exterior unit normal.
%\end{cor}
%
%\begin{proof}
%Let us define $y = T(x) = \frac{1}{R}(x-z)$ and introduce the functions
%\begin{equation*}
%\begin{aligned}
%v(y) & = u(Ry + z),\\
%B(y) & = R \beta(Ry + z),\\
%c(y)& = R^2 \xi(Ry + z).
%\end{aligned}
%\end{equation*}
%Observe that if $u$ is an upper solution to \eqref{SMP-corollary7-equation-ch3} such that $u(x) > 0$ for all $x \in B_R(z)$, then $v$ is an upper solution to 
%\begin{equation*}
%\left\{
%\begin{aligned}
%-\Delta v + \langle B(y), \gradv \rangle + c(y)v & = 0, \quad & \textup{ in } B_1(0),\\
%v & = 0, \quad & \textup{ on } \partial B_1(0),
%\end{aligned}
%\right.
%\end{equation*}
%satisfying the hypothesis of Theorem \ref{Hopf-ch3} for some $y_0 = T(x_0) \in \partial B_1(0)$. Thus, we have that
%\[ \frac{\partial u}{\partial \nu} (x_0) = \frac{1}{R} \frac{\partial v}{\partial \nu} (y_0) < 0,\]
%and the result follows.
%\end{proof}

\begin{proof}[\textbf{Proof of Theorem \ref{SMP-ch3}}]
The result follows from Theorem \ref{Hopf-ch3}
%Corollary \ref{SMP-corollary7-ch3} 
arguing as in \cite[Theorem 3.27]{T_1987}. See also \cite[Theorem 3.5]{G_T_2001_S_Ed}.
\end{proof}

\bibliographystyle{plain}
\bibliography{Bibliography}

\end{document}